\documentclass{amsart}

\usepackage[main=english]{babel}
\usepackage[T1]{fontenc}

\usepackage{hyperref}

\usepackage[dvipsnames]{xcolor}
\hypersetup{
    colorlinks,
    linkcolor={blue!80!black},
    citecolor={green!50!black},
    urlcolor={blue!80!black}
  }

\usepackage{csquotes}
\usepackage{comment}
\usepackage{enumitem}
\setlist{  
  listparindent=\parindent,
  parsep=0pt,
}

\usepackage{graphicx}
\usepackage{caption}
\graphicspath{{../}{images/}{../images/}{../../images/}} 
\usepackage{wrapfig}

\usepackage{booktabs}
\usepackage{tabularx}
\usepackage{comment}

\usepackage{amsmath,amssymb,amstext,amsthm,bm,dsfont,mathtools}

\newtheorem{theorem}{Theorem}
\newtheorem{question}{Question}
\newtheorem{lemma}[theorem]{Lemma}
\newtheorem{corollary}[theorem]{Corollary}
\newtheorem{mprop}[theorem]{Proposition}
\newtheorem{mdef}[theorem]{Definition}
\newtheorem{mnot}[theorem]{Notation}
\newtheorem{mass}[theorem]{Assumption}

\theoremstyle{definition}
\newtheorem{remark}[theorem]{Remark}
\newtheorem{example}[theorem]{Example}

\newcommand{\LL}{\mathbb{L}^\infty}
\newcommand{\LLp}{\mathbb{L}^p}

\newcommand{\ind}{1\!\kern-1pt \mathrm{I}}

\newcommand{\newinf}{\mathop{\mathrm{inf}\vphantom{\mathrm{sup}}}}

\DeclareMathOperator*{\sconv}{\mathrm{co}_{\sigma}}
\DeclareMathOperator*{\co}{\mathrm{co}}

\newcommand{\pp}[0]{\ensuremath \mathcal{P}}
\newcommand{\pq}[0]{\ensuremath \mathcal{P}^{\mathcal{Q}}}
\newcommand{\q}[0]{\ensuremath \mathcal{Q}}
\newcommand{\rrr}[0]{\ensuremath \mathcal{R}}
\newcommand{\rr}[0]{\ensuremath \mathcal{R}}
\newcommand{\qq}[0]{\ensuremath \mathcal{Q}^{\mathcal{Q}}}
\newcommand{\f}[0]{\ensuremath \mathcal{F}}
\newcommand{\fq}[0]{\ensuremath \mathcal{F}^{\q}}
\newcommand{\he}[0]{\ensuremath \mathcal{H}^{\q}}
\newcommand{\hef}[0]{\ensuremath \mathcal{H}^{\q}_{\f}}

\newcommand{\compl}[1]{\ensuremath #1^{\mathsf{c}}}

\definecolor{mypurple}{rgb}{0.5,0,98}

\title[Can the $L^{1}$-$L^{\infty}$ duality be restored?]{Can the $L^1$-$L^\infty$ duality be restored for non-dominated families of probability measures?}

\author{Irene Klein, Georg Köstenberger}

\begin{document}
\begin{abstract}
The duality $L^{\infty}\simeq (L^{1})'$ frequently breaks down in the presence of model uncertainty, where a single reference measure $P$ is replaced by a non-dominated family of probability measures $\mathcal{P}$.
The unavailability of classical measure-theoretic and functional-analytic tools in this regime poses a significant obstacle to developing robust probabilistic frameworks.
We show that this duality can be restored for a broad class of robust statistical models by extending the underlying probability space. 
Specifically, on the extended model, the space $\mathbb{L}^{\infty}(\mathcal{P})$ of $\mathcal{P}$-quasi-surely bounded functions is isometrically isomorphic to the dual of the space of finite signed measures absolutely continuous with respect to at least one element of $\mathcal{P}$.
The proposed extension is canonical: it is the smallest $\mathcal{P}$-complete extension of the original $\sigma$-algebra for which $\mathbb{L}^{\infty}(\mathcal{P})$ is the dual of any normed space. 
Our assumptions encompass several prominent non-dominated settings, including infinite product measures, Gaussian processes, the Black-Scholes model with uncertain constant volatility and drift, robust binomial models, and, more generally, infinite sequences from any parametric model with almost surely estimable parameters.
Furthermore, we unify the existing frameworks of Cohen (2012) and Liebrich et al. (2022), demonstrating that our construction is equivalent to the capacity-based approach under mild assumptions satisfied by the aforementioned examples. 
Finally, we apply our theory to extend Kraft's (1955) characterization of strictly unbiased hypothesis tests to non-dominated cases.
\end{abstract}
\thanks{Georg Köstenberger is part of Research Unit 5381 of the German Research Foundation and is supported by the Austrian Science Fund (FWF) I 5485-N}
\maketitle

\frenchspacing
\raggedbottom

\section{Introduction}
Let $(\Omega, \mathcal{F}, P)$ be a probability space. 
Many results in probability theory depend on the fact that $L^{\infty}(P)$ is the dual space of $L^{1}(P)$. 
If we pass to a robust statistical model, i.e., if we replace the single probability measure $P$ with a family of probability measures $\pp$, this duality can fail. 
In this case, $L^{\infty}(P)$ is replaced by $\LL(\pp)$, the space of $\pp$-\emph{quasi-surely} bounded measurable functions, and it is neither obvious, nor always true, that $\LL(\pp)$ is the dual space of a normed space. 
Experience tells us that some desirable properties of probability spaces can be enforced post-hoc by passing to a canonical extension of the original space. 
Hence, we ask:
\begin{question}\label{q}
  Is there a canonical way to extend or modify a statistical model $(\Omega,\mathcal F,\pp)$ so that the associated robust $\LL(\pp)$ space admits a 
  meaningful duality theory?

\end{question}
The main contribution of this paper is to provide a positive answer to this question for a large class of stochastic models (see Section \ref{sec:examples} for a list of examples). 
For this purpose, we introduce Hahn-localizability, a sufficient condition that allows one to reconstruct a robust analogue of the classical $L^1$-$L^\infty$ duality.
More precisely, we show that Hahn-localizable families of probability measures admit an isometric identification
\begin{equation}\label{eq:goal1}
  \mathrm{ca}(\pp)' \simeq \LL(\pp),
\end{equation}
where $\mathrm{ca}(\pp)$ denotes the space of finite signed measures, which are absolutely continuous with respect to at least one element of $\pp$ (equipped with the total variation norm), and $\mathrm{ca}(\pp)'$ denotes the dual space of $\mathrm{ca}(\pp)$. 
In the case of a single measure, $L^{1}(P)$ is isometrically isomorphic to $\mathrm{ca}(P)$, and hence \eqref{eq:goal1} should be thought of as a robust version of the $L^{1}$-$L^{\infty}$ duality. 

While Hahn-localizability itself seems to be a quite strong condition to impose, it turns out that a much weaker condition, which we call pre-Hahn-localizability, is satisfied by a large class of models (see Section \ref{sec:examples}).
Roughly speaking, pre-Hahn-localizability ensures that the model can be decomposed into measurable regions on which the measures are ``locally dominated''.
The answer to Question \ref{q} is provided by the so-called Hahn-extension, a canonical way of extending any pre-Hahn-localizable model into a Hahn-localizable one, which then in turn satisfies \eqref{eq:goal1}. 
 We show that this extension is compatible with the usual $\mathcal{P}$-completion of a model, in the sense that the $\mathcal{P}$-completion of a Hahn-localizable model is again Hahn-localizable (see Section \ref{sec:hahn-loc}). 
  Together with the minimality results established in Section \ref{sec:hahn-ext:min}, this implies that the $\mathcal{P}$-completion of the Hahn-extension is the smallest $\mathcal{P}$-complete extension of the original model for which $\LL(\pp)$ is the dual space of a normed space.

\subsection{Related Work and Contribution}
The study of robust statistical models dates back at least to Knight's famous notion of model uncertainty \cite{knight1921risk}. 
Such frameworks arise in a wide range of areas including mathematical finance \cite{Den-Mar}, \cite{Peng-gexp},   \cite{Den-Ker}, \cite{Hu-Peng}, \cite{Tev-Tor},  \cite{BN:15},  \cite{Bia-Bou},  \cite{Lar-Acc}, \cite{Neu-Nut}, \cite{blanch_carassus}, \cite{carassus2022pricing}, 
statistics \cite{huber-strassen}, \cite{hampel-book}, \cite{rieder-book}, \cite{huber-book}, decision theory \cite{gilboa1989141}, \cite{maccheroni}, stochastic control \cite{petersen}, \cite{parys}, \cite{yang} and robust optimization \cite{bertsimas-book} (see also \cite{rahimian_2022} for an extensive review).
The problem of robustness has been addressed by a large number of different approaches. 
For example, there are approaches based on non-linear expectations \cite{peng-2004}, \cite{chen-chen-davison-2005}, \cite{peng_2007}, \cite{denis_hu_peng_2010}, \cite{guangyan-peng-2010}, \cite{peng_Gaussian_2011}, \cite{peng_song_zhang_2014}, \cite{peng_nonlinearExpectation_book}, approaches relying on techniques of stochastic control and stochastic backward differential equations such as \cite{neufeld2017nonlinear, Peng2019, fns_2019, geuchen2022affine, CriensNiemann},  and \cite{criens2023jumps}, pathwise approaches such as \cite{davis2007range} or \cite{acciaio2016model}.

In \cite{liebrich-model-uncertainty:reverse-approach} Liebrich et al. have identified another approach that has been used somewhat implicitly in the literature: supported models.
This approach includes \cite{chau2024robust}, \cite{chau_2022_super_replication}, \cite{soner_touzi_zhang_2011}, \cite{amarante-2027-contracting}, 
\cite{bartl_kupper_neufeld_2020_pathwise_superhedging}, \cite{burzoni_frittelli_hou_pointwise_arbitrage_2019}, \cite{hou2018RobustPD}, \cite{mykland_2003}, 
\cite{blanch_carassus}. 
Related ideas have also been explored in \cite{cohen}. 
For applications in distributionally robust optimization, we refer to the extensive survey of Rahimian et al. \cite{rahimian_2022}, and the references therein. 
Robustness is also a core topic in modern statistics and data science.
We refer to the books \cite{hampel-book}, \cite{rieder-book}, and \cite{huber-book}, and the references therein. 
Moreover, some authors give up $\sigma$-additivity and work with additive set functions instead, see \cite{dubins-savage}, \cite{larsson2026completecharacterizationtestablehypotheses}, or \cite{bingham2010} for a history of this debate. 
A full survey of the vast literature on robust statistical models is beyond the scope of this paper, and we refer to the listed survey articles and books for more literature.

Our work is most closely related to works of Liebrich et al. \cite{liebrich-model-uncertainty:reverse-approach}, and Cohen \cite{cohen}. 
In fact, the name ``Hahn-localizability'' is an homage to Cohen's Hahn property. 
We show that pre-Hahn-localizability, the class (S) property of \cite{liebrich-model-uncertainty:reverse-approach}, and the Hahn property of \cite{cohen} are equivalent up to (mild) additional assumptions. 
These additional assumptions are satisfied in all positive examples considered in Section \ref{sec:examples}. 
We also provide an example of a non-pre-Hahn-localizable family of probability measures. 
In this case, our theory can still be used to compute the dual space of the proposed robust $L^{1}$ space.

While a full characterization of the robust $L^{1}$-$L^{\infty}$ duality has been derived in the seminal work of Liebrich et al. \cite{liebrich-model-uncertainty:reverse-approach} for the supported case, it is still an open question, whether many prominent models possess this property, yet alone, how to enforce it post-hoc.
This is the main contribution of this work. 
We show that any pre-Hahn-localizable model $(\Omega, \mathcal{F}, \pp)$ admits a canonical extension, the Hahn-extension, which satisfies \eqref{eq:goal1}. 
Moreover, we show that this extension is compatible with the $\mathcal{P}$-completion of the model, i.e., the $\pp$-completion of a pre-Hahn-localizable model is again pre-Hahn-localizable, and the Hahn-extension of a complete model is again complete. 
As a by-product of our theory we arrive at a measure theoretic characterization of the Dedekind completeness of $\LL$, similar to the classical case (see \cite{fremlin-3}).
Moreover, we provide direct, constructive proofs for some of the results in \cite{liebrich-model-uncertainty:reverse-approach}. 

As an application of our theory, we extend the characterization of strictly unbiased hypothesis test due to Kraft \cite{kraft} to non-dominated settings. 
This problem was already studied by Le Cam \cite{lecam}, who gave a characterization in terms of so called ``generalized tests'', which can in general not be explicitly computed in practice.
More recently Larsson et al. \cite{larsson2026completecharacterizationtestablehypotheses} gave a characterization in terms of finitely additive set function, instead of $\sigma$-additive measures as was the case in the orignial work of Kraft \cite{kraft}. 
Using our theory, we can provide a characterization of strictly unbiased hypothesis for a large class of models, which neither requires generalized tests nor non-$\sigma$-additive set functions.

\subsection{Organisation}
In Section \ref{sec:preliminaries} we briefly introduce notation and review robust $L^{\infty}$ theory.
In Section \ref{sec:hahn-loc} we introduce Hahn-localizability, discuss its basic properties, and explore its relation to the Hahn property of Cohen \cite{cohen} and the class (S) property of Liebrich et al. \cite{liebrich-model-uncertainty:reverse-approach} . 
Section \ref{sec:duality} contains a new direct, constructive proof of the fact that every Hahn-localizable family of probability measures satisfies \eqref{eq:goal1}. 
The Hahn-extension is introduced in Section \ref{sec:hahn_extension}. 
We establish its existence, uniqueness and  minimality. 
In Section \ref{sec:examples} we collect a number of examples to which our theory applies, and one example to which it does not.
In Section \ref{sec:kraft_testable_hyp}, we apply our theory to extend a classical characterization of unbiased hypothesis tests due to Kraft \cite{kraft}, and compare it to the classical results of Le Cam \cite{lecam} and the recent results of Larsson et al. \cite{larsson2026completecharacterizationtestablehypotheses}.
 In Section~\ref{sec:kraft_testable_hyp}. In the appendix, Section~\ref{binomNA}, we comment on No Arbitrage and our assumptions in the robust binomial model  for readers with an interest in mathematical finance.

\section{Preliminaries and Notation}\label{sec:preliminaries}

Let $(\Omega, \f)$ be a measurable space, and let $\pp$ be a family of probability measures on $\f$. 
We write  $\mathcal{N}_{\pp}=\{N\in\mathcal{F}\mid P(N)=0, \forall P\in\pp\}$, for the $\pp$ null sets and $\mathcal{Z}_{\pp} = \{Z \subseteq \Omega\mid \exists N\in \mathcal{N}_{\pp}: Z \subseteq N\}$ for the set of $\pp$ polars.
If $\pp = \{P\}$, we write $\mathcal{N}_{\pp} = \mathcal{N}_{P}$ and $\mathcal{Z}_{\pp} = \mathcal{Z}_{P}$. 
We write $\uplus$ to emphasize that a union is disjoint. 
Furthermore, we say that a property holds $\pp$-quasi-surely (or $\pp$-q.s. for short) if it holds outside of a polar set. 
If $\q$ is a second family of probability measures on $\f$, we write $\pp\ll\q$, if $\mathcal{N}_{\q} \subseteq \mathcal{N}_{\pp}$, and we write $\pp\lll\q$, if 
\begin{equation*}
  \forall P\in \pp:\exists Q\in \q: P\ll Q.
\end{equation*}
Clearly, $\pp\lll\q$ implies $\pp\ll\q$, but in general, these two notions are different.
If $\pp = \{P\}$ is a singleton we write $P\ll \q$ instead of $\{P\} \ll \q$. 
For $P\in \pp$, we denote with $\f^{P}$ the $P$-closure of $\f$, i.e., the set
\begin{equation}\label{eq:P-completion}
  \f^{P} = \{F\cup Z\mid F\in \f, Z\in \mathcal{Z}_{P}\}.
\end{equation}
The $\pp$-closure of $\f$ is defined as
\begin{equation*}
  \f^{\pp} = \bigcap_{P\in \pp} \f^{P}. 
\end{equation*}
Note that $\f^{\pp}$ is in general larger than $\sigma(\f\cup \mathcal{Z}_{\pp})$ (see Example 2.8 in \cite{cohen}). 
One can extend any measure $P\in \pp$ on $\f$ to a measure $P'$ on $\f^{P}$ (and hence on $\f^{\pp}$). 
Since this correspondence is one-to-one, we will slightly abuse notation and write $P$ for both measures, if not specified otherwise. 
We write $\co(\pp)$ for the convex hull of $\pp$, and $\sconv(\pp)$ for the $\sigma$-convex hull of $\pp$, i.e.,
\begin{equation*}
  \sconv(\pp) = \biggl\{ \sum_{k=1}^{\infty} \lambda_{k}P_{k}\, \Big|\,  0\leq \lambda_{k}\leq 1,  \sum_{k=1}^{\infty} \lambda_{k} = 1, P_{k}\in \pp \biggr\}. 
\end{equation*}

We denote with $\mathcal{L}^{0} = \mathcal{L}^{0}(\Omega, \f)$ the set of $\mathbb{R}$-valued $\f$-measurable functions on $\Omega$, and define $\mathcal{N} = \{f\in \mathcal{L}^{0}\mid f = 0, \pp\text{-q.s.}\}$.
Following \cite{denis_hu_peng_2010}, we set
\begin{equation*}
  \begin{split}
    \mathcal{L}^{\infty} &= \{ f\in \mathcal{L}^{0}\mid \exists M>0: |f|\leq M, \pp\text{-q.s.}\}, \\
    \LL & = \mathcal{L}^{\infty}/\mathcal{N}. 
  \end{split}
\end{equation*}
The space $\LL$ can be equipped with the norm
\begin{equation*}
  \|f\|_{\LL} = \inf\{ M>0\mid |f|\leq M, \pp\text{-q.s.}\}. 
\end{equation*}
If we want to emphasize the dependence of $\LL$ on $\f$ or $\pp$, we write $\LL(\f,\pp)$ or $\LL(\pp)$. 
Note that $\LL(\pp) = \LL(\co(\pp)) = \LL(\sconv(\pp))$, since $\pp$, $\co(\pp)$, and $\sconv(\pp)$ have the same polar sets. 
If $\pp = \{P\}$ is a singleton, then $\LL(\pp) = L^{\infty}(P)$.
Hence, the space $\LL(\pp)$ is a robust analogue of $L^{\infty}(P)$. 
In the case $\pp = \{P\}$, the pre-dual of $L^{\infty}(P)$ is given by $L^{1}(P)$. 
This allows us to define the weak* topology on $L^{\infty}(P)$.
It is natural to ask: Is there a weak* topology on $\LL$, or rather, what is the pre-dual of $\LL$? 

We denote with $\mathrm{ca} = \mathrm{ca}(\mathcal{F})$ the set of all finite signed measures on $\f$, and equip it with the total variation norm $\|\cdot\|_{TV}$. 
For $\mu \in \mathrm{ca}$ we denote with $|\mu|$ its total variation measure, i.e.,
\begin{equation*}
  |\mu|(F) = \sup \{\mu(E) - \mu(F\setminus E)\mid E\in \f, E \subseteq F\}.
\end{equation*}
Note that $\|\mu\|_{TV} = |\mu|(\Omega)$. 
For $\mu\in \mathrm{ca}$ and $P\in \pp$, we write $\mu\ll P$ if $|\mu|\ll P$. 
We write $\mathrm{ca}(P) = \{\mu \in \mathrm{ca}\mid \mu\ll P\}$ for the set of signed measures, which are absolutely continuous with respect to one $P\in \pp$.
Furthermore, we define 
\begin{equation*}
  \mathrm{ca}(\pp) = \bigcup_{P\in \pp} \mathrm{ca}(P) = \{\mu\in \mathrm{ca}\mid \exists P\in \pp \text{ with } \mu\ll P\},
\end{equation*}
the set of finite signed measures which are absolutely continuous with respect to one $P\in \pp$.
If $\pp$ is convex, then $\mathrm{ca}(\pp)$ is a topological vector space, and if $\pp$ is $\sigma$-convex, then $\mathrm{ca}(\pp)$ is a Banach space (see \cite{cuchiero2025quantitativehalmossavagetheoremsrobust} for the short proof). 
Moreover, if $\pp = \{P\}$, then $\mathrm{ca}(P)$ is isometrically ismorphic to $L^{1}(P)$. 
Hence, our candidate for a robust version of $L^{1}$ is $\mathrm{ca}(\pp)$ with the total variation norm. 

If $A$ is an arbitrary set, we write $\delta_{a}$ for the Dirac measure on $a\in A$.
For any normed space $X$, we denote with $X'$ its dual space.
If $Y$ is a second normed space, we write $X\simeq Y$ if they are isometrically isomorphic to each other. 
If there is no ambiguity which norm we are using on a normed space, we will simply write $\|\cdot\|$ for the norm. 

\section{Hahn-localizability and its main properties}\label{sec:hahn-loc}

In this section we introduce Hahn-localizability, show that Hahn-localizable families of probability measures have a meaningful duality theory, and discuss its relationship to existing approaches of Cohen \cite{cohen} and Liebrich et al. \cite{liebrich-model-uncertainty:reverse-approach}.
For the sake of readability, the more involved proofs are postponed to the end of the section. 
Let us start with the definition of Hahn-localizability.
\begin{mdef}\label{def:hahn-loc}
Let $(\Omega, \f)$ be a measurable space, and $\pp$ a family of probability measures on $\f$. 
The family $\pp$ is called pre-Hahn-localizable, if there exists a family of probability measures $\q$ on $\f$ such that
\begin{enumerate}
  \item for every $Q\in \q$ there is an $S_{Q}\in \f$ such that $Q(S_{R}) = \delta_{QR}$ for all $Q,R\in \q$,
\item $\q\lll \pp \lll \sconv(\q)$.  
\end{enumerate}
A pre-Hahn-localizable family of probability measures $\pp$ is Hahn-localizable, if for every family $\{E_{Q}\}_{Q\in \q}$ of sets $E_{Q}\in \f$ with $E_{Q}\subseteq S_{Q}$ there is an $S\in \f$ such that 
      \begin{enumerate}
        \item[a\emph{)}] $Q(E_{Q}\setminus S) = 0$ for all $Q\in \q$.
        \item[b\emph{)}] If for any set $F\in \f$ we have $Q(E_{Q}\setminus F) = 0$ for all $Q\in \q$, then $Q(S\setminus F) = 0$ for all $Q\in \q$.
      \end{enumerate}

\end{mdef}
We will call $\q$ a \emph{localization} of $\pp$, and the sets $S_{Q}$ supports (or support sets) of $Q\in \q$.
We say that a localization $\q$ has \emph{strictly disjoint} supports (or is strictly disjointly supported) if the supports $S_{Q}$ are pairwise disjoint. Note that pre-Hahn-localizability only requires $\pp$-q.s. disjoint supports.

Hahn-localizability is motivated by Lemma \ref{lemma:char-of-sigma-convex-dominated-families} below.
Note that $\LL(\pp) = \LL(\sconv(\pp))$, and hence it is sufficient to work with the set $\sconv(\pp)$ instead of $\pp$.
In other words, it would be sufficient to look at $\sigma$-convex families of probability measures. 
In this case, we can fully characterize the case that $\pp$ is dominated by a $\sigma$-finite measure in the next lemma.

\begin{lemma}\label{lemma:char-of-sigma-convex-dominated-families}
Let $(\Omega, \f, \pp)$ be a statistical model. 
If $\pp$ is $\sigma$-convex, then the following two statements are equivalent:
\begin{enumerate}
\item $\pp$ is dominated by a $\sigma$-finite measure.
\item There is a countable set $\q$ of pairwise singular probability measures, such that 
  \begin{equation*}
    \q\lll \pp\lll \sconv(\q). 
  \end{equation*}
\end{enumerate}
\end{lemma}

Hahn-localizability now simply drops the assumption that $\q$ is countable, but requires the ``supports'' of the measures in $\q$ to be ``localizable'', which is known to be a necessary and sufficient condition for the $L^{1}$--$L^{\infty}$ duality in the classical case of a single measure (see Theorem 243G in \cite{fremlin-2}). 

The localization property (properties (a) and (b) of Definition \ref{def:hahn-loc} essentially states that one can localize (in the sense of a localizable measure, see Definition 211G in \cite{fremlin-2}) along the support sets. 
If the set $\q$ is countable, and $\{E_{Q}\}_{Q\in \q}$ is a family of sets from $\f$, with $E_{Q}\subseteq S_{Q}$, then we can pick $S = \bigcup_{Q\in \q}E_{Q}$.
If $\q$ is not countable, then this choice of $S$ might not be measurable. 
Hence, localizability can be interpreted as the assumption that for every family $\{E_{Q}\}_{Q\in \q}$ as above there is a set $S\in \f$ that can be thought of as a measurable version of $\bigcup_{Q\in \q}E_{Q}$. 
Proposition \ref{thm:char:dedekind-complete} below provides an alternative interpretation: if $\{E_{Q}\}_{Q\in \q}$ is as above, then Hahn-localizability guarantees the existence of a set $S\in \f$ with
\begin{enumerate}
\item $\ind_{S}\geq \ind_{E_{Q}}$, $\pp$-q.s. for all $Q\in \q$,
\item for every other $S'\in \f$ with $\ind_{S'}\geq \ind_{E_{Q}}$ $\pp$-q.s. for every $Q\in \q$, we have $\ind_{S}\leq \ind_{S'}$ $\pp$-q.s.
\end{enumerate}
In other words, $\ind_{S}$ is the $\pp$-q.s. essential supremum of the family $\{\ind_{E_{Q}}\}_{Q\in \q}$. 
Hence, we will sometimes refer to the set $S$ as the ``essential supremum'' of the family $\{E_{Q}\}_{Q\in \q}$. 
\begin{mprop}\label{thm:char:dedekind-complete}
Let $\pp$ be a $\sigma$-convex, pre-Hahn-localizable family of probability measures on $(\Omega, \f)$.
Then the following are equivalent:
\begin{enumerate}
  \item $\pp$ is Hahn-localizable.
  \item $\LL(\pp)$ is Dedekind complete.
  \item For every $\pp$-q.s. uniformly bounded family $\{g_{Q}\}_{Q\in \q}$ of non-negative, measurable functions with $g_{Q}|_{\compl{S_{Q}}} = 0$ $\pp$-q.s., there is a $g\in \mathcal{L}^{\infty}(\pp)$ such 
 $g\ind_{S_Q} = g_{Q}\ind_{S_Q}$ $\pp$-q.s.
\end{enumerate} 

\end{mprop}

Proposition \ref{thm:char:dedekind-complete} gives a characterization of the Dedekind-completeness of $\LL$ akin to Theorem 5.3 in \cite{liebrich-model-uncertainty:reverse-approach}.
We show that this is equivalent to a certain glueing property of non-negative random variables living on the support sets $S_{Q}$. 
This proposition is based on the close relationship between the class (S) property of \cite{liebrich-model-uncertainty:reverse-approach} and Hahn-localizability, which we explore in Section \ref{sec:preliminaries:hahn-loc-vs-(S)-property-and-hahn-property}. 
This connection allows us to apply results from \cite{liebrich-model-uncertainty:reverse-approach} to our case.

While being pre-Hahn-localizable is a fairly mild condition (similar to Cohen's Hahn property \cite{cohen} or the class (S) property of Liebrich et al. \cite{liebrich-model-uncertainty:reverse-approach}), the localization property seems to be a stronger assumption.
Luckily, we will show that any pre-Hahn-localizable family $\pp$ (with strictly disjointly supported localization) can be made Hahn-localizable in a canonical way via its \emph{Hahn-extension} (see Section \ref{sec:hahn_extension} for details).

The next theorem and its corollary establish that Hahn-localizable families of probability measures have a meaningful duality theory, in the sense that $\LL(\pp)$ is the dual space of $\mathrm{ca}(\pp)$.  

\begin{theorem}\label{thm:duality}
  Let $(\Omega,\f)$ be a measurable space, and $\pp$ a convex family of probability measures on $\f$.
  If $\pp$ is Hahn-localizable, then $\mathrm{ca}(\pp)'$ is isometrically isomorphic to $\mathbb{L}^{\infty}(\pp)$, via $l: \mathbb{L}^{\infty}(\pp) \to \mathrm{ca}(\pp)'$, 
  \begin{equation*}
    l: f \mapsto \biggl(\mu \mapsto \int_{\Omega}f d \mu\biggr).
  \end{equation*}
\end{theorem}

If $\pp$ is not convex, Theorem \ref{thm:duality} can still be applied.
Since $\LL(\pp) = \LL(\co(\pp))$ we get the following corollary, which does not require the set of probability measures $\pp$ to be convex. 

\begin{corollary}\label{cor:duality}
Let $(\Omega, \f, \pp)$ be Hahn-localizable, then
\begin{equation*}
  \mathrm{ca}(\co(\pp))' \simeq \LL(\pp). 
\end{equation*}
\end{corollary}

It should be pointed out that the space $\mathrm{ca}(\co(\pp))$ is in general different from the set of ``supported signed measures'' used as the pre-dual of $\LL$ in \cite{liebrich-model-uncertainty:reverse-approach} (see Example \ref{ex:ca-vs-sca} for an example where these two spaces differ).
This is not a contradiction, since the pre-dual of a Banach space need not be unique.

Since our goal is to apply Hahn-localizability as a post-hoc completion of our model $(\Omega, \f, \pp)$ it is natural to ask, whether Hahn-localizability is compatible with the $\pp$-completion of $\f$. 
This is indeed the case; the $\pp$-completion of a pre-Hahn-localizable family is again pre-Hahn-localizable, and hence we can apply the Hahn-extension to it.
\begin{lemma}\label{lemma:pre-hahn-loc-implies-pre-hahn-loc-of-closure}
Let $(\Omega, \f, \pp)$ be pre-Hahn-localizable, then 
\begin{enumerate}
\item $\f^{\pp} = \fq$, and
\item $(\Omega, \f^{\pp}, \pp)$ is pre-Hahn-localizable. 
\end{enumerate}
Moreover, if $(\Omega, \f, \pp)$ is Hahn-localizable with strictly disjoint supports, then $(\Omega, \f^{\pp},\pp)$ is Hahn-localizable. 
\end{lemma}

\subsection{Relationship of Hahn-localizability to the class (S) and Hahn property}\label{sec:preliminaries:hahn-loc-vs-(S)-property-and-hahn-property}
The notion of pre-Hahn-localizability is closely related to the idea of the class (S) property introduced in \cite{liebrich-model-uncertainty:reverse-approach}.
In fact, we will show below, that these two definitions are almost equivalent. 
We briefly recall the definition of a supported measure in the sense of \cite{liebrich-model-uncertainty:reverse-approach}. 
\begin{mdef}\label{def:supported-measures}
Let $(\Omega, \f)$ be a measurable space, and let $\pp$ be a family of probability measures on $\f$. 
  \begin{enumerate}
  \item A measure $\mu$ on $\f$ is supported (with support $S$), if there is an $S\in \f$, such that
\begin{enumerate}
  \item $\mu(\compl{S}) = 0$, 
  \item if $\mu(N\cap S) = 0$ for some $N\in \f$, then $N\cap S$ is a $\pp$-polar. 
\end{enumerate}
\item A finite signed measure $\mu$ on $\f$ is supported, if its total variation measure $|\mu|$ is supported. We write
  \begin{equation*}
  \mathrm{sca(\pp)} = \{\mu\in \mathrm{ca}\mid \mu \text{ is supported and } \mu \ll \pp\}.
\end{equation*}
\item The family $\pp$ is of class (S) if there is a family $\rr$ of supported probability measures on $\f$, such that $\mathcal{N}_{\pp} = \mathcal{N}_{\rr}$.
  \end{enumerate}
\end{mdef}

Following \cite{liebrich-model-uncertainty:reverse-approach}, we call such a $\mathcal{R}$ a \emph{supported alternative}. 
At first glance, the class (S) property and pre-Hahn-localizability are very similar.
However, there are two important differences. 
First, in contrast to $\q$, the supported alternative $\mathcal{R}$ need not be $\pp$-q.s. disjointly supported a priori.
In Lemma 3.7 of \cite{liebrich-model-uncertainty:reverse-approach}, it is shown that a supported alternative exists, if and only if a $\pp$-q.s. disjointly supported alternative exists, i.e., there is a supported alternative $\mathcal{R}$, such that $S_{Q}\cap S_{R}$ is a $\pp$-polar whenever $R\neq Q$, where $S_{Q}$ and $S_{R}$ are the supports of $Q,R\in \mathcal{R}$ respectively. 

Second, the supported alternative $\mathcal{R}$ from Definition \ref{def:supported-measures} satisfies $\mathcal{R}\ll \pp \ll \mathcal{R}$, while the localization $\q$ from Definition \ref{def:hahn-loc} satisfies $\q\lll \pp \lll \sconv(\q)$. 
A natural question to ask is how these two notions of robust absolute continuity are related. 

It turns out that for $\sigma$-convex families of probability measures of class (S), these two notions are \emph{almost} equivalent. 
This is based on the following definition.
\begin{mdef}\label{P-positiveSQ}
  Let $\pp$ either be a pre-Hahn-localizable family with localization $\rrr$, and supports $\{S_R\}_{R\in\rrr}$, or of class (S) with $\rrr$ being a supported alternative with $\pp$-q.s. disjoint supports $\{S_R\}_{R\in\rrr}$. 
For any $\mu\in \mathrm{ca}(\pp)$, we set
\begin{equation*}
  \mathcal{R}(\mu) = \{R\in \mathcal{R}\mid |\mu|(S_{R})>0\}.
\end{equation*}
\end{mdef}

\begin{remark}\label{rem:R(Q)-at-most-countable}
Let $\rrr$ be as in Definition \ref{P-positiveSQ}.
By the very definition of $\mathcal{R}(\mu)$, we have $\mathcal{R}(\mu) = \mathcal{R}(|\mu|)$. 
Note that for any $\mu\in \mathrm{ca}(\pp)$, the set $\mathcal{R}(\mu)$ is at most countable.
Indeed, $\mathcal{R}(\mu)$ can be written as
\begin{equation*}
  \mathcal{R}(\mu) = \bigcup_{m\geq 1} \mathcal{R}_{m}(\mu),
\end{equation*}
where $\mathcal{R}_{m}(\mu) = \{R\in \mathcal{R}\mid |\mu|(S_{R})> 1/m\}$, and $|\mathcal{R}_{m}(\mu)| \leq \lceil |\mu|(\Omega)\rceil m$, since $|\mu|$ is a finite measure, and the $S_{R}$'s are $\pp$-q.s. disjoint. 
 In particular,  if $\mathcal{P}$ is pre-Hahn-localizable with localization $\mathcal{R}$, and hence $\pp\lll  \sconv(\mathcal{R})$, this implies that 
\begin{equation}\label{eq:sum_mu}
  \mu(F) = \sum_{R\in \mathcal{R}(\mu)} \mu(F\cap S_{R})
\end{equation}
for any $\mu\in \mathrm{ca}(\pp)$ and $F\in \f$. Moreover, for pre-Hahn-localizable $\mathcal{P}$ with localization $\mathcal{R}$,  each $P\in \pp$ is determined by its values on the sets $\{S_{R} \mid R\in \mathcal{R}(P)\}$, i.e.,
\begin{equation}\label{eq:prob1supports}P\bigl(\bigcup_{R\in \mathcal{R}(P)}S_{R}\bigr)=1.\end{equation}
Let us prove the statements in (\ref{eq:sum_mu}) and (\ref{eq:prob1supports}). 
Indeed, let $\mu\in\mathrm{ca}(\pp)$. Then there exists $P\in\mathcal{P}$ such that $|\mu|\ll P$. 
By assumption there is  $Q = \sum_{k=1}^{\infty} \lambda_{k}R_{k} \in \sconv(\rrr)$ such that $P \ll Q$, therefore $|\mu|\ll Q$.
Observe that $\mathcal{R}(\mu)\subseteq\{R_k\mid k\geq1\}$.
Indeed, if $R\notin\{R_k\mid k\geq1\}$ then $Q(S_R)=0$ and hence $|\mu|(S_R)=0$. 
Therefore $R\notin\mathcal{R}(\mu)$. 
Set $S = \bigcup_{R\in \mathcal{R}(\mu)}S_{R}\subseteq  \bigcup_{k\geq 1}S_{R_k}$.
We have that $Q((\bigcup_{k\geq 1}S_{R_k})^c)=0$ and hence $|\mu|(\compl{(\bigcup_{k\geq 1}S_{R_k})})=0$. 
Define $N=\{k\geq 1\mid R_k\notin\mathcal{R}(\mu)\}$.
Then it follows that
$$|\mu|(\compl{S})\leq |\mu|\biggl(\compl{\biggl(\bigcup_{k\geq 1}S_{R_k}\biggr)}\biggr)+\sum_{k\in N}|\mu|(S_{R_k})=0.$$
Now, obviously, for every $F\in\mathcal{F}$, we have that
$$\mu(F)=\mu(F\cap S)+\mu (F\cap \compl{S})=\mu(F\cap S),$$
as $\mu(\compl{S})=0$. 
But as $\mathcal{R}(\mu)$ is a countable family of $\pp$-q.s. disjoint sets (and therefore $\mu$-a.e. disjoint) we have that
$$\mu(F)=\mu(F\cap S)=\sum_{R\in \mathcal{R}(\mu)} \mu(F\cap S_{R}),$$
hence (\ref{eq:sum_mu}) holds.
 Clearly $P\in \mathrm{ca}(\pp)$. Apply the above considerations to $P\in\mathcal{P}$ and $F=\Omega$. Hence, for $S=\bigcup_{R\in \mathcal{R}(P)}S_{R}$ we have that $P(\compl{S})=0$, and hence $P(S) =1$.

\end{remark}

Observe that for the weaker condition  $\pp\ll \mathcal{R}\ll \pp$ in Definition~\ref{P-positiveSQ} it is not obvious that the equations (\ref{eq:sum_mu}) and (\ref{eq:prob1supports}) hold. Therefore in order to give an alternative characterization of  families $\mathcal{P}$ satisfying the class (S) property we need the additional assumption that $P\bigl(\bigcup_{R\in \mathcal{R}(P)}S_{R}\bigr) = 1$, in other words, every $P$ is determined by its values on the sets $\{S_{R} \mid R\in \mathcal{R}(P)\}$.
With this additional assumption, we can give an alternative characterization of  families $\mathcal{P}$ satisfying the class (S) property. 

\begin{lemma}\label{lemma:equivalence_(S)_preHahnloc}
Let $(\Omega, \f)$ be a measurable space, and let $\pp$ be a $\sigma$-convex family of probability measures on $\f$. 
The following are equivalent:
\begin{enumerate}
  \item $\mathcal{P}$ has the class (S) property and for all $P\in \pp$ we have $P \bigl(\bigcup_{Q\in \mathcal{R}(P)}S_{Q} \bigr)=1$, where $\mathcal{R}$ is a $\pp$-q.s. disjointly supported alternative of $\mathcal{P}$. 
  \item $\pp$ is pre-Hahn-localizable. 

\end{enumerate}
\end{lemma}
Next, we would like to explore the relationship between the extensive duality results of Liebrich et al. \cite{liebrich-model-uncertainty:reverse-approach} and our Theorem \ref{thm:duality}. 
In \cite{liebrich-model-uncertainty:reverse-approach}, the pre-dual of $\LL$ is given by the supported signed measures $\mathrm{sca}(\pp)$, given by $\mathrm{sca}(\pp)$.
The sets $\mathrm{ca}(\pp)$ and $\mathrm{sca}(\pp)$ need not be identical, as is illustrated in the following example. 
\begin{example}\label{ex:ca-vs-sca}
  
Let $\Omega = [0,1]$, $\f= 2^{\Omega}$, $\q= \{\delta_{x}\mid x\in \Omega\}$, and $\pp = \co(\q)$.
Clearly, $\q$ is a supported alternative to $\pp$, and we have
\begin{equation*}
  \mathrm{ca}(\pp) = \biggl\{ \sum_{k=1}^{n} a_{k}\delta_{x_{k}}\, \big| \,  a_{k}\in \mathbb{R}, n \geq 0, x_{k}\in [0,1] \biggr\}.
\end{equation*}
On the other hand, we claim that 
\begin{equation}\label{eq:sca-example-claim}
  \mathrm{sca}(\pp) \supseteq \biggl\{ \sum_{k=1}^{\infty} a_{k} \delta_{x_{k}}\, \Big|\, a_{k}\in \mathbb{R}, \sum_{k=1}^{\infty}|a_{k}| < \infty, x_{k}\in [0,1] \biggr\} \supsetneq \mathrm{ca}(\pp).
\end{equation}
Indeed, let $\mu = \sum_{k=1}^{\infty}a_{k} \delta_{x_{k}}$ be of this form, then we define $S = \{x_{k}\mid a_{k}\neq 0\}$, and note that $|\mu|(\compl{S}) = \sum_{k=1}^{\infty} |a_{k}| \delta_{x_{k}}(\compl{S}) = 0$. 
Moreover, assume that for some $N\in \mathcal{F}$ we have $|\mu|(S\cap N) = 0$.
Since
\begin{equation*}
  |\mu|(S\cap N) = \sum_{k=1}^{\infty} |a_{k}| \delta_{x_{k}}(S\cap N) = 0,
\end{equation*}
this implies that $S\cap N = \varnothing$, and hence $S\cap N$ is a $\pp$-polar.
This proves the claim in \eqref{eq:sca-example-claim}.
In total, this implies that $\mathrm{ca}(\pp)$ is a proper subset of $\mathrm{sca}(\pp)$. 
\end{example}

Let us now discuss the relationship between Hahn-localizability and Cohen's Hahn property \cite{cohen}.  
\begin{mdef}[Hahn property, \cite{cohen}]\label{def:hahn-property}
The family $\pp$ is said to have the Hahn property, if there is a family of probability measures $\rr$ on $\f^{\pp}$ such that 
\begin{enumerate}
  \item $\pp$ and $\rr|_{\f}$ have the same polar sets, and $\mathcal{L}^{0}(\f^{\pp}) = \mathcal{L}^{0}(\f^{\rr})$,
  \item for every $R\in \rr$, there is a set $W_{R}\in \f^{\pp}$ with $R(W_{R}) = 1$, such that the sets $W_{R}$ are disjoint, and
  \item for all $F\in \f$ and $P\in \pp$, we have
    \begin{equation*}
P(F\cap W_{R}) = 0 \text{ for all } R\in \rr \text{ implies } P(F) = 0.
    \end{equation*}
\end{enumerate}
\end{mdef}
Following \cite{cohen}, we will call such a family $\mathcal{R}$  a dominating set. 
The set $\rr$ from the Hahn property should be viewed in contraposition to the localization $\q$ from Definition \ref{def:hahn-loc}. 
The Hahn property and Hahn localization are closely related.
The main difference between these two notions is that the Hahn property allows for the supports to be taken from $\f^{\pp}$, while the supports for pre-Hahn-localizable families are elements of $\f$ itself. 
On the other hand, the $W_{R}$ need to be disjoint, while the $S_{Q}$'s need not be disjoint. 

Their relation can be summarized as follows: Every pre-Hahn-localizable family of probability measures with a strictly disjointly supported localization has the Hahn property, with $\rr = \q$, and $S_{Q}$ (from Definition \ref{def:hahn-loc}) taking the role of $W_{R}$ from above (see Lemma \ref{lemma:pre-hahn-loc-implies-hahn-property}). And if a family of probability measures has the Hahn property and satisfies an additional (mild) condition, then its $\pp$-completion is pre-Hahn-localizable and has the class (S) property (see Lemma \ref{lemma:hahn-property+assumption-implies-pre-hahn-loc}). 

\begin{lemma}\label{lemma:pre-hahn-loc-implies-hahn-property}
  If $(\Omega, \f, \pp)$ is pre-Hahn-localizable with a strictly disjointly supported localization, then it has the Hahn property. 
\end{lemma}

\begin{proof}
Let $(\Omega, \f, \pp)$ be pre-Hahn-localizable, with localization $\q$. 
First, we note that $\pp$ and $\q$ have the same polar sets. 
Next, we need to check that $\mathcal{L}^{0}(\f^{\pp}) = \mathcal{L}^{0}(\fq)$. 
This is equivalent to $\f^{\pp} = \fq$, which we

already established in Lemma \ref{lemma:pre-hahn-loc-implies-pre-hahn-loc-of-closure}.

Since the supports $S_{Q}$ of the $Q$'s in $\q$ are assumed to disjoint, the second requirement of Definition \ref{def:hahn-property} is also satisfied. 

Finally, we need to show that for all $F\in \f$ and $P\in \pp$, we have
\begin{equation*}
  P(F\cap S_{Q}) = 0\text{ for all }Q\in \q\text{ implies }P(F) =0.
\end{equation*}
Following Remark \ref{rem:R(Q)-at-most-countable} this is an immediate consequence of
\begin{equation*}
  P(F) = \sum_{Q\in \q(P)} P(F\cap S_{Q}) = 0. 
\end{equation*}
\end{proof}

\begin{lemma}\label{lemma:hahn-property+assumption-implies-pre-hahn-loc}
  If $(\Omega, \f, \pp)$ has the Hahn property with a 'dominating' set of probability measures $\rr$, and pairwise strictly disjoint supports $S_{R}\in \f^{\pp}$ (in the sense of \cite{cohen}), such that for all $P\in \pp$ we have $P \bigl(\bigcup_{R\in \mathcal{R}(P)} S_{R} \bigr)=1$, then $(\Omega, \f^{\pp}, \pp)$ is pre-Hahn-localizable with localization $\mathcal{R}$ (and has the class (S) property). 
\end{lemma}
\begin{proof}
  By Lemma \ref{lemma:equivalence_(S)_preHahnloc} it is sufficient to show that $(\Omega, \f^{\pp}, \pp)$ has the class (S) property with disjointly supported alternative $\rr$. 
  We first note that $\rr$ and $\pp$ have the same polar sets by the Hahn property. 
  Hence, it is sufficient to show that every $R\in \rr$ is supported. 
  Clearly, $R(\compl{S_{R}}) = 0$.
  Now, let $N\in \f^{\pp}$, and assume that $R(N\cap S_{R}) = 0$. 
  We need to show that $N\cap S_{R}$ is a $\pp$-polar. 
  Since the $S_{R}$'s are disjoint by assumption, and $R(N\cap S_{R}) = 0$, we have
  \begin{equation*}
    Q(N\cap S_{R}) = 0 \qquad \text{for all } Q\in \rr.
  \end{equation*}
  This means that $N\cap S_{R}$ is  an $\rr$-polar. 
  Since $\rr$ and $\pp$ have the same polar sets, this implies that $N\cap S_{R}$ is a $\pp$-polar, and hence $R$ is supported.

\end{proof}

As a consequence of the last two lemmas, we get the following corollary. 
\begin{corollary}\label{cor:pre-hahn-loc-equivalent-to-all-mod-completion}
Let $(\Omega, \f, \pp)$ be a statistical model. Then the following are equivalent:
\begin{enumerate}
\item $(\Omega, \f^{\pp}, \pp)$ is pre-Hahn-localizable with a strictly disjointly supported localization $\q$. 
\item $(\Omega, \f^{\pp}, \pp)$ has the class (S) property with a strictly disjoint supported alternative $\rr$, such that for all $P\in \pp$ we have $P \bigl(\bigcup_{ R\in \mathcal{R}(P)}S_{R} \bigr)=1$.
\item $(\Omega, \f^{\pp},\pp)$ has the Hahn property with a 'dominating' set of probability measures $\rr$, such that for all $P\in \pp$ we have $P \bigl(\bigcup_{R\in \mathcal{R}(P)}S_{R} \bigr)=1$.
\end{enumerate}
\end{corollary}
\subsection{Proof of Theorem \ref{thm:duality}}\label{sec:duality}
The goal of this section is to prove Theorem \ref{thm:duality}.
It has already been shown in Lemma 3.4 of \cite{cuchiero2025quantitativehalmossavagetheoremsrobust} that the map $l: \LL(\pp)\to \mathrm{ca}(\pp)'$, given by
\begin{equation*}
  f\mapsto \biggl( \mu \mapsto \int_{\Omega}f d \mu \biggr)
\end{equation*}
is a linear isometry for arbitrary convex families $\pp$. 
Hence, all we have to show is that $l$ is surjective. 
The proof of Theorem \ref{thm:duality} requires a few lemmas. 
\begin{lemma}\label{lemma:TV-norm-of-singular-sum}
  Let $(\Omega,\f)$ be a measure space, $\mu$ be a finite signed measure on $\f$, and $\mu_{k}$ be a family of pairwise singular, finite, signed measures on $\f$, such that $\mu(F) = \sum_{k=1}^{\infty}\mu_{k}(F)$ for all $F\in \f$, then the series $\sum_{k=1}^{\infty} \mu_{k}$ converges in total variation norm, and
\begin{equation*}
  |\mu|(F) = \sum_{k=1}^{\infty}|\mu_{k}|(F)
\end{equation*}
for every $F\in \f$.
\end{lemma}
\begin{proof}
  Note that for $F\in \f$
  \begin{equation*}
    \begin{split}
    |\mu|(F) &= \sup \biggl\{\sum_{j=1}^{\infty} |\mu(F_{j})|\, \biggl|\, F_{j}\in \f: F = \biguplus_{j=1}^{\infty}F_{j}\biggr\}\\
              &\leq \sum_{k=1}^{\infty}\sup \biggl\{\sum_{j=1}^{\infty} |\mu_{k}(F_{j})| \,\biggl|\, F_{j}\in \f: F = \biguplus_{j=1}^{\infty}F_{j}\biggr\} = \sum_{k=1}^{\infty}|\mu_{k}|(F). 
    \end{split}
  \end{equation*}

  For the direction $|\mu|(F)\geq \sum_{k=1}^{\infty}|\mu_{k}|(F)$, we denote with $S_{k}$ the \emph{support} of $\mu_{k}$, i.e., a family of sets, such that $\mu_{k}(S_{j}) =0$ for $k\neq j$, and $\mu(S_{k}\cap F) = \mu_{k}(S_{k}\cap F) = \mu_{k}(F)$ for all $F\in \f$. 
  These sets exist by the very definition of pairwise singular measures. 
  Note that one can choose the $S_{k}$'s to be pairwise disjoint. 
  We write $S= \compl{\bigl(\bigcup_{k\geq 1}S_{k}\bigr)}$.
  For $\varepsilon>0$ and $k\geq 1$, let $\{F_{i}^{k}\}_{i=1}^{\infty}$ be a partition of $F$, such that
  \begin{equation*}
    |\mu_{k}|(F) \leq \sum_{i=1}^{\infty}|\mu_{k}(F_{i}^{k})| + \frac{\varepsilon}{2^{k}},
  \end{equation*}
  and note that 
  \begin{equation*}
    F = (F\cap S)\uplus \biguplus_{k=1}^{\infty}\biguplus_{i=1}^{\infty} (F_{i}^{k}\cap S_{k})
  \end{equation*}
  Since $\mu(F\cap S) = 0$, this implies 
  \begin{equation*}
    |\mu|(F) \geq \sum_{k=1}^{\infty}\sum_{i=1}^{\infty} |\mu(F_{i}^{k}\cap S_{k})| \geq \sum_{k=1}^{\infty} |\mu_{k}|(F) - \varepsilon,
  \end{equation*}
  which implies the claim, as $\varepsilon>0$ was arbitrary.

\end{proof}

\begin{lemma}\label{lemma:h_in_Linft}
Let $x'\in \mathrm{ca}(\pp)'$ such that there exists a measurable $h$ with $x'(\mu)=\int h d\mu$, for all $\mu\in \mathrm{ca}(\pp)$. Then $h$ is unique $\pp$-q.s. and $h\in\LL(\pp)$.
\end{lemma}

\begin{proof}
For uniqueness, assume to the contrary that there would exist measurable $h_1$ and $h_2$ which are not equal $\pp$-q.s. Then there exists $P\in\pp$ such that $P(h_1\ne h_2)>0$. W.l.o.g. we can assume that $P(A)>0$ where $A=\{h_1>h_2\}$. Note that $A\in\f$ by measurability of $h_1$, $h_2$. 
There exists $n\geq  1$ such that $P(A_n)>0$ with  $A_n=\{h_1-h_2\geq\frac1{n}\}$. Define a measure $\mu\in\mathrm{ca}(\pp)$ via $\frac{d\mu}{dP}=\ind_{A_n}$. Then
$$0=x'(\mu)-x'(\mu)=\int(h_1-h_2)d\mu\geq\frac1{n}P(A_n)>0,$$
a contradiction.

Now we will show that $h\in\LL(\pp)$. First we will show that $h\in\mathcal{L}^p(\pp),$ for all $1\leq p<\infty$. 
Indeed, as $x'\in  (\mathrm{ca}(\pp))'$, for all $\mu\in \mathrm{ca}(\pp)$, 
\begin{equation}\label{old:c}\biggl|\int h d\mu\biggr|=|x'(\mu)|\leq c \|\mu\|_{TV}.\end{equation}
where $c=\|x'\|_{\mathrm{ca}(\pp)'}$. For a fixed $P\in\pp$ and all $p$, $1\leq p<\infty$, define $\mu^{P,p}$ as follows
$$\frac{d\mu^{P,p}}{dP}=|h|^{p-1}(\ind_{\{h\geq0\}}-\ind_{\{h<0\}}).$$
Note that for the Hahn-Jordan decomposition of $\mu^{P,p}=(\mu^{P,p})^+-(\mu^{P,p})^-$ we have that $\Omega=
(\Omega^{P,p})^+\cup(\Omega^{P,p})^-$ where $(\Omega^{P,p})^+=\{h\geq0\}$ and 
$(\Omega^{P,p})^-=\{h<0\}$. Hence

$$\|\mu^{P,p}\|_{TV}=\mu^+((\Omega^{P,p})^+)+\mu^-((\Omega^{P,p})^-)=E_P[|h|^{p-1}].$$

Note further that 
$$\biggl|\int hd\mu^{P,p}\biggr|=\int|h|^{p-1}(\ind_{\{h\geq0\}}-\ind_{\{h<0\}})hdP=\int |h|^pdP=E_P[|h|^p].$$ 

Hence by \eqref{old:c} we get
$$E_P[|h|^p]\leq cE_P[|h|^{p-1}],$$
for all $p\geq1$.  Assume for the moment that $p\in\mathbb{N}$. Then applying the above inequality for $p-1,p-2,\dots, 1$ we get
$E_P[|h|^p]\leq c^p$ and thus $\|h\|_{L^p(P)}\leq c$,
for all $p\in\mathbb{N}$, which implies $\|h\|_{L^p(P)}\leq c$ also for the $p\geq 1$ which are not in $\mathbb{N}$. Apply this for every $P\in\pp$ to find, for all $p\geq1$,
\begin{equation}\label{old:uniform_bound_in_Lp}
\|h\|_{\LLp}=(\sup_{P\in\pp}E_P[|h|^p])^{\frac1{p}}\leq c.
\end{equation}
This does not only show that $h\in\mathcal{L}^p(\pp),$ for all $1\leq p<\infty$, but that the $p$-norms are all bounded by the uniform constant $c$.

An easy application of Markov's inequality, see also Lemma~13 of \cite{denis_hu_peng_2010} shows that, for all $\alpha>0$,
$$\sup_{P\in\pp}P(|h|\geq\alpha)\leq \frac{\sup_{P\in\pp}E_P[|h|^p]}{\alpha^p}.$$
Apply this for $\alpha=n$, $n\geq1$, to see that, by \eqref{old:uniform_bound_in_Lp} we get, for all $p\geq1$,

\begin{equation}\label{old:n}
\sup_{P\in\pp}P(|h|\geq n)\leq \left(\frac{c}{n}\right)^p    
\end{equation}
Choose now an arbitrary but fixed $n_0>c$ then $(\frac{c}{n_0})^p\to0$ for $p\to\infty$. Hence we get $\sup_{P\in\pp}P(|h|\geq n_0)=0$, hence $|h|\leq n_0$-q.s. and $h\in \mathcal{L}^{\infty}(\pp)$.
\end{proof}

\begin{proof}[Proof (Theorem \ref{thm:duality}).]
Let $\q \subseteq \pp$ be a Hahn localization of $\pp$.
We are first going to restrict our attention to the subspaces $\mathrm{ca}(Q)$ of $\mathrm{ca}(\pp)$, for $Q \in \q$.
Recall that $\mathrm{ca}(Q)$ is isometrically isomorphic to $L^{1}(Q)$ via 
\begin{equation*}
  T_{Q}: L^{1}(Q) \to \mathrm{ca}(Q), \quad f \mapsto \biggl(E \mapsto \int_{E}fdQ\biggr).
\end{equation*}
The inverse of $T_{Q}$ maps $\mu\in \mathrm{ca}(Q)$ to its Radon--Nikodym derivative $d \mu/dQ$.
If $x'\in (\mathrm{ca}(\pp))'$, then $x'\circ T_{Q}\in L^{1}(Q)' \simeq L^{\infty}(Q)$, i.e., there is a $g_{Q}\in L^{\infty}(Q)$ such that for all $f\in L^{1}(Q)$,
\begin{equation*}
  x'(T_{Q}(f)) = \int fg_{Q}dQ.
\end{equation*}
Since $Q$ is supported on $S_{Q}$, we have
\begin{equation*}
  \int fg_{Q}dQ = \int fg_{Q}\ind_{S_{Q}}dQ,
\end{equation*}
for every $f\in L^{1}(Q)$.
In other words, we can assume with out loss of generality that $g_{Q}$ is supported on $S_{Q}$ (otherwise we replace it by $g_{Q}\ind_{S_{Q}}$).
Now, for $\mu\in \mathrm{ca}(Q)$, we get
\begin{equation}\label{eq:l(m)=int}
  x'(\mu) = x'(T_{Q}T_{Q}^{-1}\mu) = \int g_{Q} \frac{d \mu}{dQ}dQ = \int g_{Q}d \mu.
\end{equation}

Next, we would like to \textit{glue} the $g_{Q}$'s into a $\pp$-quasi bounded function $h$, such that $h = g_{Q}$ $Q$-a.s. on every $S_{Q}$.
For $q\in \mathbb{Q}$ and $Q \in \q$ let
\begin{equation*}
  E_{q,Q} = \{ \omega \in S_{Q} \mid g_{Q}\geq q\} = S_{Q} \cap \{g_{Q}\geq q\} \in \f.
\end{equation*}
Since $E_{q,Q}\subseteq S_{Q}$ and $S_{Q}$ is a Hahn localization of $\pp$, the essential supremum $E_{q}$ of the family $\{E_{q,Q}\}_{Q\in \q}$ exists. 
We set
\begin{equation*}
  h(\omega) = \sup\{q\in \mathbb{Q}\mid \omega\in E_{q}\},
\end{equation*}
using the convention that $\sup \varnothing = -\infty$.
First we check that $h$ is measurable.
Observe that for $a\in \mathbb{R}$  
\begin{equation*}
  \begin{split}
    \{h>a\} &  = \{\sup\{q \in \mathbb{Q}\mid \omega \in E_{q}\} >a\}\\
            & = \{\omega \in \Omega\mid  \exists q \in \mathbb{Q}: q>a \text{ and }\omega \in E_{q}\} = \bigcup_{q \in \mathbb{Q}\atop q >a} E_{q} \in \f.
  \end{split}
\end{equation*}

Next, we want to show that $Q(E_{q,Q}\triangle (E_{q}\cap S_{Q})) = 0$.
To this end, define $F = \compl{(S_{Q}\setminus E_{q,Q})}$, and note that $E_{q,R}\setminus F = E_{q,R}\cap (S_{Q}\setminus E_{q,Q})$. 
This implies $Q(E_{q,R}\setminus F) = 0$, since $E_{q,R} \subseteq S_{R}$, and $R(E_{q,R}\setminus F) = 0$, since $E_{q,R}\setminus F \subseteq S_{Q}$.
Since $E_{q}$ is the essential supremum of $\{E_{q,Q}\mid Q\in \q\}$ and $Q(E_{q,Q}\setminus F) = 0$ for every $Q\in \q$, we have $Q(E_{q}\setminus F) = 0$ for every $Q\in \q$. 
Again, since $E_{q}$ is the essential supremum of $\{E_{q,Q}\mid Q\in \q\}$, we have that $Q(E_{q,Q}\setminus E_{q}) = 0$ for all $Q\in \q$. 
Since $E_{q}\cap (S_{Q}\setminus E_{q,Q}) = E_{q}\setminus F$, this implies
\begin{equation*}
  \begin{split}
    Q(E_{q,Q}\triangle (E_{q}\cap S_{Q})) & = Q(E_{q,Q}\setminus (E_{q}\cap S_{Q})) + Q((E_{q}\cap S_{Q})\setminus E_{q,Q}) \\
                                          &\leq Q(E_{q,Q}\setminus E_{q}) + Q(E_{q}\cap (S_{Q}\setminus E_{q,Q})) = 0. 
  \end{split}
\end{equation*}
Hence,
\begin{equation*}
  H_{Q} = \bigcup_{q\in \mathbb{Q}}E_{q,Q}\triangle (E_{q}\cap S_{Q})
\end{equation*}
is a $Q$ nullset.
In fact, $R(H_{Q}) = 0$ for every $R\in \q$, since $H_{Q}\subseteq S_{Q}$.

Now observe that for all $q\in \mathbb{Q}$ and $\omega\in S_{Q}\setminus H_{Q}$, we have
\begin{equation}\label{eq:set-equiv}
  \omega \in E_{q} \iff \omega \in E_{q,Q}.
\end{equation}
To see this, we first assume that $\omega\in S_{Q}\setminus H_{Q}$, but $\omega \not \in E_{q}$. If $\omega$ were in $E_{q,Q}$ it would be in $H_{Q}$, which is a contradiction.
On the other hand, assume that $\omega\in S_{Q}\setminus H_{Q}$, but $\omega \not \in E_{q,Q}$. If $\omega$ were in $E_{q}$, it would be in $E_{q}\cap S_{Q}$ and hence in $H_{Q}$, which is a contradition. 
The equivalence in \eqref{eq:set-equiv} implies that $h= g_{Q}$ on $S_{Q}\setminus H_{Q}$: If $\omega\in S_{Q}\setminus H_{Q}$, we have
\begin{equation*}
  \begin{split}
    h(\omega) &= \sup \{q\in \mathbb{Q}\mid \omega\in E_{q} \} = \sup\{q\in \mathbb{Q}\mid \omega \in E_{q,Q}\} \\
                       & = \sup\{q\in \mathbb{Q}\mid g_{Q}(\omega)\geq q\} = g_{Q}(\omega).
  \end{split}
\end{equation*}
In other words, $h = g_{Q}$ $Q$-a.s. on $S_{Q}$.
The map $l$ is an isometry by Lemma 3.4 in \cite{cuchiero2025quantitativehalmossavagetheoremsrobust}.
We still have to show that $h\in \mathbb{L}^{\infty}(\pp)$ and $x'(\mu) = \int h d \mu$ for all $\mu \in \mathrm{ca}(\pp)$.
Since every measure $\mu\in \mathrm{ca}(\pp)$ has a Jordan-Hahn decomposition into two non-negative measure $\mu^{+},\mu^{-}$ with $\mu = \mu^{+}-\mu^{-}$, it is enough to consider non-negative measures $\mu\in \mathrm{ca}(\pp)$.
For any such $\mu$, there are $P\in \pp$,  $\lambda_{k}\geq 0$, and $Q_{k}\in \q$ such that
\begin{equation*}
  \mu \ll P \ll \sum_{k=1}^{\infty} \lambda_{k} Q_{k}.
\end{equation*}
This implies 
\begin{equation*}
  \mu(A) = \sum_{k=1}^{\infty}\mu(A\cap S_{Q_{k}})
\end{equation*}
for all $A\in \f$.
We are going to show that $h \in L^{\infty}(\mu)$.

Recall that $h\ind_{S_{Q}} = g_{Q}$ $Q$-a.s. for every $Q\in \mathcal{Q}$. 
This implies 
\begin{equation*}
  \|h\ind_{S_{Q_{k}}}\|_{L^{\infty}(Q_{k})} = \|g_{Q_{k}}\|_{L^{\infty}(Q_{k})} = \|x'\circ T_{Q_{k}}\|_{\mathrm{ca}(Q)'} \leq \|x'\|\|T_{Q_{k}}\| = \|x'\|. 
\end{equation*}
Hence, we know that $|h|\ind_{Q_{k}} \leq \|x'\|$ $Q_{k}$-a.s for every $k\geq 1$. 
In particular, we have
\begin{equation*}
  \sum_{k=1}^{\infty}\lambda_{k} Q_{k}(|h| > \|x'\|) = 0,
\end{equation*}
and since $P\ll \sum_{k=1}^{\infty}\lambda_{k} Q_{k}$, we have that $P(|h| > \|x'\|) = 0$, and hence $\mu(|h| >\|x'\|) = 0$, i.e., $h\in L^{\infty}(\mu)$. 

We write $\mu_{k}$ for the non-negative finite measure $\mu_{k}(A) = \mu(A\cap S_{Q_{k}})$.
Note that the series $\sum_{k=1}^{\infty}\mu_{k} = \mu$ converges in total variation norm. Indeed, Lemma \ref{lemma:TV-norm-of-singular-sum} implies
\begin{equation*}
  \biggl\|\mu - \sum_{k=1}^{K}\mu_{k}\biggr\| = \biggl\|\sum_{k=K+1}^{\infty}\mu_{k}\biggr\| = \sum_{k=K+1}^{\infty}\|\mu_{k}\| \to 0,
\end{equation*}
as $K\to \infty$.
This implies in particular
\begin{equation*}
x'(\mu)=  \sum_{k=1}^{\infty}x'(\mu_{k}).
\end{equation*}
Since $h$ is $\mu$-a.s bounded, $h\in L^{1}(\mu)$, and hence 
\begin{equation*}
  \int h d \mu = \sum_{k=1}^{\infty}\int h\ind_{S_{Q_{k}}}d \mu = \sum_{k=1}^{\infty}\int h d \mu_{k}.
\end{equation*}
Since $\mu_{k} \ll Q_{k}$, and $g_{Q_{k}} = h$ $Q_{k}$-a.s. on $S_{Q_{k}}$ (and hence $\mu_{k}$-a.s), we have
\begin{equation*}
   \int h d \mu = \sum_{k=1}^{\infty}\int h d \mu_{k} = \sum_{k=1}^{\infty}\int g_{Q_{k}} d \mu_{k} = \sum_{k=1}^{\infty}x'(\mu_{k}) = x'(\mu).
\end{equation*}
Finally, by Lemma~\ref{lemma:h_in_Linft} above, we get that $h\in \mathbb{L}^{\infty}(\pp)$  and $h$ is unique $\pp$-q.s.
\end{proof}

\subsection{Proofs of the remaining results}\label{sec:preliminaries:proofs}
\begin{proof}[Proof (Lemma \ref{lemma:char-of-sigma-convex-dominated-families}).]
First, lets assume that there is a countable family of pairwise singular probability measures $\q$ such that $\q\lll \pp\lll \sconv(\q)$. 
Since $\q$ is countable, we can write $\q = \{Q_{1},Q_{2},\dots \}$, and define the measure
\begin{equation*}
  Q = \sum_{k=1}^{\infty} 2^{-k} Q_{k}. 
\end{equation*}
Clearly $\sconv(\q)\ll Q$, and hence $\pp$ is dominated by a $\sigma$-finite measure. 

On the other hand, assume that $\pp$ is dominated by a $\sigma$-finite measure $\mu$. 
Since $\mu$ is $\sigma$-finite, there are pairwise disjoint sets $S_{1},S_{2},\dots, \in \f$, such that  $\Omega=\bigcup_{j=1}^{\infty}S_j$ and $\mu(S_{j})<\infty$. 
For $j\in \mathbb{N}$, we define
\begin{equation*}
  \pp_{j} = \biggl\{ \frac{P(S_{j}\cap \cdot)}{P(S_{j})}\, \big|\, P(S_{j})>0 \biggr\}. 
\end{equation*}
Note that $\pp_{j}$ is still dominated by $\mu$.
The Halmos-Savage Theorem \cite{halmos-savage} implies that there is a non-trivial (i.e., $\lambda_{k}>0$) countable convex combination
\begin{equation*}
P_{j}^{*}(\cdot) = \sum_{k=1}^{\infty} \lambda_{k} \frac{P_{k}^{j}(S_{j}\cap \cdot)}{P_{k}^{j}(S_{j})},
\end{equation*}
of measures $P_{k}^{j}\in \pp$, with $P_{k}^{j}(S_{j})>0$, such that 
\begin{equation*}
 P_{j}^{*} \ll \pp_{j} \ll P_{j}^{*}. 
\end{equation*}
We define the probability measure $Q_{j}$ via
\begin{equation*}
  Q_{j}(F) = \frac{\sum_{k=1}^{\infty} 2^{-k} P_{k}^{j}(S_{j}\cap F)}{\sum_{k=1}^{\infty} 2^{-k} P_{k}^{j}(S_{j})},\quad F\in \f. 
\end{equation*}
Note that the $Q_{j}$'s are pairwise singular, since the $S_{j}$'s are pairwise disjoint.
Clearly, for every $Q_{j}$ there is a $P_{j}\in \pp$ such that $Q_{j} \ll P_{j}$, since $\pp$ is $\sigma$-convex, e.g., pick
\begin{equation*}
  P_{j} = \sum_{k=1}^{\infty} 2^{-k} P_{k}^{j}\in \pp,
\end{equation*}
and hence, $\q \coloneqq\{Q_{j}\mid j\in\mathbb{N} \} \lll \pp$. 

On the other hand, let $P\in \pp$. 
We define 
\begin{equation*}
  Q = \sum_{j=1}^{\infty} 2^{-j} Q_{j} \in \sconv(\q). 
\end{equation*}
We are going to show that $P\ll Q$. 
Assume that $Q(F) = 0$ for some $F\in \f$. 
This implies that $Q_{j}(F) = 0$ for all $j$, and hence 
\begin{equation*}
  \frac{P(S_{j}\cap F)}{P(S_{j})} = 0, \quad \text{for every $j\in \mathbb{N}$, such that } P(S_{j})>0.
\end{equation*}
This immediately implies 
\begin{equation*}
  P(F) = \sum_{j=1}^{\infty} P(S_{j}\cap F) = 0,
\end{equation*}
and hence $P \lll \sconv(\q)$. 

\end{proof}

\begin{proof}[Proof (Proposition \ref{thm:char:dedekind-complete}).]
  We start with $1 \Rightarrow 2$. 
  By Theorem \ref{thm:duality} we have $\mathrm{ca}(\pp)'\simeq \LL(\pp)$.
  Note that $\mathrm{ca}(\pp)$ is a Banach lattice, with the lattice order $\mu \leq \nu$ if and only if $\mu(F) \leq \nu(F)$ for all $F\in \f$ (cf. \cite{liebrich-model-uncertainty:reverse-approach}, section 2). 
  The least upper bound of two signed measures $\mu, \nu\in \mathrm{ca}(\pp)$ is given by $\mu \vee \nu = \frac{1}{2}\bigl( \mu+\nu + |\mu-\nu| \bigr)$, and the largest lower bound $\mu\wedge \nu$ is given by $\mu \wedge \nu = \frac{1}{2} \bigl(\mu + \nu - |\mu-\nu| \bigr)$, where $|\cdot|$ denotes the total variation measure. 
  Since $\mathrm{ca}(\pp)$ is a Banach lattice, Theorem 5.3 in \cite{liebrich-model-uncertainty:reverse-approach} implies that $\LL$ is Dedekind complete. 
  
  Next, we move on to $2\Rightarrow 3$. We associate the functions $g_{Q}$ with their $\LL$-equivalence classes. 
  The space $\LL$ is a Banach lattice with the order given by $h \leq g$ if and only if $h\leq g$ $\pp$-q.s. 
  If $\LL$ is Dedekind complete and $\mathcal{G} = \{g_{Q}\}_{Q\in\q}$ is as in the statement, we pick the (order)-supremum $g = \sup \mathcal{G}$ (i.e., the least upper bound on the set $\mathcal{G}$ in $\LL$, which exists by Dedekind completeness).
  We want to show that $g = g_{Q}$ on $S_{Q}$ $\pp$-q.s., i.e., we want to show that 
  \begin{equation*}
    P(S_{Q}\cap \{g\neq g_{Q}\}) = 0, \quad \text{for all }P\in \pp. 
  \end{equation*}
  Since $g$ is an upper bound on $\mathcal{G}$, we have $g_{Q}\leq g$ $\pp$-q.s., and hence we only need to show that $S_{Q}\cap \{g_{Q}<g\}$ is a $\pp$-null set. 
  Assume that this is not the case, i.e., there is a $P\in \pp$ with $P(S_{Q}\cap \{g_{Q}< g\}) >0$. 
  We set
  \begin{equation*}
    h = g_{Q}\ind_{S_{Q}} + g\ind_{\compl{S_{Q}}}, 
  \end{equation*}
and note that $h\leq g$, and $g_R \leq h$, $\pp$-q.s. for all $R\in \q$. 
Indeed, for $R=Q$, there is nothing to show, and for $R\neq Q$, we have $g_R= 0 \leq  g_Q = h$, on $S_Q$, and $g_Q \leq g = h $ on $\compl{S_Q}$, $\pp$-q.s. 
  Note that $\{h<g\} \subseteq S_{Q}$. 
  Furthermore, we have
  \begin{equation*}
    \{h < g\} = \{ \omega \in S_{Q}\mid h(\omega) < g(\omega)\} = \{\omega \in S_{Q}\mid g_{Q}(\omega) < g(\omega)\} = S_{Q}\cap \{g_{Q}<g\}. 
  \end{equation*}
  This implies that $P(h<g) = P(S_{Q}\cap \{g_{Q}<g\})>0$, and hence $g$ cannot be the least upper bound on $\mathcal{G}$ in $\LL$, leading to a contradiction. 

  Finally, we show $3\Rightarrow 1$. Let $E_{Q}\in \f$ with $E_{Q}\subseteq S_{Q}$, and set $g_{Q} = \ind_{E_{Q}}$.
  We pick $g$ as in the assumption, i.e., $g|_{S_{Q}} = \ind_{E_{Q}}$, and set $E = g^{-1}(\{1\})$. 
  Note that on $S_{Q}$ we have $\ind_{E} = \ind_{E_{Q}}$. 
  Now, we check that for all $Q\in \q$, $Q(E_{Q}\setminus E) = 0$. 
  Indeed, on $S_{Q}$ we have $g = g_{Q} = \ind_{E_{Q}}$, and hence
  \begin{equation*}
    \ind_{E_{Q}\setminus E} = \ind_{E_{Q}} \ind_{\compl{E}} = \ind_{E_{Q}} (1- \ind_{E}) = \ind_{E_{Q}}(1-\ind_{E_{Q}}) = 0.
  \end{equation*}
  This implies
  \begin{equation*}
    Q(E_{Q}\setminus E) = Q(S_{Q}\cap (E_{Q}\setminus E)) = \int_{S_{Q}} \ind_{E_{Q}\setminus E} = 0. 
  \end{equation*}
  Next, assume there is a $F\in \f$ with $Q(E_{Q}\setminus F) = 0$ for all $Q\in \q$. 
  We need to show that $Q(E\setminus F) = 0$ for all $Q\in \q$. 
  This is again a simple consequence of the fact that $\ind_{E} = \ind_{E_{Q}}$ on $S_{Q}$. 
  Indeed,
  \begin{equation*}
    Q(E\setminus F) = Q(S_{Q}\cap (E\setminus F)) = \int_{S_{Q}} \ind_{E}\ind_{\compl{F}} = \int_{S_{Q}} \ind_{E_{Q}}\ind_{\compl{F}} = Q(E_{Q}\setminus F) = 0. 
  \end{equation*}

\end{proof}

\begin{lemma}\label{lemma:completion-sigma-convex}
Let $(\Omega, \f, \pp)$ be a statistical model, then 
\begin{equation*}
  \f^{\pp} = \f^{\sconv(\pp)}. 
\end{equation*}
\end{lemma}
\begin{proof}
  Since $\pp \subseteq \sconv(\pp)$, we have $\f^{\sconv(\pp)} \subseteq \f^{\pp}$. 

  On the other hand, let $A\in \f^{\pp}$. 
  By \eqref{eq:P-completion}, we have
  \begin{equation*}
    \f^{\sconv(\pp)} = \bigcap_{R\in \sconv(\pp)} \{F \cup Z\mid F\in \f, Z\in \mathcal{Z}_{R}\}. 
  \end{equation*}
  In other words, we need to show that for every $R = \sum_{k=1}^{\infty} \lambda_{k} P_{k}\in \sconv(\pp)$, there is an $A_{R}\in \f$ and a $Z_{R}\in \mathcal{Z}_{R}$ with $A = A_{R}\cup Z_{R}$. 
  Since $A\in \f^{\pp}$, we know that for every $k\geq 1$, there is an $A_{k}\in \f$ and a $Z_{k}\in \mathcal{Z}_{P_{k}}$ with $A = A_{k}\cup Z_{k}$. 
  We define
  \begin{equation*}
    A_{R} = \bigcup_{k\geq 1} A_{k} \quad \text{and} \quad Z_{R} = \bigcap_{k\geq 1} Z_{k},
  \end{equation*}
  and note that 
  \begin{equation*}
    \begin{split}
      A\setminus A_{R} & = A\cap  \compl{\biggl( \bigcup_{k\geq 1} A_{k} \biggr)} = \bigcap_{k\geq 1} (A\cap \compl{A_{k}}) \subseteq \bigcap_{k\geq 1} Z_{k} = Z_{R}. 
    \end{split}
  \end{equation*}
  This implies $A = A_{R} \cup Z_{R}$, with $A_{R}\in \f$, and $Z_{R}\in \mathcal{Z}_{R}$, and hence $A\in \f^{\sconv(\pp)}$. 
  \end{proof}

  \begin{lemma}\label{lemma:E/EQ}
    Let $(\Omega, \mathcal{F}, \pp)$ be Hahn-localizable, with localization $\mathcal{Q}$, and support sets $S_{Q}$. 
    Let $E_{Q} \in \mathcal{F}$, with $E_{Q} \subseteq S_{Q}$ and denote with $E$ the essential supremum of the $E_{Q}$'s. 
    Then $Q(E\setminus E_{Q} ) =0$ for all $Q\in \mathcal{Q}$. 
  \end{lemma} 

  \begin{proof} 
    Assume there is a $R\in \mathcal{Q}$ with $R(E\setminus E_{R})>0$. 
    Set
    \begin{equation*}
      F = E\setminus (S_{R}\cap \compl{E_{R}}).
    \end{equation*}
     Recall that  $Q(E_Q\setminus E)=Q(E_Q\cap\compl{E})=0$ for all $Q\in\mathcal{Q}$ by the defining properties of the essential supremum $E$. Further,
    note that  $\compl{F}=\compl{E}\cup(S_R\cap \compl{E_R})$, and hence
    \begin{equation*}
      R(E_{R}\setminus F) = R(E_{R}\cap \compl{F}) \leq R(E_{R}\cap \compl{E}) + R(E_{R}\cap \compl{E}_{R}) = 0. 
    \end{equation*}
    For $Q\in \mathcal{Q}\setminus \{R\}$, we have
    \begin{equation*}
      Q(E_{Q}\setminus F) = Q(E_{Q}\cap \compl{F}) \leq Q(E_{Q}\cap \compl{E}) + Q(E_{Q}\cap S_{R}) = 0. 
    \end{equation*}
    By the defining properties of an essential supremum, this implies
     \begin{equation*}
       Q(E\setminus F) = 0
     \end{equation*}
      for all $Q\in \mathcal{Q}$. 
     In particular, we get for $Q=R$  that $R(E\setminus F)=0$.  Observe that further
     \begin{equation*}
     R(E\setminus E_{R})  = R(S_R\cap E\cap \compl{E_R})=R(E\setminus F)= 0,
     \end{equation*}
     which is a contradiction to $R(E\setminus E_{R}) >0$. 
  \end{proof}

\begin{proof}[Proof (Lemma \ref{lemma:pre-hahn-loc-implies-pre-hahn-loc-of-closure}).]
We start by showing $\fq \subseteq \f^{\pp}$, i.e., we need to show that for every $A\in \fq$ and $P\in \pp$ there are $A_{P}\in \f$ and $Z_{P}\in \mathcal{Z}_{P}$ with $A = A_{P}\cup Z_{P}$. 
Fix $A\in \fq$ and $P\in \pp$. 
By Lemma \ref{lemma:completion-sigma-convex} we have $\fq = \f^{\sconv(\q)}$.
Since $\pp\lll\sconv(\q)$, there is a $R\in \sconv(\q)$ with $P\ll R$. 
Since $A\in \fq = \f^{\sconv(\q)}$, there are $A_{R}\in \f$ and $Z_{R}\in \mathcal{Z}_{R}$ such that $A = A_{R}\cup Z_{R}$. 
Since $P\ll R$, we also have $Z_{R}\in \mathcal{Z}_{P}$, and hence $A\in \f^{P}$ for every $P\in \pp$, i.e., $A\in \f^{\pp}$. 
The same argument shows that $\f^{\pp} \subseteq \fq$.

We move on to the second point. 
Clearly, $Q(S_{R}) = \delta_{QR}$ for $Q,R\in \q$ is not affected by taking the $\pp$-closure of $\f$, since $S_{Q}\in \f$. 
Hence, we only need to show that $\q\lll \pp \lll \sconv(\q)$. 
Let us start with $\pp \ll \sconv(\q)$. 
We need to show that for all $P\in \pp$, there is a $R\in \sconv(\q)$ such that for all $A\in \f^{\pp}$, we have $R(A) = 0$ implies $P(A) = 0$.
We already know that for all $P\in \pp$, there is an $R\in \sconv(\q)$ such that for all $A\in \f$, we have $R(A) = 0$ implies $P(A) = 0$.
Hence, the difficulty lies in extending this property from $\f$ to $\f^{\pp}$. 
Assume that $R(A) = 0$ for some $A\in \f^{\pp}$.
We want to show that $P(A) = 0$. 
Since $\f^{\pp} = \fq = \f^{\sconv(\q)}$ by the first part of this lemma and Lemma \ref{lemma:completion-sigma-convex}, we get that there are $A_{R}\in \f$ and $Z_{R} \in \mathcal{Z}_{R}$ such that $A = A_{R}\cup Z_{R}$.
Hence, there is a $N_{R}\in \mathcal{N}_{R}$ with $Z_{R} \subseteq N_{R}$. 
Since $R(A) = R(A_{R}) = 0$ and $R(N_{R}) = 0$, and both $A_{R}$ and $N_{R}$ are elements of $\f$, we get $P(A_{R}) = P(N_{R}) = 0$. 
This implies $P(A) = P(A_{R}) = 0$, whenever $R(A) = 0$. 
The same argument shows that $\q\lll \pp$ as measures on $\f^{\pp}$, finishing the proof.

Now, lets assume that $(\Omega, \f, \pp)$ is Hahn-localizable. 
We have already shown that this implies that $(\Omega, \f^{\pp}, \pp)$ is pre-Hahn-localizable, and $\f^{\pp} = \fq$.
We have to show that every family $E_{Q}\in \f^{\pp}$ with $E_{Q} \subseteq S_{Q}$ admits an essential supremum, i.e., a set $S\in \f^{\pp}$ with
\begin{enumerate}
\item $Q(E_{Q}\setminus S) = 0$ for all $Q\in \q$, and
\item if for some $F\in \f^{\pp}$ we have $Q(E_{Q}\setminus F) = 0$ for all $Q\in \q$, then we also have $Q(S\setminus F) = 0$ for all $Q\in \q$. 
\end{enumerate}
Since $E_{Q} \in \f^{\pp} = \fq$ there are $E'_{Q}\in \f$ and $Z_{Q}\in \mathcal{Z}_{Q}$ with
\begin{equation*}
  E_{Q} = E_{Q}' \cup Z_{Q} \subseteq S_{Q}.
\end{equation*}
We define $S' \in \f$ to be the essential supremum of the family $\{E_{Q}'\}_{Q\in \q}$, which exists, since $(\Omega, \f, \pp)$ is Hahn-localizable, and set $Z = \bigcup_{Q\in \q} Z_{Q}$. 
Note that $Z$ is a $Q$-polar for every $Q$, since $Z = Z_{Q} \cup Z_{Q}'$ with
\begin{equation*}
  Z_{Q}'= \bigcup_{R\in \q\setminus \{Q\}} Z_{R} \subseteq \compl{S_{Q}}, 
\end{equation*}
where we have used that the localization $\q$ has strictly disjoint supports. 
Hence, the set 
\begin{equation*}
  S \coloneqq S'\cup Z
\end{equation*}
is in $\fq = \f^{\pp}$. 
The set $S$ is our candidate for the essential supremum of the family $\{E_{Q}\}_{Q\in \q}$. 
We first show that $Q(E_{Q}\setminus S) = 0$ for every $Q\in \q$.
Since $S' \subseteq S$, and $E_{Q} = E_{Q}'$ $Q$-a.s., we have
\begin{equation*}
  Q(E_{Q}\setminus S) = Q(E_{Q}'\setminus S) \leq Q(E_{Q}'\setminus S') = 0,
\end{equation*}
for every $Q\in \q$. 
Next, assume that there is an $F\in \f^{\pp}$ such that for every $Q\in \q$, we have $Q(E_{Q}\setminus F) = 0$. 
We need to to show that this implies $Q(S\setminus F) = 0$ for all $Q\in \q$. 
Fix $Q\in \q$, and recall that $Z$ is a $Q$-nullset. 
Lemma \ref{lemma:E/EQ} implies that $Q(S'\setminus E_{Q'}) = 0$ for all $Q\in \mathcal{Q}$, which in turn yields
\begin{equation*}
  \begin{split}
    Q(S\setminus F) &= Q(S'\setminus F) = Q(S'\cap \compl{F}\cap E_{Q}') + Q(S'\cap \compl{F}\cap \compl{(E_{Q}')})\\
                    & \leq Q(E_{Q}'\setminus F) + Q(S'\setminus E_{Q}') \\
                    & = Q(E_{Q}\setminus F) + Q(S'\setminus E_{Q}') = 0, 
  \end{split}
\end{equation*}
and hence $(\Omega, \f^{\pp},\pp)$ is Hahn-localizable.

\end{proof}

\begin{proof}[Proof (Lemma \ref{lemma:equivalence_(S)_preHahnloc}).]
  First, assume that Condition~2. of the lemma holds, i.e. $\pp$ is pre-Hahn-localizable  with localization $\mathcal{Q}$.
  It is easy to see that $\mathcal{Q}$ is a disjointly supported alternative to $\pp$.
  Indeed, disjointness is clear by point 1. of Definition \ref{def:hahn-loc}, and by point 2. of Definition \ref{def:hahn-loc} it follows that $\q\ll\pp\ll\q$. 
  To see that each $Q\in\q$ is supported observe that, for $A\in\mathcal{F}$ with $Q(A\cap S_Q)=0$ we have that for $R\in\q$, $R\not=Q$, $R(A\cap S_Q)\leq R(S_Q)=0$. Hence $\sup_{R\in\q}R(A\cap S_Q)=0$, and as $\pp\ll\q$ it follows that $A\cap S_Q$ is $\pp$-polar. 
  By Remark~\ref{rem:R(Q)-at-most-countable} (\ref{eq:prob1supports}) we have that $P \bigl( \bigcup_{Q\in \q(P)} S_{Q} \bigr) = 1$. Hence, $\mathcal{P}$ satisfies Condition~1. of the lemma with supported alternative $\rrr=\q$.

Now, let us assume that $\pp$ satisfies Condition~1., i.e., it has the class (S) property, and  for every $P\in \pp$ we have $P \bigl(\bigcup_{Q\in \mathcal{R}(P)}S_{Q} \bigr) = 1$. 
We set $\q = \mathcal{R}$. 
Clearly, for every $R,Q\in \q$ we have $R(S_{Q}) = \delta_{RQ}$. Indeed $Q(S_Q)=1$ by definition and for $R\in\q$, $R\not=Q$, $R(S_Q)=R(S_R\cap S_Q)=0$  as $S_R\cap S_Q$ is $\pp$-polar and $\q\ll\pp$. Hence, we only need to show that $\q \lll \pp \lll \sconv(\q)$. 
We start with $\pp \lll \sconv(\q)$. 
Fix any $P\in \pp$, and denote with $Q_{1},Q_{2},\dots$ the measures in $\q(P)$ (as in Definition \ref{P-positiveSQ}). 
For the sake of simplicity, we assume that there are infinitely many such $Q_{k}$'s, but the same argument applies to finitely many as well. 
We claim that $P \ll Q^{*} =  \sum_{k=1}^{\infty} 2^{-k} Q_{k}$. 
Let $S = \bigcup_{Q\in \q(P)}S_{Q}$, and note that $P(\compl{S}) = 0$ by assumption. 
This implies that for every $F\in \f$
\begin{equation*}
  P(F) = P(F\cap S) = P \biggl( \bigcup_{Q\in \q(P)} (F\cap S_{Q}) \biggr) = \sum_{Q\in \q(P)} P(F\cap S_{Q}),
\end{equation*}
since the $S_{Q}$'s are $\mathcal{P}$-q.s. disjoint. 
Now, assume that $Q^{*}(F) = 0$ for some $F\in \f$, then, clearly $Q_k(F)=0$, for all $k\geq1$. 
We are going to show that $P(F\cap S_{Q_{k}}) = 0$ for all $k\geq 1$. 
Note that $Q(F\cap S_{Q_{k}}) \leq Q(S_{Q_{k}})= 0$ for all $Q\in\q\setminus \{Q_{k}\}$, and $Q_{k}(F\cap S_{Q_{k}}) = Q_{k}(F) = 0$.
This implies 
\begin{equation*}
  \sup_{Q\in \q} Q(F\cap S_{Q_{k}})  = 0,
\end{equation*}
and since $\pp\ll\q$ by assumption, we have $R(F\cap S_{Q_{k}}) = 0$ for all $R\in \pp$.
In particular, we have $P(F\cap S_{Q_{k}}) = 0$ for all $k\geq 1$. 
Hence, $P(F)=0$ and thus $P\ll Q^*$, and we get that $\mathcal{P} \lll \sconv(\q)$. 

Next, we will show that $\q \lll \pp$. 
Fix $Q\in \q$, and set
\begin{equation*}
  \pp^{Q} = \{P \in \pp\mid P(S_{Q})>0\}.
\end{equation*}
First, we note a few things:
\begin{enumerate}
  \item[(i)] $\pp^{Q}\neq \varnothing$. Indeed, if $\pp^{Q} = \varnothing$, then $\sup_{P\in \pp} P(S_{Q}) = 0$, and hence we have $\sup_{R\in \q}R(S_{Q}) = 0$, which implies $Q(S_{Q}) = 0$, contradicting $Q(S_{Q}) = 1$. 
  \item[(ii)] $\pp^{Q}$ is $\sigma$-convex. If $P_{k}\in \pp^{Q}$, and $\lambda_{k}\geq 0$ with $\sum_{k=1}^{\infty} \lambda_{k} = 1$, then there is at least one $\lambda_{k'}>0$, and hence
    \begin{equation*}
      \sum_{k=1}^{\infty} \lambda_{k} P_{k}(S_{Q}) \geq \lambda_{k'}P_{k'}(S_{Q}) > 0.
    \end{equation*}
  \item[(iii)] For all $P\in \pp^{Q}$, $P$ and $Q$ cannot be singular. Indeed, suppose there would exist $P\in \pp^{Q}$ which is singular to $Q$. 
    In other words, there are disjoint sets $A, B\in \f$ with $Q(A) =1$, $Q(B) =0$, $P(A) = 0$, and $P(B) = 1$. 
    Set $\widetilde{ B } = B \cap S_{Q}$, and note that $P(\widetilde{ B } ) = P(S_{Q}) >0$ as $P\in\pp^{Q}$. But $Q(\widetilde{ B }) \leq Q(B) = 0$. 
    Since $Q$ is supported, $Q(\widetilde{ B }) = Q(\widetilde{ B }\cap S_{Q} ) = 0$ implies $\sup_{R\in \pp} R(\widetilde{ B }) = 0$, which in turn implies $P(\widetilde{ B }) = 0$, a contradiction.
\end{enumerate}
Now, fix $P\in \pp^{Q}$, and decompose $Q$ into its absolutely continuous and singular part with respect to $P$,
\begin{equation*}
  Q = Q^{c,P} + Q^{s,P}.
\end{equation*}
Hence, we can decompose $S_{Q}$ into 
\begin{equation*}
  S_{Q} = S_{Q}^{c,P} \uplus S_{Q}^{s,P},
\end{equation*}
with $P(S_{Q}^{c,P}) = P(S_{Q})>0$. 
We set 
\begin{equation*}
  \alpha = \sup_{P\in \pp^{Q}} Q(S_{Q}^{c,P})>0,
\end{equation*}
and claim that there is a $P^{*}\in \pp^{Q}$ such that 
\begin{equation}\label{eq:lemma:equivalence-of-conditions:alpha-claim}
  \alpha = Q\bigl(S_{Q}^{c,P^{*}}\bigr). 
\end{equation}
To prove the claim, let $P_{k}\in \pp^{Q}$ such that 
\begin{equation*}
  Q(S_{Q}^{c,P_{k}})> \alpha- 2^{k},
\end{equation*}
and set $P^{*} = \sum_{k=1}^{\infty} 2^{-k} P_{k}$. 
Since $\pp^{Q}$ is $\sigma$-convex, we have $P^{*}\in \pp^{Q}$. 
Now, set $Q^{c}(F) = Q(F\cap S_{Q}^{c,P^{*}})$, where
\begin{equation*}
  S_{Q}^{c,P^{*}} = \bigcup_{k=1}^{\infty} S_{Q}^{c,P_{k}}. 
\end{equation*}
Note that this is not an abuse of notation, since $Q^{c}$ as defined above is the absolutely continuous part of $Q$ with respect to $P^{*}$. 
Indeed, assume that $P^{*}(F) =0$ for some $F\in \f$. 
Then $P_{k}(F) = 0$ for all $k\geq 1$, and hence
\begin{equation*}
  Q^{c}(F) \leq \sum_{k=1}^{\infty} Q\bigl(F\cap S_{Q}^{c,P_{k}}\bigr) = 0,
\end{equation*}
since $Q(\cdot \cap S_{Q}^{c,P_{k}})$ is the absolute continuous part of $Q$ with respect to $P_{k}$. 
Next, we show that $P^{*}\bigl(S_{Q}\setminus S_{Q}^{c,P^{*}}\bigr) = 0$.
This is a simple calculation
\begin{equation*}
  \begin{split}
    P^{*}\bigl(S_{Q}\setminus S_{Q}^{c,P^{*}}\bigr) &= \sum_{k=1}^{\infty} 2^{-k} P_{k}\biggl(S_{Q}\cap \bigcap_{l\geq 1} \compl{\bigl(S_{Q}^{c,P_{l}}\bigr)}\biggr) \\
                                          &\leq \sum_{k=1}^{\infty}2^{-k}  P_{k}\bigl(S_{Q}\cap \compl{\bigl(S_{Q}^{c,P_{k}}\bigr)}\bigr) =0.
  \end{split}
\end{equation*}
In particular, this implies $P^{*}(S_{Q}) = P^{*}\bigl(S_{Q}^{c,P^{*}}\bigr) >0$. 
Moreover, we have 
\begin{equation*}
  \begin{split}
  Q(S_{Q}^{c,P^{*}}) = \lim_{n\to \infty} Q \biggl( \bigcup_{k=1}^{n} S_{Q}^{c,P_{k}} \biggr) \geq \lim_{n\to \infty} Q \bigl( S_{Q}^{c,P_{n}} \bigr) = \alpha,
  \end{split}
\end{equation*}
which proves the claim in \eqref{eq:lemma:equivalence-of-conditions:alpha-claim}. 

Next, we are going to show that $\alpha=1$. 
If this is not the case, i.e., if $\alpha = Q(S_{Q}^{c,P^{*}}) < 1$, then we set $C = S_{Q}\setminus S_{Q}^{c,P^{*}}$. Then $Q(C)=Q( (S_{Q}^{c,P^{*}})^c)=1-\alpha>0$.
Since $\q \ll \pp$, there exists $\widetilde{ P }\in \pp$ with $\widetilde{ P }(C)>0$. 
In particular, this implies $\widetilde{ P }(S_{Q}) \geq \widetilde{ P }(C) >0$ and hence $ \widetilde{ P }\in \pp^Q$. Further
\begin{equation}\label{eq:lemma:equivalence-of-conditions:P-tilde>0}
  \widetilde{ P }\bigl( C\cap S_{Q}^{c,\widetilde{ P }} \bigr) = \widetilde{ P }(C) >0,
\end{equation}
as $\widetilde{ P }(S_{Q}) = \widetilde{ P }\bigl( S_{Q}^{c,\widetilde{ P }}\bigr)$.
Suppose that $Q\bigl(C\cap S_{Q}^{c, \widetilde{ P }}\bigr)  =0$.
Since $Q$ is supported and $C\cap S_{Q}^{c,\widetilde{ P }} \subseteq S_{Q}$, this implies
\begin{equation*}
  \sup_{R\in \pp} R\bigl(C\cap S_{Q}^{c, \widetilde{ P }}\cap S_{Q}\bigr) = \sup_{R\in \pp} R\bigl(C\cap S_{Q}^{c,\widetilde{ P }}\bigr) = 0,
\end{equation*}
and hence $\widetilde{ P }(C\cap S_{Q}^{c,\widetilde{ P }}) = 0$, which is a contradiction to \eqref{eq:lemma:equivalence-of-conditions:P-tilde>0}. 
Hence, we have $Q\bigl(C\cap S_{Q}^{c,\widetilde{ P }}\bigr) >0$. 
Next, we set $P_{0} = \frac{1}{2}(P^{*} + \widetilde{ P})\in \pp^{Q}$ and note that $S_{Q}^{c,P_{0}} = S_{Q}^{c,P^{*}} \cup (C\cap S_{Q}^{c,\widetilde{ P }})$. Moreover 
$$Q(S_{Q}^{c,P_{0}})=Q(S_{Q}^{c,P^{*}})+Q(C\cap S_{Q}^{c,\widetilde{ P }})>\alpha$$ as $Q(C\cap S_{Q}^{c,\widetilde{ P }})>0$. This is a contradiction to the definition of $\alpha$.
Hence $\alpha =1$, and $Q\ll P^{*}$.

\end{proof}

\section{Hahn Extension}\label{sec:hahn_extension}
Let us start this section with an example. 
\begin{example}
  Consider the space $(\Omega, \f) = (\mathbb{R}, \mathcal{B}(\mathbb{R}))$, where $\mathcal{B}(\mathbb{R})$ denotes the Borel $\sigma$-algebra on $\mathbb{R}$, and  set $\q = \{\delta_{x}\mid x\in \mathbb{R}\}$, as well as $\pp = \sconv(\q)$. 
  Note that $\pp$ is pre-Hahn-localizable, but not Hahn-localizable. 
  This is mainly a measurability issue.
  Here, the sets in the definition of Hahn-localizabilty can be of the form $\bigcup_{x\in A}\{x\} = A$ for arbitrary subsets $A \subseteq \mathbb{R}$. 
 Note that any $P = \sum_{k=1}^{\infty} \lambda_{k} \delta_{x_{k}} \in \pp$ can assign probabilities to arbitrary subsets $A \subseteq \mathbb{R}$ via
 \begin{equation*}
   P(A) = \sum_{\substack{k\geq 1\\ x_{k}\in A}} \lambda_{k},
 \end{equation*}
 and not just Borel sets. 
 In this sense, the non-Hahn-localizability of $\pp$ is a consequence of the arguably too restrictive choice of the $\sigma$-algebra $\mathcal{B}(\mathbb{R})$. 
 If we define $\q$ (and hence $\pp$) instead on the powerset of $\mathbb{R}$, then $\pp$ is Hahn-localizable. 
\end{example}

The previous example begs the question: Can any pre-Hahn-localizable set of probability measures be made Hahn-localizable by extending the $\sigma$-algebra in a canonical way?
The somewhat surprising answer to this question is yes, almost. 

Throughout this section, we assume that $\pp$ is a pre-Hahn-localizable set of probability measures on a measurable space $(\Omega, \f)$, and that the supports of its localization $\q$ are pairwise disjoint.
The assumption of strictly disjointly supported localization $\q$ is mainly imposed for ease of presentation. 
We define the \textit{Hahn-extension} of $\f$ with respect to $\q$ as 
\begin{equation}\label{def:sigma-algebra-hahn-extension}
  \he = \sigma\biggl(\f\cup \biggl\{\bigcup_{Q\in \q} E_{Q}: E_{Q}\in \f, E_{Q}\subseteq S_{Q}\biggr\}\biggr).
\end{equation}
If we want to emphasize the dependence of $\he$ on $\f$, we write $\hef$. 
Note that the Hahn extension $\he$ depends on the choice of the family $\{S_{Q}\}_{Q\in \q}$, which is in general not unique. 
However, the sets $S_{Q}$ are unique up to $\pp$-polars, and we show in Corollary \ref{cor:hahn-ext-unique} that the $\pp$-completion of $\mathcal{H}_{\mathcal{F}}^{\mathcal{Q}}$ does not depend on the specific choice of the localization $\mathcal{Q}$. 
For the remainder of this chapter we fix a choice of supports $\{S_{Q}\}_{Q\in \q}$.

Let us give a brief overview of the remainder of this section.
Lemma \ref{lemma:hahn-closure} essentially states that any set $A\in \hef$ is in $\f$ \emph{locally}, i.e., on any support $S_{Q}$.
Using this lemma, Theorem \ref{thm:hahn-extension-is-well-defined} guarantes that any signed measure $\mu\in \mathrm{ca}(\pp)$ can be extended to a finite signed measure $\mu^{\q}$ on $\hef$. 
This allows us to define the set $\pq = \{P^{\q}\mid P\in \pp\}$. 
In Lemma \ref{lemma:hahn-extension-is-hahn-localizable} we show that $\pq$ on $\hef$ is Hahn-localizable with localization $\qq$ and  the corresponding support sets $S_{Q}$. 
The main consequence of this lemma is Theorem \ref{thm:duality-wo-hahn-loc}, in which we show that $\mathrm{ca}(\f,\pp)$ is isometrically isomorphic to $\mathrm{ca}(\hef, \pq)$, and hence 
\begin{equation*}
  \mathrm{ca}(\f,\pp)' \simeq \LL(\hef, \pq).
\end{equation*}
Theorem \ref{thm:duality-wo-hahn-loc} complements Theorem 5.3 of \cite{liebrich-model-uncertainty:reverse-approach}, since it gives the dual space of $\mathrm{ca}(\f,\pp)$ in the case where $\LL(\f,\pp)$ is not Dedekind complete. 
In Section \ref{sec:hahn-ext:min}, we show that the Hahn-extension is the smallest extension of $\mathcal{F}$ (up to $\mathcal{P}$-polars) for which $\LL(\pp)$ is the dual space of a normed space. 
As an immediate consequence of this, we get that the $\mathcal{P}$-completion of $\mathcal{H}_{\mathcal{F}}^{\mathcal{Q}}$ does not depend on the choice of support sets, and hence, is the smalles $\pp$-complete extension of the original $\sigma$-algebra for which $L^{\infty}(\pp)$ is the dual space of a normed space.

\begin{lemma}\label{lemma:hahn-closure}
  Let $\mathcal{P}$ be a pre-Hahn-localizable set of probability measures on $(\Omega, \mathcal{F})$ with strictly disjointly supported localization $\q$. 
For every $A\in \hef$ and $Q\in \q$, we have $A\cap S_{Q}\in \f$.
\end{lemma}
\begin{proof}
  We are going to use the \textit{principle of good sets}. 
  Let 
  \begin{equation*}
    \mathcal{G} = \{ A \in \hef\mid A \cap S_{Q}\in \f \text{ for all } Q\in \q\}.
  \end{equation*}
  Note that $\mathcal{G}$ contains the generator 
  \begin{equation*}
    \f\cup \biggl\{\bigcup_{Q\in \q} E_{Q}\mid E_{Q}\in \f, E_{Q}\subseteq S_{Q}\biggr\}
  \end{equation*}
  of $\hef$ (here we use that the supports are strictly disjoint).

  Next, we are going to show that $\mathcal{G}$ is a $\sigma$-algebra.
  Clearly, $\varnothing, \Omega\in \mathcal{G}$.
  For $A\in \mathcal{G}$, we have
  \begin{equation*}
    A^{c}\cap S_{Q} = S_{Q}\setminus (S_{Q}\cap A) \in \f
  \end{equation*}
  for every $Q\in \q$, and hence $A^{c}\in \mathcal{G}$.
  Finally, if $A_{i} \in \mathcal{G}$ is a countable family of sets, we have
  \begin{equation*}
  \biggl(\bigcup_{i=1}^{\infty}A_{i}\biggr)\cap S_{Q} = \bigcup_{i=1}^{\infty}(A_{i}\cap S_{Q})\in \f
  \end{equation*}
  for every $Q\in \q$, and hence $\bigcup_{i=1}^{\infty}A_{i}\in \mathcal{G}$, i.e., $\mathcal{G}$ is a $\sigma$-algebra.
  Now, we simply observe that
  \begin{equation*}
    \mathcal{G} \subseteq \hef\subseteq \sigma(\mathcal{G}) = \mathcal{G}.
  \end{equation*}
\end{proof}

\begin{theorem}\label{thm:hahn-extension-is-well-defined}
Let $\mathcal{P}$ be a pre-Hahn-localizable family of probability measures on $(\Omega,\f)$ with strictly disjointly supported localization $\q$, and $\hef$ its Hahn-extension. 
Then every $\mu\in \mathrm{ca}(\mathcal{P})$ can be extended to a finite signed measure $\mu^{\q}$ on $\hef$, and 
\begin{equation*}
  |\mu^{\q}| =|\mu|^{\q}.
\end{equation*}
\end{theorem}
\begin{proof}
  Let $\mu\in \mathrm{ca}(\pp)$, then $\q(\mu) = \{Q\in \q\mid |\mu|(S_{Q})>0\}$ is at most countable (see Remark \ref{rem:R(Q)-at-most-countable}). 
  For $A\in \hef$, we define 
  \begin{equation*}
    \mu^{\q}(A) = \sum_{Q\in \q(\mu)} \mu(A\cap S_{Q}).
  \end{equation*}
  First, we note that since $A\cap S_{Q}\in \f$ by Lemma \ref{lemma:hahn-closure}, the expression $\mu(A\cap S_{Q})$ is well-defined. 
  To make sure that $\mu^{\q}$ is well-defined, we have to argue why the family $\{\mu(A\cap S_{Q})\mid Q\in \q(\mu)\}$ is summable. 
  To this end, define $A_{Q} = A\cap S_{Q}$, and note that the $A_{Q}$'s are pairwise disjoint, since the $S_{Q}$'s are pairwise disjoint by assumption. 
  This implies 
  \begin{equation*}
    \sum_{Q\in \q(\mu)}|\mu|(A_{Q}) = |\mu|\biggl( \bigcup_{Q\in \q(\mu)} A_{Q} \biggr) \leq |\mu|(\Omega) < \infty,
  \end{equation*}
  and since $|\mu(F)| \leq |\mu|(F)$ for all $F\in \f$, this implies that the family $\{\mu(A_{Q})\mid Q\in \q(\mu)\}$ is (absolutely) summable. 
  Hence, $\mu^{\q}$ is well-defined. 

  Next, we have to show that $\mu^{\q}$ is a signed measure. 
Clearly, $\mu^{\q}(F) = \mu(F)$ for all $F\in \f$. 
This implies $\mu^{\q}(\varnothing) = 0$. 
Next, let $A_{k}\in \hef$ be pairwise disjoint for $k\in \mathbb{N}$. 
Since $\{\mu(A_{k}\cap S_{Q})\mid Q\in \q(\mu), k\geq 1\}$ is absolutely summable, Fubini's Theorem implies
\begin{equation*}
  \begin{split}
    \mu^{Q}\biggl( \bigcup_{k\geq 1} A_{k} \biggr) &= \sum_{Q\in \q(\mu)} \mu \biggl( \bigcup_{k\geq 1} A_{k}\cap S_{Q} \biggr) = \sum_{Q\in \q(\mu)} \sum_{k=1}^{\infty} \mu(A_{k}\cap S_{Q}) = \sum_{k=1}^{\infty} \sum_{Q\in \q(\mu)} \mu(A_{k}\cap S_{Q})\\
                                                   &= \sum_{k=1}^{\infty} \mu^{\q}(A_{k}),
  \end{split}
\end{equation*}
and $\mu^{\q}$ is a finite signed measure on $\hef$. 

Finally, we need to show that $|\mu^{\q}| = |\mu|^{\q}$.
First, we note that 
\begin{equation*}
  \mu^{\q} = \sum_{Q\in \q(\mu)} \mu_{Q}^{\q},
\end{equation*}
where $\mu_{Q}(F) = \mu(F\cap S_{Q})$. 
By Lemma \ref{lemma:TV-norm-of-singular-sum} this implies $|\mu^{\q}| = \sum_{Q\in \q(\mu)} |\mu^{\q}_{Q}|$.
On the other hand, we have $|\mu| = \sum_{Q\in \q(\mu)} |\mu_{Q}|$ by Lemma \ref{lemma:TV-norm-of-singular-sum}, which implies $|\mu|^{\q} = \sum_{Q\in \q(\mu)} |\mu_{Q}|^{\q}$. 
Hence, it is enough to show that $|\mu_{Q}^{\q}| = |\mu_{Q}|^{\q}$. 
For $A\in \hef$ we have
\begin{equation*}
  \begin{split}
    |\mu_{Q}^{\q}|(A) & = \sup \biggl\{ \sum_{l=1}^{\infty}| \mu_{Q}^{\q}(A_{l})| :  A_{l}\in \hef \text{ with } A = \biguplus_{l\geq 1}A_{l} \biggr\}\\
                      & = \sup \biggl\{ \sum_{l=1}^{\infty}| \mu_{Q}(A_{l}\cap S_{Q})|  :  A_{l}\in \hef \text{ with } A = \biguplus_{l\geq 1}A_{l} \biggr\}\\
                      &\leq \sup \biggl\{ \sum_{l=1}^{\infty} |\mu_{Q}(B_{l})| :   B_{l}\in \f \text{ with }A\cap S_{Q} = \biguplus_{l\geq 1} B_{l} \biggr\} \\
                      & = |\mu_{Q}|(A\cap S_{Q})\\
                      & = \sum_{R\in \q(\mu)} |\mu_{Q}|(A\cap S_{R}) = |\mu_{Q}|^{\q}(A),
  \end{split}
\end{equation*}
where we have used $|\mu_{Q}|(S_{R}) = 0$ for $R\neq Q$, in the last line. 
On the other hand, we have
\begin{equation*}
  \begin{split}
    |\mu_{Q}|^{\q}(A) &= |\mu_{Q}|(A\cap S_{Q}) \\
                      & = \sup \biggl\{ \sum_{l=1}^{\infty} |\mu_{Q}(B_{l})| :   B_{l}\in \f \text{ with }A\cap S_{Q} = \biguplus_{l\geq 1} B_{l} \biggr\}\\
                      & \leq  \sup \biggl\{ \sum_{l=1}^{\infty} |\mu_{Q}^{\q}(B_{l})| :   B_{l}\in \hef \text{ with }A\cap S_{Q} = \biguplus_{l\geq 1} B_{l} \biggr\}\\ 
                      & = |\mu_{Q}^{\q}|(A\cap S_{Q}) \leq |\mu_{Q}^{\q}|(A),
  \end{split}
\end{equation*}
and hence $|\mu_{Q}^{\q}| = |\mu_{Q}|^{\q}$, finishing the proof. 

\end{proof}

\begin{lemma}\label{lemma:hahn-extension:absolute-continuity}
Let $\mathcal{P}$ be a pre-Hahn-localizable family of probability measures on $(\Omega,\f)$ with strictly disjointly supported localization $\q$, and $\hef$ its Hahn-extension. 
For $\mu\in \mathrm{ca}(\pp)$ and $P\in \pp$, we have 
\begin{equation*}
  \mu \ll P \iff \mu^{\q}\ll P^{\q}.
\end{equation*}
\end{lemma}
\begin{proof}
  First, we show that $\mu^{\q}\ll P^{\q}$ implies $\mu \ll P$. 
  Let $F\in \f$ and $P(F) = 0$. 
  Since $P(F) = P^{\q}(F)$, this implies that $|\mu|(F) = |\mu^{\q}|(F) = 0$, and hence $\mu\ll P$. 

  Next, we show that $\mu \ll P$ implies $\mu^{\q}\ll P^{\q}$. 
  Assume that $P^{\q}(A) = 0$ for some $A\in \hef$. 
  Since
  \begin{equation*}
    P^\q(A) = \sum_{Q\in \q(P)} P(A\cap S_{Q}) = 0,
  \end{equation*}
  this implies that $P(A\cap S_{Q}) =0$ for all $Q\in \q$.
  Since $\mu\ll P$, we have $|\mu|(A\cap S_{Q}) = 0$ for all $Q\in \q$, and hence by Theorem \ref{thm:hahn-extension-is-well-defined}

   \begin{equation*}
    |\mu^{\q}|(A) = |\mu|^{\q}(A) =  \sum_{Q\in \q(\mu)} |\mu|(A\cap S_{Q}) = 0.
  \end{equation*}

\end{proof}

\begin{lemma}\label{lemma:hahn-extension-is-hahn-localizable}
Let $\mathcal{P}$ be a pre-Hahn-localizable family of probability measures on $(\Omega,\f)$ with strictly disjointly supported localization $\q$, and $\hef$ its Hahn-extension. 
Then $\pq$ is Hahn-localizable, and $\qq$ is a Hahn-localization of $\pq$ with support sets $S_{Q^{\q}} = S_{Q}$. 
\end{lemma}

\begin{proof}
  We first note that $R^{\q}(S_{Q}) = \delta_{RQ}$, since
  \begin{equation*}
    R^{\q}(S_{Q}) = R(S_{Q}) = \delta_{RQ},
  \end{equation*}
 and for every $A\in \hef$, as the supports are disjoint by assumption, we have 
  \begin{equation*}
    Q^{\q}(A) = \sum_{R\in \q(Q)} Q(A\cap S_{R}) = Q(A\cap S_{Q}). 
  \end{equation*}
  Furthermore, we note that Lemma \ref{lemma:hahn-extension:absolute-continuity} implies $\qq \lll \pq$. 
  Next, we show that $\pq \lll \sconv(\qq)$. 
  Let $P^{\q}\in \pq$, and $R = \sum_{k=1}^{\infty} \lambda_{k}Q_{k}\in \sconv(\q)$ with $P\ll R$. 
  We show that $R^{\q} = \sum_{k=1}^{\infty}\lambda_{k} Q_{k}^{\q}$. 
  This is an immediate consequence of Tonelli's Theorem, since $\q(R) \subseteq\{ Q_{k}\mid k\geq 1\}$. 
  For $A\in \hef$ we have
  \begin{equation*}
    R^{\q}(A) = \sum_{Q\in \q(R)} R(A\cap S_{Q}) = \sum_{Q\in \q(R)} \sum_{k=1}^{\infty} \lambda_{k} Q_{k}(A\cap S_{Q}\cap S_{Q_{k}}) = \sum_{k=1}^{\infty} \lambda_{k} Q_{k}^{\q}(A). 
  \end{equation*}

Finally, we establish the localization property (Property 2 of Definition \ref{def:hahn-loc}). 
Let $\{E_{R}\}_{R\in \qq}$ be a family of $\hef$ measurable sets with $E_{R}\subseteq S_{R}$.
Since $E_{R} = E_{R}\cap S_{R}\in \f$ (by Lemma \ref{lemma:hahn-closure}), the definition of $\hef$ implies that $E = \bigcup_{R\in \qq} E_{R}\in \hef$.
We are going to show that $E$ satisfies
\begin{enumerate}
  \item $R(E_{R}\setminus E) = 0$ for all $R\in \qq$.
  \item If for any set $F\in \hef$ we have $R(E_{R}\setminus F) = 0$ for all $R\in \qq$, then $R(E\setminus F) = 0$ for all $R\in \qq$.
\end{enumerate}
The first point is obvious, since $E_{R}\setminus E = \varnothing$ for all $R\in \qq$.
Next, let $F\in \hef$ such that $R(E_{R}\setminus F) = 0$ for all $R\in \qq$.
We have
\begin{equation*}
  R(E\setminus F) = R((E\setminus F)\cap S_{R}) = R(E_{R}\setminus F) = 0,
\end{equation*}
establishing the second property $E$ has to satisfy.
Hence, $E$ is the essential supremum of $\{E_{R}\}_{R\in \qq}$.

\end{proof}

\begin{theorem}\label{thm:duality-wo-hahn-loc}
 Let $\mathcal{P}$ be a convex and pre-Hahn-localizable family of probability measures on $(\Omega,\f)$ with strictly disjointly supported localization $\q$, and $\hef$ its Hahn-extension.  
Then $\mathrm{ca}(\f,\pp)$ is isometrically isomorphic to $\mathrm{ca}(\hef, \pq)$, and we have
  \begin{equation*}
    \mathrm{ca}(\f,\pp)' \simeq \mathbb{L}^{\infty}(\hef,\pq).
  \end{equation*}
\end{theorem}
\begin{proof}
We denote with $J: \mathrm{ca}(\hef, \pq) \to \mathrm{ca}(\f,\pp)$ the restriction onto $\f$, i.e., $J(\mu^{\q})(A) = \mu^{\q}(A) = \mu(A)$ for $A\in \f$.
The map $J$ is linear and surjective.
We are going to show that $J$ is an isometry, which automatically implies that $J$ is injective.
We have
\begin{equation*}
  \begin{split}
    \|\mu^{\q}\|_{TV} & = \sup \biggl\{ \sum_{i=1}^{\infty} \biggl|\sum_{Q\in \q(\mu)}\mu(E_{i}\cap S_{Q})\biggr|\, \bigg\vert \, \Omega = \biguplus_{i=1}^{\infty}E_{i}, E_{i}\in \hef \biggl\}\\
                               & \geq \sup \biggl\{ \sum_{i=1}^{\infty} \biggl|\sum_{Q\in \q(\mu)}\mu(E_{i}\cap S_{Q})\biggr|\, \bigg\vert \, \Omega = \biguplus_{i=1}^{\infty}E_{i}, E_{i}\in \f \biggl\} = \|\mu\|_{TV} = \|J(\mu^{\q})\|_{TV}.
  \end{split}
\end{equation*}

On the other hand, we have 
\begin{equation}
  \begin{split}
   \|\mu^{\q}\|_{TV} & = \sup \biggl\{ \sum_{i=1}^{\infty} \biggl|\sum_{Q\in \q(\mu)}\mu(E_{i}\cap S_{Q})\biggr|\, \bigg\vert \, \Omega = \biguplus_{i=1}^{\infty}E_{i}, E_{i}\in \hef \biggl\}\\
                     &\leq \sup \biggl\{ \sum_{i=1}^{\infty} \sum_{Q\in \q(\mu)}|\mu(E_{i}\cap S_{Q})|\, \bigg\vert \, \Omega = \biguplus_{i=1}^{\infty}E_{i}, E_{i}\in \hef \biggl\}\\
                     &\leq \sup \biggl\{ \sum_{i=1}^{\infty} |\mu(E_{i})|\, \bigg\vert \,\Omega = \biguplus_{i=1}^{\infty}E_{i}, E_{i}\in \f \biggr\} \\
                     &= \|\mu\|_{TV} = \|J(\mu^{\q})\|_{TV},
  \end{split}
\end{equation}
since $\{E_{i}\cap S_{Q}\mid i\geq 1, Q\in \q(\mu)\}$ is a family of pairwise disjoint sets from $\f$.

Since $\mathrm{ca}(\f,\pp)$ and $\mathrm{ca}(\hef,\pq)$ are isometrically isomorphic, so are their dual spaces.
By Lemma \ref{lemma:hahn-extension-is-hahn-localizable}, the family $\pq$ on $(\Omega,\hef)$ is Hahn-localizable. 
Theorem \ref{thm:duality} then implies
\begin{equation*}
  \mathrm{ca}(\f,\pp)' \simeq \mathrm{ca}(\hef,\pq)' \simeq \mathbb{L}^{\infty}(\hef, \pq).
\end{equation*}
\end{proof}

\subsection{Minimality and Uniqueness of the Hahn-extension}\label{sec:hahn-ext:min}
In this section, we are going to show that the Hahn-extension is in some sense the smallest extension of the original $\sigma$-algebra $\mathcal{F}$, for which $\LL(\pp)$ is a dual space of a normed space. 
This is summarized in the following proposition.
\begin{mprop}\label{prop:minimality-of-hahn-extension}
Let $(\Omega, \mathcal{G}, \pp)$ be a robust statistical model with convex $\pp$, and let $\mathcal{F} \subseteq \mathcal{G}$ be a sub-$\sigma$-algebra, such that $(\Omega, \mathcal{F}, \pp)$ is pre-Hahn-localizable, with localization $\mathcal{Q}$.
If $\LL(\mathcal{G}, \pp)$ is the dual space of a normed space, then
\begin{equation*}
  \mathcal{H}_{\mathcal{F}}^{\mathcal{Q}} \subseteq \mathcal{G}^{\pp}. 
\end{equation*}
\end{mprop}

As an immediate corollary to Proposition \ref{prop:minimality-of-hahn-extension}, we get that the $\pp$-completion of any two Hahn-extensions are identical.
\begin{corollary}\label{cor:hahn-ext-unique}
  Let $(\Omega, \mathcal{F},\pp)$ be pre-Hahn-localizable and $\pp$ convex, with two localizations $\mathcal{Q}$ and $\mathcal{Q}'$, and corresponding support sets $\{S_{Q}\mid Q\in \mathcal{Q}\}$ and $\{T_{Q'}\mid Q'\in \mathcal{Q}'\}$, then we have 
  \begin{equation*}
    (\mathcal{H}_{F}^{\mathcal{Q}})^{\pp} = (\mathcal{H}_{\mathcal{F}}^{\mathcal{Q}'})^{\pp}. 
  \end{equation*}
\end{corollary}
\begin{proof}
  By Theorem \ref{thm:duality}, we get that $\LL(\mathcal{H}_{\mathcal{F}}^{\mathcal{Q}'}, \pp)$ is the dual space of a normed space. 
  Hence, applying Proposition \ref{prop:minimality-of-hahn-extension} to $\mathcal{G} = \mathcal{H}_{\mathcal{F}}^{\mathcal{Q}'}$ yields 
  \begin{equation*}
    \mathcal{H}_{\mathcal{F}}^{\mathcal{Q}} \subseteq (\mathcal{H}_{\mathcal{F}}^{\mathcal{Q}'})^{\pp},
  \end{equation*}
  and since $(\mathcal{H}_{\mathcal{F}}^{\mathcal{Q}'})^{\pp}$ is $\pp$-complete, we get
  \begin{equation*}
    (\mathcal{H}_{\mathcal{F}}^{\mathcal{Q}})^{\pp} \subseteq (\mathcal{H}_{\mathcal{F}}^{\mathcal{Q}'})^{\pp}.
  \end{equation*}
  The reverse inclusion follows by symmetry. 
\end{proof}

The proof of Proposition \ref{prop:minimality-of-hahn-extension} requires a few lemmas. 

\begin{lemma}\label{lemma:min-ext:sigma-alg}
  In the same setting as in Proposition \ref{prop:minimality-of-hahn-extension}, the measures in $\pp$ can be extended to the $\sigma$-algebra
  \begin{equation*}
  \mathcal{Z} = \sigma \biggl( \mathcal{G}\cup \biggl\{ \bigcup_{Q\in \mathcal{Q}} E_{Q}\, \big|\, E_{Q}\subseteq S_{Q}, E_{Q}\in \mathcal{F} \biggr\} \biggr). 
  \end{equation*}
\end{lemma}
\begin{proof}Same as the proof of Lemma \ref{lemma:hahn-closure}. 
\end{proof}

\begin{lemma}\label{lemma:min-ext:esssup}
  In the same setting as in Proposition \ref{prop:minimality-of-hahn-extension}, every family $E_{Q}\in \mathcal{G}$ with $E_{Q}\subseteq S_{Q}$, has a $\pp$-essential supremum $E\in \mathcal{G}$, i.e.,
  \begin{enumerate}
  \item $P(E_{Q}\setminus E) = 0$, for all $P\in \pp$.
  \item If for some $G\in \mathcal{G}$, we have $P(E_{Q}\setminus G) = 0$ for all $P\in \mathcal{P}$, then $P(E\setminus G) = 0$ for all $P\in \pp$.
  \end{enumerate}
  Moreover, we have $\ind_{E\cap S_{Q}} = \ind_{E_{Q}}$ $\pp$-q.s. 
\end{lemma}
\begin{proof}
  Theorem 5.3 in \cite{liebrich-model-uncertainty:reverse-approach} implies that $\LL(\mathcal{G}, \pp)$ is Dedekind-complete. 
  Let $g$ be the (order)-supremum of $\ind_{E_{Q}}$, and set $E = g^{-1}(\{1\})$. 
  In the proof of Proposition \ref{thm:char:dedekind-complete}, we have already seen that $g\ind_{S_{Q}} =\ind_{E_{Q}}$ $\pp$-q.s.  
  Next, we show that $g = \ind_{E}$ $\pp$-q.s.
  Set $A = \{g \neq 0,1\}$, and recall that $g\ind_{S_{Q}} = \ind_{E_{Q}}$ $\pp$-q.s.
  Hence, $P(A) = 0$ for all $P\in \pp$, and $g\in \{0,1\}$ $\pp$-q.s. 
  In other words, we have $g = \ind_{E}$ $\pp$-q.s. 
  Note that this immediately implies $\ind_{E\cap S_{Q}} = \ind_{E_{Q}}$ $\pp$-q.s. 

Finally, we are going to show that $E$ is an essential supremum of $\{E_{Q}\}_{Q\in \mathcal{Q}}$. 
We start by showing that $P(E_{Q}\setminus E) = 0$ for all $P\in \mathcal{P}$. 
Note that 
\begin{equation*}
  \ind_{E_{Q}\setminus E} = \ind_{E_{Q}}\ind_{\compl{E}} = \ind_{E_{Q}}(1 - \ind_{E}) = \ind_{E_{Q}}(1-\ind_{E_{Q}}) = 0,\quad \pp\text{-q.s.,}
\end{equation*}
since $\ind_{E}\ind_{S_{Q}} = g\ind_{S_{Q}} = \ind_{E_{Q}}$ $\pp$-q.s., and $E_{Q}\subseteq S_{Q}$. 
Next, we assume that $G\in \mathcal{G}$ satisfies $P(E_{Q}\setminus G) = 0$ for all $P\in \pp$. We need to show that $P(E\setminus G) = 0$ for all $P\in \pp$. 
Since $\ind_{E\cap S_{Q}} = \ind_{E_{Q}}$ $\pp$-q.s., we have
\begin{equation*}
  P(E\setminus G) = \sum_{Q\in \mathcal{Q}(P)} P(S_{Q}\cap E\cap \compl{G}) = \sum_{Q\in \mathcal{Q}(P)} P(E_{Q}\cap \compl{G})  = 0. 
\end{equation*}
\end{proof}

\begin{proof}[Proof (Proposition \ref{prop:minimality-of-hahn-extension})]
  We will show that $\mathcal{G}^{\pp}$ contains the generator
  \begin{equation*}
    \mathcal{F}\cap \biggl\{ \bigcup_{Q\in \mathcal{Q}} E_{Q}\, \Big|\, E_{Q}\subseteq S_{Q}, E_{Q}\in \mathcal{F} \biggr\} 
  \end{equation*}
  of $\mathcal{H}_{\mathcal{F}}^{\mathcal{Q}}$. 
  We have $\mathcal{F} \subseteq \mathcal{G}$ by assumption, and hence, we only need to show that for every $E_{Q}\in \mathcal{F}$ with $E_{Q} \subseteq S_{Q}$, we have an $E\in \mathcal{G}$ such that 
  \begin{equation*}
    P \biggl( E \triangle \biggl( \bigcup_{Q\in \mathcal{Q}}E_{Q} \biggr) \biggr) = 0. 
  \end{equation*}
  Note that by Lemma \ref{lemma:min-ext:sigma-alg}, the above expression is well-defined
  We choose the essential supremum $E$ of $\{E_{Q}\}_{Q\in \mathcal{Q}}$ (whose existence is guaranteed by Lemma \ref{lemma:min-ext:esssup}), and set $U=\bigcup_{Q\in \mathcal{Q}} E_{Q}$. 
  First, note that 
  \begin{equation*}
    P(U\setminus E) = \sum_{Q\in \mathcal{Q}(P)} P(S_{Q}\cap U \cap \compl{E}) = \sum_{Q\in \mathcal{Q}(P)} P(E_{Q}\setminus E) = 0.
  \end{equation*}
  Moreover, since $\ind_{E\cap S_{Q}} = \ind_{E_{Q}}$ $\pp$-q.s. (see Lemma \ref{lemma:min-ext:esssup}), we have
  \begin{equation*}
    P(E\setminus U) = \sum_{Q\in \mathcal{Q}(P)} P(S_{Q}\cap E \cap \compl{U}) \leq \sum_{Q\in \mathcal{Q}(P)} P(S_{Q}\cap E \cap \compl{E_{Q}}) = 0. 
  \end{equation*}

\end{proof}

\section{Examples of pre-Hahn-localizable models and a non-example}\label{sec:examples}
Let us start with the most simple example: arbitrary discrete measures.
\begin{example}[Arbitrary families of discrete measures]
  Let $\Omega$ be an arbitrary set, $A \subseteq \Omega$ an arbitrary subset, and consider $\mathcal{Q} = \{\delta_{x}\mid x\in A\}$, the set of Dirac probability masses on $A$, and $\pp = \sconv(\mathcal{Q})$. 
  Clearly, $\pp$ is pre-Hahn-localizable, and if the $\sigma$-algebra on $\Omega$ is chosen to be its power set, it is Hahn-localizalbe. 
\end{example}

The next class of examples arises whenever the ``parameters'' of a stochastic model determine the distribution of a stochastic process uniquely, and can be consistently estimated, in the sense that a strong law of large numbers applies. 
Here, a dichotomy arises. 
In many cases the law of a finite sequence $X_{1},\dots,X_{n}$ of such random variables is dominated by a $\sigma$-finite measure, but the law of the entire process $(X_{t})_{t\geq 1}$ is not. 
Our theory covers this exact case. 
This is illustrated in the following example. 

\begin{example}
  Let us denote with $\mathcal{H}$ the set of all positive semi-definite symmetric $d\times d$ matrices.
  We assume that the sequences of i.i.d. $\mathbb{R}^{d}$-valued Gaussian random variables $X^{\mu, \Sigma} = (X_{t}^{\mu, \Sigma})_{t\geq 1}$, $\mu\in \mathbb{R}^{d}$ and $\Sigma\in \mathcal{H}$, are defined on a common probability space  $(\Omega, \mathcal{F}, P)$. 
  Clearly, the pushforward of $P$  under any finite sequence $(X_{1}^{\mu, \Sigma}, \dots, X_{n}^{\mu, \Sigma})$ is dominated by the Lebesgue measure. 
  However, the pushforward of $P$ under the entire sequence $X^{\mu, \Sigma}$ is supported on the set
  \begin{equation*}
    S_{\mu, \Sigma} = \biggl\{ x\in \mathbb{R}^{\mathbb{N}}\, \Big|\, \lim_{n\to \infty} \frac{1}{n} \sum_{k=1}^{n} x_{k} = \mu\text{ and } \lim_{n\to\infty} \frac{1}{n}\sum_{k=1}^{n}x_{k}x_{k}^{T} = \Sigma\biggr\}. 
  \end{equation*}
  Note that $S_{\mu_{1}, \Sigma_{1}}$ and $S_{\mu_{2},\Sigma_{2}}$ are disjoint, unless $\mu_{1} = \mu_{2}$ and $\Sigma_{1} = \Sigma_{2}$. 
We define $\mathcal{Q} = \{Q_{\mu,\Sigma} = P_{X^{\mu, \Sigma}}\mid \mu\in \mathbb{R}^{d}, \Sigma \in \mathcal{H}\}$, and $\pp = \sconv(\mathcal{Q})$. 
Then $\pp$ is pre-Hahn-localizable, with localization $\mathcal{Q}$. 

It is important to note that there is nothing special about Gaussian measures. 
The same argument applies to any sequence of random variables whose distribution is uniquely determined by a set of parameters we can estimate almost surely. 
In particular, the sequences need not be i.i.d. 
\end{example}

The next example shows that the parameters in question by no means need to be finite dimensional. 
\begin{example}
We denote with 
\begin{equation*}
  \mathcal{D} = \biggl\{  g(t)=\sum_{h\in \mathbb{Z}} \gamma(h) e^{iht}\, \Big| \, \sum_{h\in \mathbb{Z}}|\gamma(h)| < \infty, \gamma(0)> 0, \text{ and } \gamma(h) = \gamma(-h), h\in \mathbb{Z} \biggr\}
\end{equation*}
the set of potential spectral densities of a real-valued stochastic process. 
We assume that the gaussian processes $X^{\mu,g} = (X_{t}^{\mu,g})_{t\in \mathbb{Z}}$ with mean $\mu$ and spectral density $g$ are defined on a common probability space $(\Omega, \mathcal{F}, P)$, for $\mu\in \mathbb{R}$, and $g\in \mathcal{D}$. 
The pushforward of $P$ under $X^{\mu,g}$ is supported on the set
\begin{equation*}
  S_{\mu, g} = \biggl\{ x\in \mathbb{R}^{\mathbb{Z}}\, \Big|\, \lim_{n\to \infty} \frac{1}{n} \sum_{t=1}^{\infty} x_{t} = \mu, \lim_{n\to \infty} \frac{1}{n} \sum_{t=1}^{n} x_{t+l}x_{t+j} = \frac{1}{2 \pi}\int_{0}^{2 \pi} g(\theta) e^{-i(l-j)\theta} d\theta, l,j\in \mathbb{Z} \biggr\}. 
\end{equation*}
Again, the sets $S_{\mu_{1}, g_{1}}$ and $S_{\mu_{2}, g_{2}}$ are disjoint, unless $\mu_{1} = \mu_{2}$ and $g_{1} = g_{2}$. 
We set $\mathcal{Q} = \{P_{X^{\mu, g}}\mid \mu\in \mathbb{R}, g\in \mathcal{D}\}$, and $\pp = \sconv(\mathcal{Q})$. 
  The set $\pp$ is pre-Hahn-localizable with localization $\mathcal{Q}$. 
\end{example}
The above examples work in every parametric model, whose parameters can be estimated almost surely.

\begin{example}[Parametric models]
Let $(\Omega, \mathcal{F}, P)$ be a probability space, and let $P_{\theta}$ be the distribution of the sequence $X^{\theta} = (X_{n}^{\theta})_{n\geq 1}$ under $P$, where $\theta \in \Theta \subseteq M$ is a parameter taking values in some metric space $(M,d)$.
We assume that the $X_{n}^{\theta}$ take values in some measurable space $(\Sigma, \mathcal{G})$.
We say that the parametric model $(P_{\theta})_{\theta\in \Theta}$ satisfies an abstract law of large numbers, if for every $n\geq 1$, there is a measurable function $T_{n}: \Sigma^{n}\to M$, such that 
\begin{equation*}
  T_{n}(X_{1}^{\theta}, \dots, X_{n}^{\theta}) \to \theta, \quad P\text{-a.s.}
\end{equation*}
In the above setting, the measures $P_{\theta}$ are supported on the sets 
\begin{equation*}
  T_{\theta} = \biggl\{ x \in \Sigma^{\mathbb{N}} \, \Big|\,  \lim_{n\to \infty} T_{n}(x_{1},\dots, x_{n})  = \theta\biggr\}, 
\end{equation*}
and the family $\pp = \sconv(P_{\theta}\mid \theta \in \Theta)$ is pre-Hahn-localizable. 
\end{example}

The following corollary to Kakutani's Theorem shows that this principle extends to non-parametric settings as well. 
\subsection{A corollary to Kakutani's Theorem}
Let us briefly recall the setup of Kakutani's seminal paper \cite{kakutani}. 
Let $(\Omega_{k}, \mathcal{F}_{k}, \mu_{k})$ be a sequence of probability spaces for $k\in \mathbb{N}$.
We set $\Omega = \bigtimes_{k\geq 1} \Omega_{k}$, $\mathcal{F} = \bigotimes_{k\geq 1} \mathcal{F}_{k}$, and $\mu = \bigotimes_{k\geq 1} \mu_{k}$. 
If we are given a family of sequence of measures $Q_{k}^{i}$ on $\mathcal{F}_{k}$, such that $Q_{k}^{i}\sim \mu_{k}$, meaning that $Q_{k}^{i}$ and $\mu_{k}$ have the same nullsets, we can define the quasi-metric
\begin{equation*}
  \rho(Q_{k}^{i}, \mu_{k}) = \int_{\Omega_{k}} \sqrt{\frac{dQ_{k}^{i}}{d\mu_{k}}} d \mu_{k}. 
\end{equation*} 
The probability measures $Q = \bigotimes_{k\geq 1} Q_{k}^{i}$ und $\mu$ are singular, if and only if $\prod_{k\geq 1} \rho(Q_{k}^{i}, \mu_{k})>0$ (see p. 218, \cite{kakutani}). 
Moreover, we can define 
\begin{equation*}
  \varphi_{k}^{i} = \frac{d Q_{k}^{i}}{d \mu_{k}} \circ p_{k},
\end{equation*}
where $p_{k}: \Omega\to \Omega_{k}$ is the projection onto the $k$-th coordinate. 
It has been shown (p. 221, \cite{kakutani}) that $\prod_{k=1}^{n} \varphi_{k}^{i}$ converges $\mu$-a.s.
Let 
\begin{equation*}
  \mathcal{C}_{0} = \biggl\{\lim_{n\to \infty} \prod_{k=1}^{n} \varphi_{k}^{i}\mid i\in I\biggr\},
\end{equation*}
be the set of potential limit functions.
From $\mathcal{C}_{0}$, we select (by, e.g., using the Zorn's Lemma) a family of functions which are different at every $\omega\in \Omega$, i.e., for all $\omega\in \Omega$ and $f,g\in \mathcal{C}$, we have $f(\omega) \neq g(\omega)$. 
We have the following corollary.
\begin{corollary}
  The family $\mathcal{P} = \sconv(\{\bigotimes_{k\geq 1} Q_{k}^{i_{f}}\mid f\in \mathcal{C}\})$ is pre-Hahn-localizable. 
\end{corollary}
\begin{proof}
  The measure $Q^{f} = \bigotimes_{k\geq 1} Q_{k}^{i_{f}}$ is supported on the set
  \begin{equation*}
    S_{f} = \biggl\{ \omega\in \Omega\mid \lim_{n\to \infty}\prod_{k=1}^{n} \varphi_{k}^{i_{f}}(\omega) = f(\omega)\biggr\}. 
  \end{equation*}
  Since the functions in $\mathcal{C}$ are different at each point, we get $S_{f}\cap S_{g} = \varnothing$. 
\end{proof}
 The next example shows that the theory applies to continuous-time processes as well. 
\subsection{Black-Scholes model with volatility and drift uncertainty}
In this example, we show that our theory can be used to  construct a  Black-Scholes model with uncertain constant volatility and drift and robust $L^1$-$L^{\infty}$ duality. 
Following Example~3.8 in  \cite{cohen} and the discussion in \cite{liebrich-model-uncertainty:reverse-approach} on volatility uncertainty (see 3.2.2 there) we define a special case of their setting.
Let $\Omega$ be the Wiener space with canonical process $(W_t)_{t\geq0}$ starting at 0.
That is, $\Omega$ consists of all continuous functions $\omega:[0,+\infty)\to\mathbb{R}$ with $\omega(0)=0$ and $W_t(\omega)=\omega(t)$, $t\in[0,+\infty)$. Let $(\mathcal{F}_t)_{t\geq0}$ be the natural filtration, i.e., $\mathcal{F}_t=\sigma(\{W_s, 0\leq s\leq t\})$, $t\geq0$.
Let $P_0$ be the Wiener measure. Then $(W_t)_{t\geq0}$ is a standard Brownian motion with respect to $P_0$ for the given filtration. 
There is an adapted process $\langle W\rangle$ such that under each probability measure $P$, with respect to which $(W_t)_{t\geq0}$ is a local martingale, $\langle W\rangle$ agrees with the usual quadratic variation, see \cite{Kar}. Fix a $0<T<\infty$ and a set $\Sigma\subseteq(0,+\infty)$ of possible constant strictly positive volatilities  $\sigma\in\Sigma$.  
For each $\sigma\in\Sigma$ let $Q^{\sigma}$ be such that under $Q^{\sigma}$ the canonical process is a local martingale with quadratic variation $v^{\sigma}_t=\sigma^2t$, $0\leq t\leq T$. Define $S_{Q^{\sigma}}=\{\omega: \langle W\rangle_t=v_t^{\sigma}\text{ for all }s\leq T\}\in\mathcal{F}_T$. Then $Q^{\sigma}(S_{Q^{\sigma}})=1$ and the sets  $S_{Q_{\sigma_{1}}}$ and $S_{Q_{\sigma_{2}}}$ are disjoint for $\sigma_1\neq \sigma_2$.
Observe that this is a special case of  the example in \cite{cohen} and the sets given in \cite{liebrich-model-uncertainty:reverse-approach}, which treat more general choices of processes $v^{\sigma}_t$. This already shows that we can treat a robust Black-Scholes model with uncertain constant volatility as the given setting is pre-Hahn-localizable with disjoint support sets $S_{Q^{\sigma}}$. 
Let us now find a way to introduce an uncertain constant drift parameter $\mu$. To this end let $M\subseteq\mathbb{R}$ be a set of possible constant drift parameters. Fix an arbitrary $\sigma\in\Sigma$. 
Observe that $W_t=\sigma W^{\sigma}_t$, $Q^{\sigma}$-a.s., where $(W^{\sigma}_t)_{t\geq0}$ is a $Q^{\sigma}$-standard Brownian motion. For each $\mu\in M$ define a probability measure $P^{\mu,\sigma}\sim Q^{\sigma}$ with $\frac{dP^{\mu,\sigma}}{dQ^{\sigma}}=\exp(\frac{\mu}{\sigma}W^{\sigma}_T-\frac{\mu^2}{2\sigma^2}T)$. Then, by Girsanov's Theorem, see, e.g., \cite{KaratzasShreve},  $W^{\mu, \sigma}_t:=W^{\sigma}_t-\frac{\mu}{\sigma}t$, $t\in[0,T]$, is a standard Brownian motion with respect to $P^{\mu,\sigma}$. Moreover,   the following holds $P^{\mu,\sigma}$-a.s.
\begin{equation}\label{robustBS}
W_t=\sigma W^{\sigma}_t=\sigma\biggl(W^{\sigma}_t-\frac{\mu}{\sigma}t\biggr)+\mu t=\sigma W^{\mu, \sigma}_t +\mu t.
\end{equation}
Define now $\mathcal{P}=\sconv\{P^{\mu,\sigma}, \mu\in M, \sigma\in\Sigma\}$ and $\mathcal{Q}=\{Q^{\sigma}, \sigma\in\Sigma\}$. 
Then $\q\lll \pp\lll \sconv(\q)$ as, for each $\sigma\in\Sigma$, we have that $P^{\mu,\sigma} \sim Q^\sigma$ for each $\mu \in M$.
This shows that $\pp$ is pre-Hahn-localizable with
strictly disjointly supported localization $\q$. Define now the stock price process $(S_t)_{t\geq0}$ on the Wiener space with $S_0=s_0\in(0,+\infty)$ as  the stochastic exponential of the canonical process under the measures $P\in\mathcal{P}$, i.e., for each $P\in\pp$, $S$ satisfies the following stochastic differential equation:
$$dS_t=S_tdW_t.$$ Then, for $P=P^{\mu,\sigma}$, by (\ref{robustBS}), the stochastic differential equation looks as follows
$$dS_t=S_t(\sigma dW^{\mu,\sigma}_t+\mu dt),$$
which is the usual definition of the stock price process in a Black-Scholes model with $\mu$ and $\sigma$ fixed, that is, $S_t=s_0\exp(\sigma W^{\mu,\sigma}_t+\mu t), t\in[0,T]$, $P^{\mu,\sigma}$-a.s.
Therefore $(S_t)_{t\in[0,T]}$ together with $\pp$ can be seen as robust version of the classical Black Scholes model with uncertain drift and volatility (which includes even mixtures of countably many choices of parameters in the measures). As $\pp$ is pre-Hahn-localizable our theory can be applied to obtain a probability space, for which this robust Black-Scholes model  satisfies the robust $L^{1}$-$L^{\infty}$ duality. The model can be further generalized by using the ideas of \cite{cohen} and the discussion of \cite{liebrich-model-uncertainty:reverse-approach} on volatility uncertainty.
Looking at the examples above, one could gain the impression that our theory only applies to infinite sequences. 
The next example shows that this is not the case.

\subsection{A one step robust binomial model and its Hahn extension}\label{binom1}   
We will define a robust one step binomial model. Note that this will just be the robust binomial model  of \cite{blanch_carassus} for the time set $\mathbb{T}=\{0,1\}$.
Let us recall the model. Let $\Omega=(0,\infty)$ with the Borel $\sigma$-algebra $\mathcal{F}=\mathcal{B}((0,\infty))$ on it. Assume that the riskless interest rate $r=0$.  The price process of the risky asset $(S_t)_{t\in\mathbb{T}}$ is given by $S_0=1$ and $S_{1}=S_0Y_{1}$, where $Y_{1}: (0,+\infty)\to(0,+\infty)$ is a bijective measurable
map. In the one-step case, we can just use the identity map $Y_1(\omega)=\omega$. Let $\mathcal{M}_1$ be the set of all probability measures on $((0,\infty),\mathcal{F})$. Define the model parameters with uncertainty such that the assumptions for each $t$ in \cite{blanch_carassus} apply here for the easy case $t=0$ as follows:
\begin{mass}[Blanchard, Carassus for $t=0$]\label{ass_blan_car_1step} Let $u_0, U_0, d_0, D_0,\pi_0, \Pi_0\in\mathbb{R}$  such that
\begin{enumerate}
\item $0<{\pi}_0\leq  {\Pi}_0<1$,
\item $d_0\leq D_0$, $u_0\leq U_0$,
\item $0<d_0<1<U_0$.
\end{enumerate}
\end{mass}
Let us now define the set of possible models which amounts to the corresponding definitions for $t=T=1$ in \cite{blanch_carassus} with a similar notation as in \cite{liebrich-model-uncertainty:reverse-approach}. For the definition of the set $\mathcal{P}$ the following sets are used. Let $E_0\subset \mathbb{R}^3$ be defined as $E_0=[u_0,U_0]\times[d_0,D_0]\times [{\pi}_0,{\Pi}_0]$ and let $\mathcal{L}_1=\{\pi\delta_{u}+(1-\pi)\delta_{d}\mid (u,d,\pi)\in E_0\}$. We will now define the crucial sets of probability measures describing the uncertainty set of the robust one step binomial model. For a probability measure $R$ on $(\Omega, \mathcal F)$ denote by $R\circ Y_1^{-1}$ the law of $Y_1$ with respect to $R$, that is, for $A\in\mathcal{B}((0,\infty))$, $R\circ Y_1^{-1}(A)=R(Y_1\in A)$.

\begin{mdef}\label{uncertainty_sets} Let
   \begin{equation}
\mathcal{R} =\bigl\{R\in\mathcal{M}_1\mid R\circ Y_{1}^{-1}\in \mathcal{L}_{1}\bigr\}.\label{supported_alt_0}\end{equation}The uncertainty set of \cite{blanch_carassus} (for $t=T=1$) is given as the set $\mathcal{P}$ of all convex combinations of measures in $\mathcal{R}$.
For our purposes we will also allow  $\sigma$-convex-combinations as well, hence we define a slighty extended uncertainty set $\widetilde{\pp}$ as well:
\begin{align}
\mathcal{P} &=\co(\mathcal{R})\label{uncer_set}\\
\widetilde{\pp} &=\sconv(\mathcal{R})=\biggl\{P=\sum_{k=1}^\infty\alpha_kR_k\, \Big|\, R_k\in\mathcal{R}, 0\leq \alpha_k\leq1, \sum_{k=1}^{\infty}\alpha_k=1\biggr\} \label{sigma_pp}\end{align}
\end{mdef}

\begin{remark}\label{rem_supp}
\begin{enumerate}
\item Observe that $\mathcal{R}=\mathcal{R}_1$ and $\mathcal{P}=\mathcal{P}_1$ amount to the analogous sets of \cite{blanch_carassus} and \cite{liebrich-model-uncertainty:reverse-approach} for $t=T=1$. 
\item Observe that the law of $Y_1$ under each measure $R\in\mathcal{R}$ obviously is a Binomial distribution satisfying  $R(Y_1=u_R)={\pi}_R$ and $R(Y_1=d_R)=1-{\pi}_R$ for some $(u_R,d_R,{\pi}_R)\in E_0$. 
By \cite{liebrich-model-uncertainty:reverse-approach}, Proposition~3.10 
applied to the case $t=T=1$, the set $\mathcal{R}$ is a supported alternative of $\mathcal{P}$. The support set $S_R$  of a measure $R\in\mathcal{R}$ with parameters $(u_R,d_R,{\pi}_R)\in E_0$  is given by
\begin{equation}\label{support_R}
   S_R=\{Y_1=u_R\}\cup\{Y_1=d_R\}. 
\end{equation}
Note that, as $Y_1$ is bijective, $S_R=\{Y_1^{-1}(u_R), Y_1^{-1}(d_R)\}$ which, for the identity function  $Y_1(\omega)=\omega$, reduces to $S_R=\{u_R, d_R\}$.
Equation \eqref{support_R} follows from  \cite{liebrich-model-uncertainty:reverse-approach}, Appendix~D, where the proof is given for the technically involved multiperiod case. In the one period case it is obvious. Observe that the concrete choice of $\pi\in[\pi_0,\Pi_0]$ does not change the set $S_R$. 
\end{enumerate}
\end{remark}

\begin{mnot}\label{nota_d_u}
Let $R\in\mathcal{R}$ . 
In the remainder of this subsection and in Section~\ref{binomNA} we will use the notation $u_R$ and $d_R$ for the parameters of the up and down values of $R$, where $u_R\in[u_0,U_0]$, $d_R\in[d_0,D_0]$.
\end{mnot}

We aim at an application of Lemma~\ref{lemma:equivalence_(S)_preHahnloc} and of Section~\ref{sec:hahn_extension} to the one-step binomial model. To this end we would like to find a supported alternative $\widetilde{\rrr}$ for $\widetilde{\pp}$ such that $S_{R_1}\cap S_{R_2}=\varnothing$ for $R_1,R_2\in\widetilde{\rrr}$, $R_1\ne R_2$. 
To achieve this, we will slightly strengthen Assumption~\ref{ass_blan_car_1step} by the following additional restrictions on the possible parameters.

\begin{mass}\label{addass_ud}
    Assume that all condition of Assumption~\ref{ass_blan_car_1step} hold. Assume, moreover, that
    \begin{equation}
      d_0<\min(u_0,D_0), \quad \text{and} \quad \max(u_0,D_0)<U_0. 
    \end{equation}
\end{mass}

Assumption~\ref{addass_ud} has the following consequences. First,  there is real uncertainty in the $u$ as well as the $d$ parameter as $d_0<D_0$ and $u_0<U_0$. Second, the smallest possible value of $d$
is strictly less than the smallest possible value of $u$, and the largest possible value of $d$
is strictly less than the largest possible value of $u$.  For readers with a particular interest in mathematical finance we will comment on no arbitrage considerations in Section~\ref{binomNA} and will see there that this assumption does not restrict the generality in this aspect.

The next result shows that the set $\rrr$ obviously does not have disjoint supports.  As a consequence we have to modify the set to find a strictly disjointly supported alternative, which is then given in Definition~\ref{disjoint_supp} below.

\begin{lemma}\label{not_disjoint}
Under Assumption~\ref{addass_ud} on $E_0$, $\mathcal{R}$ is not a disjointly supported alternative. 
\end{lemma}

\begin{remark}
    Note that this also holds in the general multiperiod case with the supports given as in Appendix~D of \cite{liebrich-model-uncertainty:reverse-approach}. Observe that the supports given in \eqref{support_R} above are exactly these supports in the case of $T=1$. Hence, the fact that the supports, in general,  are not disjoint follows already from the one period case.
\end{remark}

\begin{proof}
 We will choose $R_1$, $R_2$ in $\rrr$ such that $R_1\ne R_2$. Let $(u_1,d_1,\pi_1), (u_2,d_2,\pi_2)\in E_0$, defining $R_1$, $R_2$, respectively.  Recall from Remark~\ref{rem_supp} that $S_{R_i}=\{Y_1=u_i\}\cup\{Y_1=d_i\}$, for $i=1,2$.
  The choice of $\pi_i$, $i=1,2$, is obviously irrelevant for the supports. Even leaving $\pi_i$ aside, under Assumption~\ref{addass_ud} there are several  cases admissible where $\{d_1,u_1\}\cap\{d_2,u_2\}\ne\varnothing$ but still $\{d_1,u_1\}\ne\{d_2,u_2\}$. For example, fix $d_1=a$ with $d_0<1<u_0\leq a\leq  D_0<U_0$, $u_1\in(a, U_0]\subset [u_0,U_0]$. As $u_0\leq a<U_0$ we can choose  $u_2=a$. Choose $d_2\in[d_0,a)$. Then $d_1=u_2=a$ and $d_2<a$, $u_1>a$. Hence, $\{u_1,d_1\}\ne\{u_2,d_2\}$ therefore $R_1\ne R_2$, and since $\{u_1,d_1\}\cap\{u_2,d_2\}=\{a\}$, we get
 $$S_{R_1}\cap S_{R_2}= \{Y_1=a\}\ne\varnothing.$$
 Moreover, observe that $R_1(S_{R_2})\geq R_1(Y_1=d_1=a)= 1-\pi_1>0$ and $R_2(S_{R_1})\geq R_2(Y_1=u_2=a)=\pi_2>0$. Hence, the disjointness property does not even hold in a q.s. way. There are various other possible combinations to  get further non-disjoint supports.
\end{proof}

Our aim is now to find a supported alternative $\widetilde{\rrr}$ with disjoint supports.  We have to represent the supports in a way that makes $\{d,u\}$ unique. We do this by defining a helpful parametrizing function $f$ and the observation that for  certain choices of the set $E_0$, i.e., $u_0\leq D_0$, we have that measures of the form $R\circ Y_1^{-1}=\delta_{a}$, $a\in [u_0,D_0]$, satisfy $R\in\rrr$. 

\begin{mdef}\label{def_disj_supp} Let Assumption~\ref{addass_ud} hold. Let $m_0=\min(u_0,D_0)$ and $M_0=\max(u_0,D_0)$.
  Let $f:[d_0,m_0]\to[M_0,U_0]$ be a continuous strictly decreasing function with $f(d_0)=U_0$ and $f(m_0)=M_0$. Let $\widetilde{\pi}=\frac{\pi_0+\Pi_0}{2}$. Define 
 $$\widetilde{\rrr}_0=\{R\in\rrr \mid R\circ Y_1^{-1}=\widetilde{\pi}\delta_{f(d)}+(1-\widetilde{\pi})\delta_{d}, d\in[d_0,m_0)\}.$$  Define further
 \begin{align*}\widetilde{\rrr}_1 &=\{R\in\rrr \mid R\circ Y_1^{-1}=\delta_{a}, a\in[d_0,D_0]\cap[u_0,U_0]\}\\
 \widetilde{\rrr}_2 &=\{R\in\rrr \mid R\circ Y_1^{-1}=\widetilde{\pi}\delta_{f(d)}+(1-\widetilde{\pi})\delta_{d}, d\in[m_0,D_0]\cap[D_0,M_0)\}.\end{align*}
 The set $\widetilde{\rrr}$ is now given as
 \begin{equation}
     \widetilde{\rrr}=\widetilde{\rrr}_{0}\cup\widetilde{\rrr}_1\cup\widetilde{\rrr}_2.\label{rrr}
 \end{equation}
\end{mdef}

Observe that the choice $\widetilde{\pi}$ for the probability is just an arbitrary fixed choice in $[\pi_0,\Pi_0]$ as the support sets do not depend on the concrete choice of probability $\pi$ (though, note that $0<\widetilde{\pi}<1$ by Assumption \ref{ass_blan_car_1step}). 
Obviously we can choose $f$ as a strictly decreasing line, i.e., $$f(x)=-\frac{U_0-M_0}{m_0-d_0}(x-d_0)+U_0.$$

\begin{theorem}\label{disjoint_supp} Let Assumption~\ref{addass_ud} hold. For all $R_1,R_2\in\widetilde{\rrr}$ of \eqref{rrr} in Definition~\ref{def_disj_supp} with $R_1\ne R_2$ we have that
  $S_{R_1}\cap S_{R_2}=\varnothing$ . Moreover, $\widetilde{\rrr}\subset\rrr$ and it has the same polar sets as $\rrr$. Thus  $\widetilde{\rrr}$ is a supported alternative with strictly disjoint support sets.   
\end{theorem}

We will first characterize how the three subsets of $\widetilde{\rrr}$ look for different possibilities of the form of $E_0$.

\begin{lemma} \label{char_rrr}
If $D_0<u_0$ then $\widetilde{\rrr}_1=\varnothing$ and $\widetilde{\rrr}_2=\{R_0\}$, where $R_0\circ Y_1^{-1}=\widetilde{\pi}\delta_{f(D_0)}+(1-\widetilde{\pi})\delta_{D_0}$. As a consequence, for $D_0<u_0$,
\begin{equation}\widetilde{\rrr}=\{R\in\rrr\mid R\circ Y_1^{-1}=\widetilde{\pi}\delta_{f(d)}+(1-\widetilde{\pi})\delta_{d}, d\in[d_0,D_0]\}.\label{D_0<u_0}\end{equation}

If $D_0\geq u_0$, then $\widetilde{\rrr}_1=\{R\in\rrr\mid R\circ Y_1^{-1}=\delta_{a}, a\in[u_0,D_0]\}\subset\widetilde{\rrr}$ and  $\widetilde{\rrr}_2=\varnothing$.
 \end{lemma}

\begin{proof} Let $D_0<u_0$. 
Then $[d_0,D_0]\cap[u_0,U_0]=\varnothing$ and hence, $\widetilde{\rrr}_1=\varnothing$. Moreover,  $m_0=D_0$ and $M_0=u_0$, hence $[m_0,D_0]=\{D_0\}$ and $[D_0,M_0)=[D_0,u_0)$,  therefore
$[m_0,D_0]\cap [D_0,M_0)=\{D_0\}$ and  $\widetilde{\rrr}_2=\{R_0\}$. In the definition of the set $\widetilde{\rrr}_0$ only $d\in[d_0,m_0)$ appear, but $m_0=D_0$, hence the measure $R_0$ now adds the right boundary of the interval $[d_0,m_0]=[d_0,D_0]$, and $\widetilde{\rrr}$ is given by \eqref{D_0<u_0}.

Let now $D_0\geq u_0$. Then $m_0=u_0$, $M_0=D_0$. Thus, in this case $[D_0, M_0)=[D_0,D_0)=\varnothing$ and hence $\widetilde{\rrr}_2=\varnothing$. Concerning $\widetilde{\rrr}_1$, note that by Assumption~\ref{addass_ud} it holds that $d_0<u_0\leq D_0<U_0$, hence
$[d_0,D_0]\cap[u_0,U_0]=[u_0,D_0]\neq\varnothing$. 
Let us now show that for all $a\in[u_0,D_0]$ we have that $R_a\in\rrr$ where  $R_a\circ Y_1^{-1}=\delta_{a}$. Indeed, let $\pi\in[\pi_0,\Pi_0]$ arbitrary. As $a\in[d_0,D_0]\cap[u_0,U_0]$ we can choose $a=d_R=u_R$ and observe that $R_a= R\in\rrr$ with $(u_R, d_R, \pi)\in E_0$ where $R$ is given by $R\circ Y_1^{-1}=\pi \delta_{u_R}+(1-\pi)\delta_{d_R}=\delta_{a}$.
\end{proof}
\begin{proof}[Proof (Theorem~\ref{disjoint_supp})]
 By definition and by Lemma~\ref{char_rrr} it is obvious that
$\widetilde{\rrr}\subset \rrr$.  Let us first show that the supports of measures in $\widetilde{\rrr}$ are disjoint.
\newline
\emph{Case~1:} $D_0<u_0$.
By Lemma~\ref{char_rrr} we know that $\widetilde{\rrr}=\{R\in\rrr \mid R\circ Y_1^{-1}=\widetilde{\pi}\delta_{f(d)}+(1-\widetilde{\pi})\delta_{d}, d\in[d_0,D_0]\}$. For each $d\in[d_0,D_0]$ we have that $f(d)=u\in[u_0,U_0]$. Suppose $R_1\ne R_2$ with $R_1,R_2\in \widetilde{\rrr}.$ Let $\{d_i,u_i\}$ with $d_i\in[d_0,D_0]$ and $u_i=f(d_i)$ be the parameters of $R_i$, $i=1,2$, respectively. As $f$ is bijective we have that $d_1=d_2$ if and only if $u_1=u_2$. By assumption, $\{d_1,u_1\}\neq\{d_2,u_2\}$ as $R_1\neq R_2$, hence $d_1\neq d_2$ and $u_1\neq u_2$ has to hold. Moreover $d_i\neq u_j$, for $i\neq j$ with $i,j\in\{1,2\}$. Indeed $u_i=f(d_i)\in[u_0,U_0]$, $i=1,2$, and $d_i\in[d_0,D_0]$ and, as $D_0<u_0$, obviously
 $[u_0,U_0]\cap [d_0,D_0]=\varnothing$. Hence, whenever $R_1\neq  R_2$ it follows that $\{d_1,u_1\}\cap\{d_2,u_2\}=\varnothing$. Therefore $S_{R_1}\cap S_{R_2}=\varnothing$ as, by Remark~\ref{rem_supp}, $S_{R_i}=\{Y_1=u_i\}\cup\{Y_1=d_i\}$, for $i=1,2$. 
\newline
\emph{ Case~2:} $D_0\geq u_0$. By Lemma~\ref{char_rrr} we know that $\widetilde{\rrr}=\widetilde{\rrr}_0\cup \widetilde{\rrr}_1$.

For $R\in \widetilde{\rrr}_0$, we have that $d_0\leq d<m_0\leq D_0$ and $u_0\leq M_0< f(d)=u\leq U_0$.
By an analogous argument as in Case~1 we see that for measures $R_1,R_2\in \widetilde{\rrr}_0$, we have that $S_{Q_1}\cap S_{Q_2}=\varnothing$. 
Indeed, in Case~2, $m_0=u_0$, $M_0=D_0$, and hence $f: [d_0,u_0]\to[D_0,U_0]$. 
Let now $R_1\ne R_2\in \widetilde{\rrr}_0$ for $d_1, d_2\in [d_0, u_0)$. 
Then, as  $f$ is bijective, we have that $d_1\neq d_2$ and $u_1=f(d_1)\neq f(d_2)=u_2$ as before. 
Assume, e.g., $d_1=u_2=f(d_2)=:a$ would hold. We have $a=d_1\in[d_0,u_0)$  and $a=f(d_2)\in(D_0,U_0]$ by the definition of $\widetilde{\rrr}_0$. 
This is not possible as in Case~2 $[d_0,u_0)\cap (D_0,U_0]=\varnothing$. Therefore  $\{d_1,u_1\}\cap\{d_2,u_2\}=\varnothing$. and thus $S_{R_1}\cap S_{R_2}=\varnothing$. 

Consider now two measures $R_1\neq R_2$ with $R_1, R_2\in \widetilde{\rrr}_1$.
It is obvious that the support of $R_i$ with  
 $R_i\circ Y_1 ^{-1}=\delta_{a_i}$, $a_i\in[d_0,D_0]\cap[u_0,U_0]=[u_0,D_0]$,  is of the form $S_{R_i}=\{Y_1={a_i}\}$. Hence, it has to hold that $a_1\neq a_2$ and so $S_{R_1}\cap S_{R_2}=\varnothing$. 
 
 It remains to show that $R_1\in\widetilde{\rrr}_0$ and $R_2\in \widetilde{\rrr}_1$  have disjoint supports. Let $\{d,u\}$, with $d\in[d_0,u_0)$ and $u=f(d)$, be the parameters defining $R_1$ and $a\in[u_0,D_0]$ be the parameter defining $R_2$. Obviously $d\neq a$ as $d<u_0$ and $a\geq u_0$. By definition $u=f(d)\in (D_0, U_0]$, hence $u\neq a$. Therefore $S_{R_1}\cap S_{R_2}=\varnothing$. Hence, we have seen that $\widetilde{\rrr}$ is indeed disjointly supported.
\newline\newline
Now, we will prove that $\widetilde{\rrr}$  and $\rrr$ have the same polar sets. Indeed, as $\widetilde{\rrr}\subset\rrr$ it is obvious that for $A\in\mathcal{F}$ with $\sup_{R\in\rrr}R(A)=0$ it follows that $\sup_{R\in\widetilde{\rrr}}R(A)=0$. Suppose now that there exists $Q\in\rrr\setminus\widetilde{\rrr}$ and $A\in\mathcal{F}$ with $\sup_{R\in\widetilde{\rrr}}R(A)=0$ but 
 \begin{equation}\label{badsetmeas} Q(A)>0.\end{equation} 
 Let $S_Q=\{Y_1=u_Q\}\cup\{Y_1=d_Q\}$ with $d_Q\in[d_0,D_0]$ and $u_Q\in[u_0,U_0]$. 
 
 Assume first $D_0<u_0$. 
 Then $u_Q\neq f(d_Q)$, because otherwise $Q\in\widetilde{\rrr}$ which was excluded. 
 Define ${Q}_1, {Q}_2\in\widetilde{\rrr}$ with the corresponding parameters as follows: $d_1=d_Q$, $u_1=f(d_Q)$, $u_2=u_Q$, $d_2=f^{-1}(u_Q)$. By \eqref{D_0<u_0} of Lemma~\ref{char_rrr}, $Q_1, Q_2\in\widetilde{\rrr}$. 
 Moreover, it holds that
 \begin{equation}S_Q=\{Y_1=d_Q\}\cup\{Y_1=u_Q\}\subseteq S_{Q_1}\cup S_{Q_2}\label{bigger_supp}\end{equation}
 As $Q_1, Q_2\in\widetilde{\rrr}$, by assumption $Q_i(A)=0$, $i=1,2$. But then, obviously, $Q_i(A)=Q_i(A\cap S_{Q_i})=0$, for $i=1,2$. 
 Recall that $\widetilde{\rrr}\subset{\rrr}$ and each $R\in\rrr$ is supported, therefore $Q_i$, $i=1,2$ are supported. 
 Therefore,  $\sup_{R\in\rrr}R(A\cap S_{Q_i})=0$, for $i=1,2$. 
 From this together with \eqref{bigger_supp} it follows that
 $$Q(A)=Q(A\cap S_Q)\leq Q(A\cap S_{Q_1})+Q(A\cap S_{Q_2})=0.$$ 
 Thus we get a contradiction to \eqref{badsetmeas} and $A$ is a polar set for $\rrr$ in the case $D_0<u_0$.

 Assume now that $D_0\geq u_0$. 
 We distinguish four cases.
 \begin{enumerate}
   \item[(i)] Assume first that $d_Q\in[d_0,u_0)$ and $u_Q\in(D_0,U_0]$. It holds that $u_Q\neq f(d_Q)$ (and thus $d_Q\neq f^{-1}(u_Q)$) because otherwise $Q\in\widetilde{\rrr}_0\subset\widetilde{\rrr}$ which was excluded. Similarily as before choose $Q_1, Q_2\in \widetilde{\rrr}_0$ with $d_1=d_Q$, $u_1=f(d_Q)$, $d_2=f^{-1}(u_Q)$, $u_2=u_Q$. In this case $S_Q\subset S_{Q_1}\cup S_{Q_2}$ and as above this implies that $Q(A)=0$. 
   \item[(ii)]  Assume that $d_Q\in[u_0,D_0]$ and  $u_Q\in(D_0,U_0]$. 
     Then choose $Q_1\in \widetilde{\rrr}_1$ with $a=d_Q$. Choose  $Q_2$ as in (i). 
     Again, we get $S_Q\subset S_{Q_1}\cup S_{Q_2}$.
     As $Q_1, Q_2\in \widetilde{\rrr}\subset\rrr$ and therefore are supported, this implies again that $Q(A)=0$.
   \item[(iii)]  Assume that $d\in[d_0,u_0)$ and  $u_Q\in[u_0,D_0]$. Choose $Q_1$ as in (i) and choose $Q_2\in \widetilde{\rrr}_1$ with $a=u_Q$. Then $S_Q\subset S_{Q_1}\cup S_{Q_2}$ and as $Q_1, Q_2\in \widetilde{\rrr}\subset\rrr$, they are supported, which implies again that $Q(A)=0$.
   \item[(iv)] Finally assume that $d_Q, u_Q\in [u_0,D_0]$. Choose $Q_1$ as in (ii) and $Q_2$ as in (iii).  $u_Q=d_Q$ is not possible because otherwise $Q\in\widetilde{\rrr}_1\subset{\widetilde{\rrr}}$ which was excluded, hence $Q_1\neq Q_2$. Moreover, then $S_Q= S_{Q_1}\cup S_{Q_2}$.  Again, as $Q_1, Q_2\in \widetilde{\rrr}\subset\rrr$   are supported this implies  that $Q(A)=0$ as above.
 \end{enumerate}
   Hence, we get a contradiction to \eqref{badsetmeas} in all four cases and $A$ is a polar set for $\rrr$ in the case $D_0\geq u_0$.
 \end{proof}

A carefully reading of the proof of Theorem~\ref{disjoint_supp} implies the following corollary.

\begin{corollary}\label{abscontilde}
For each  $Q\in\rrr\setminus\widetilde{\rrr}$ there exist $Q_1, Q_2\in\widetilde{\rrr}$ such that  $S_Q\subseteq S_{Q_1}\cup S_{Q_2}$. 
In particular, this implies that $Q\ll\frac{Q_1+Q_2}{2}$. 
\end{corollary}

Theorem \ref{disjoint_supp} together with Theorem \ref{thm:duality-wo-hahn-loc} allow us to construct a robust, one-step binomial model, which is Hahn-localizable. We follow the construction given in Definition~\ref{def_disj_supp}. 
Let $\Omega = (0,\infty)$, $\mathcal{F} = \mathcal{B}(\Omega)$ and $\widetilde{ \pp }$ as in \eqref{sigma_pp}. 
The model $(\Omega, \mathcal{F}, \widetilde{ \pp })$ is pre-Hahn-localizable, with strictly disjointly supported localization $\widetilde{ \mathcal{R} }$. 
Clearly, $\widetilde{ \mathcal{R} }$ satisfies Property~1 of Definition \ref{def:hahn-loc}. 
Moreover, since $\widetilde{ \mathcal{R}} \subseteq \mathcal{R} \subset\widetilde{ \mathcal{P}}$, we have $\widetilde{ \mathcal{R}} \lll \mathcal{R} \lll\widetilde{ \mathcal{P}}$.
Hence, we only need to show that $\mathcal{R}\lll \sconv(\widetilde{ \mathcal{R} })$. 
Fix $Q \in \mathcal{R}$. 
Then either $Q\in \widetilde{ \mathcal{R}}$ and there is nothing to do, or $Q\in\rrr\setminus\widetilde{ \mathcal{R}}$. 
From Corollary~\ref{abscontilde}, we know that there are ${ Q }_{1}, { Q }_{2}\in \widetilde{ \mathcal{R} }$ such that $Q \ll \frac{1}{2}( { Q }_{1} +{ Q }_{2})$.
Hence, as this implies $\mathcal{R}\lll \co(\widetilde{ \mathcal{R}})$ and as $\widetilde{\mathcal{P}}= \co_{\sigma}(\rrr)$, clearly $\widetilde{\mathcal{P}}\lll  \co_{\sigma}(\widetilde{ \mathcal{R}})$, and Property~2 of Definition~\ref{def:hahn-loc} is satisfied. 

Now, denote with $\widetilde{\mathcal{F}} = \mathcal{H}_{\mathcal{F}}^{\widetilde{ \mathcal{R} }}$ the Hahn-extension of $\mathcal{F}$ for the localization $\widetilde{ \mathcal{R} }$, and denote with $\widetilde{ \pp }_{1}$ the extension of $\widetilde{ \pp }$ to $\f$ (as in Section \ref{sec:hahn_extension}). 
This implies that $(\Omega, \widetilde{\mathcal{F}},\widetilde{ \pp }_{1})$ is Hahn-localizable, and $\LL(\widetilde{ \pp }_{1})$ thus satisfies the robust $L^{1}$-$L^{\infty}$ duality.  

\begin{corollary}\label{Hahnlocaliz_binom}
  Given Assumption \ref{addass_ud}, the robust one step binomial model based on $(\Omega, \widetilde{\mathcal{F}}, \widetilde{ \pp }_{1})$ is Hahn-localizable. 
\end{corollary}

\subsection{A non-example and the dual of $\mathrm{ca}(\pp)$ in this case}
Finally, we give an example of a statistical model, which is not pre-Hahn-localizable. 
Let $\Omega = (0,1)$, $\mathcal{F} = \mathcal{B}(\Omega)$ the Borel $\sigma$-algebra on $\Omega$, and $\lambda$ the Lebesgue measure. 
We consider the set $\mathcal{P}$ of all probability measures $P$ on $\Omega$ of the form
\begin{equation*}
P= \frac{1}{2}\lambda + \frac{1}{2}\sum_{k=1}^{\infty} \alpha_{k}\delta_{x_{k}},
\end{equation*}
for some $x_{k}\in (0,1)$, and $\alpha_{k}\geq 0$ with $\sum_{k=1}^{\infty} \alpha_{k} = 1$. 
Assume that $\mathcal{P}$ is pre-Hahn-localizable, with localization $\mathcal{Q}$, and supports $\{S_{Q}\mid Q\in \mathcal{Q}\}$.
Recall that, by Definition~\ref{P-positiveSQ} and Remark~\ref{rem:R(Q)-at-most-countable}, for every $P\in \pp$, the set
 \begin{equation*}
   \mathcal{Q}(P) = \{ Q\in\mathcal{Q}\mid P(S_{Q})>0\}
 \end{equation*}
 is at most countable, and $P(\bigcup_{ Q\in \mathcal{Q}(P)} S_Q) = 1$. 
 As a consequence, at least one of the support sets $S_{Q}$ must be uncountable.
 We denote this support with $S$, and write $Q$ for the corresponding dominating  measure  in $\mathcal{Q}$ supported on $S$.
 In particular, $Q$ must dominate every $P$ of the form $P^{x} = \frac{1}{2}\lambda + \frac{1}{2} \delta_{x}$, for $x\in S$.
 This implies $Q(\{x\})>0$ for every $x\in S$.   Indeed, suppose there would exists $x\in S$ with $Q(\{x\})=0$. By Property~2 of Definition~\ref{def:hahn-loc} there exists $R\in \co_{\sigma}(\mathcal{Q})$  with $P^x\ll R$. 
   Note that for all $\tilde{Q}\in\mathcal{Q}$ with $\tilde{Q}\ne Q$ we have that $\tilde{Q}(\{x\})\leq \tilde{Q}(S)=0$ as $S$ is the support set of $Q$.
   This implies $Q(\{x\})>0$, for otherwise we would get $R(\{x\})=0$ and $P^x\not\ll R$ (as
 clearly $P^x(\{x\})>0$).
But then $Q(\{x\})>0$ for every $x\in S$, which is a contradiction to $Q(\Omega)<\infty$, since $S$ is uncountable. 
 However, our theory still provides insights in such cases. 

\begin{corollary}\label{cor:abs-cont+hahn}
Let $(\Omega, \f)$ be a measurable space, and $\pp$ a convex family of probability measures on $\f$. Assume that there is a $\sigma$-finite measure $\lambda$ on $\f$, and $\q\subseteq \pp$ such that
\begin{enumerate}
 \item there are pairwise disjoint sets $S_{Q}\in \f$ with $R(S_{Q}) = \delta_{RQ}$ for all $Q,R\in \q$,
 \item $\lambda$ is pairwise singular to any $Q\in \q$, and
\item every measure $P\in \pp$ can be written as
  \begin{equation*}
    P = P_{\lambda} + P_{\q},
  \end{equation*}
  where $P_{\lambda}\ll \lambda$ and $P_{\q}\ll \sum_{k=1}^{\infty}\lambda_{k} Q_{k}$ with $\lambda_{k}\geq 0$, $\sum_{k=1}^{\infty}\lambda_{k}=1$, and $Q_{k}\in \q$. 
\end{enumerate}
Then every $\mu \in \mathrm{ca}(\pp)$ can be written as $\mu  = \mu_{\lambda} + \mu_{\q}$ with $\mu_{\lambda}\ll \lambda$ and $\mu_{\q}\lll \sconv(\q)$, and every element $x'$ of the dual space of $\mathrm{ca}(\pp)$ is of the form
\begin{equation*}
  x'(\mu) = \int f_{\lambda}d \mu_{\lambda} + \int f_{\q} d \mu_{\q}^{\q},
\end{equation*}
where $f_{\lambda}\in L^{\infty}(\f,\lambda)$, $f_{\q}\in \mathbb{L}^{\infty}(\hef, \sconv(\q)^{\q})$.
\end{corollary}

\begin{proof}
We first show that any $\mu\in \mathrm{ca}(\pp)$ can be written as
\begin{equation}\label{eq:theorem:lambda+Q:decomposition-of-signed-measures}
  \mu  = \mu_{\lambda} + \mu_{\q},
\end{equation}
with $\mu_{\lambda}\ll \lambda$ and $\mu_{\q}\ll Q^{*}$ for some $Q^{*}\in \sconv(\q)$. 
Let $\mu\in \mathrm{ca}(\pp)$ and $P\in \pp$ with $\mu \ll P$. 
Let $P = P_{\lambda} + P_{\q}$ be the decomposition of $P$ according to the assumptions of the theorem, with $Q^{*}\in \sconv(\q)$ such that $P_{\q}\ll Q^{*}$. 
Then $P_{\lambda}$ and $P_{\q}$ are singular, and there are disjoint sets $\Omega_{\lambda}, \Omega_{\q}\in \f$ supporting $P_{\lambda}$ and $P_{\q}$ respectively. 
We set $\mu_{\lambda}(F) = \mu(F\cap \Omega_{\lambda})$ and $\mu_{\q}(F) = \mu(F\cap \Omega_{\q})$, and claim that $\mu_{\lambda}\ll \lambda$ and $\mu_{\q}\ll Q^{*}$. 
We will only show that $\mu_{\lambda}\ll \lambda$, since the proof for $\mu_{\q} \ll Q^{*}$ can be done in the same manner. 
Let $F\in \f$ with $\lambda(F) = 0$. 
Then $P_{\lambda}(F) = P(F\cap \Omega_{\lambda}) = 0$, and hence $\mu_{\lambda}(F) = \mu(F\cap \Omega_{\lambda}) = 0$. 

Now, let $x'\in \mathrm{ca}(\pp)'$, and note that $x'(\mu) = x'(\mu_{\lambda}) + x'(\mu_{\q})$, where $\mu = \mu_{\lambda} + \mu_{\q}$ as in \eqref{eq:theorem:lambda+Q:decomposition-of-signed-measures}. 
Since $\mathrm{ca}(\f,\lambda)$ is a (closed) subspace of $\mathrm{ca}(\f,\pp)$, $x'|_{\mathrm{ca}(\f,\lambda)}\in \mathrm{ca}(\f,\lambda)'$.
Since $\lambda$ is $\sigma$-finite, we get that $\mathrm{ca}(\lambda)$ is isometrically isomorphic to $L^{1}(\lambda)$.
As such, $L^{1}(\f,\lambda)' \simeq L^{\infty}(\f,\lambda)$, and there is a $f_{\lambda}\in L^{\infty}(\f,\lambda)$ such that 
\begin{equation*}
  x'(\mu_{\lambda}) = \int f_{\lambda}d \mu_{\lambda}.
\end{equation*}

On the other hand, $\mathrm{ca}(\f,\sconv(\q))$ is also a (closed) subspace of $\mathrm{ca}(\f,\pp)$, and $\mathrm{ca}(\f,\sconv(\q))$ does satisfy the assumptions of Theorem \ref{thm:duality-wo-hahn-loc} (with $\q\subseteq \sconv(\q)$), i.e.
\begin{equation*}
  \mathrm{ca}(\f,\sconv(\q))' \simeq \mathbb{L}^{\infty}(\f^{\sconv(\q)}, \sconv(\q)^{\q}).
\end{equation*}
Again, $x'|_{\mathrm{ca}(\f,\sconv(\q))}\in \mathrm{ca}(\f,\sconv(\q))'$ and there is a $f_{\q}\in  \mathbb{L}^{\infty}(\f^{\sconv(\q)}, \sconv(\q)^{\q})$ such that 
\begin{equation*}
  x'(\mu_{\q}) = \int f_{\q}d \mu_{\q}.
\end{equation*}
Indeed, looking at Definition \ref{def:supported-measures}, it is easy to see that  $\f^{\sconv(\q)} = \hef$. 
\end{proof}

\section{Characterization of strictly unbiased hypothesis}\label{sec:kraft_testable_hyp}
Let $(\Omega, \mathcal{F})$ be a measurable space, and let the null and alternative hypothesis $\mathcal{H}_{0}$ and $\mathcal{H}_{1}$ be two families of probability measures on $(\Omega, \mathcal{F})$. 
A test is a measurable function $\phi: \Omega \to \mathbb{R}$, such that $0 \leq \phi \leq 1$ $\mathcal{H}_{0}\cup \mathcal{H}_{1}$-q.s. 
The worst case type-I error and power of a test $\phi$ are given by $\sup_{H\in \mathcal{H}_{0}} \mathbb{E}_{H}( \phi )$ and $\inf_{H\in \mathcal{H}_{1}} \mathbb{E}_{H}( \phi )$. 
The risk of a test $\phi$ is given by the sum of the worst type-I and type-II errors, i.e., by
\begin{equation*}
  R(\phi) = R(\phi, \mathcal{H}_{0}, \mathcal{H}_{1}) = \sup_{\mu \in \mathcal{H}_{0}} \mathbb{E}_{\mu}( \phi ) +  \sup_{\nu \in \mathcal{H}_{1}} \mathbb{E}_{\nu}(1- \phi ).
\end{equation*}
A test $\phi$ is called strictly unbiased, if 
\begin{equation*}
  \sup_{H\in \mathcal{H}_{0}} \mathbb{E}_{H}( \phi ) < \inf_{H\in \mathcal{H}_{1}} \mathbb{E}_{H}( \phi ),
\end{equation*}
or equivalently, $R(\phi) <1$. 
Some argue \cite{larsson2026completecharacterizationtestablehypotheses} that strict unbiasedness is a ``reasonable minimum requirement'' for a test. 
In the case where $\mathcal{H}_{0}$ and $\mathcal{H}_{1}$ have a common dominating $\sigma$-finite measure, Kraft \cite{kraft} gave a full characterization of unbiased tests in terms of the total variation distance between $\co(\mathcal{H}_{0})$ and $\co(\mathcal{H}_{1})$, given by
\begin{equation*}
  \mathrm{d}_{TV}(\mu, \nu) = \frac{1}{2}\|\mu - \nu\|_{TV}. 
\end{equation*}
\begin{theorem}[\cite{kraft}]\label{thm:kraft}
If $\mathcal{H}_{0}$ and $\mathcal{H}_{1}$ have a common $\sigma$-finite dominating measure $\rho$, then for every $\varepsilon>0$
\begin{equation*}
  \exists \text{ test }\phi\in L^{\infty}(\rho): \inf_{\nu\in \mathcal{H}_{1}}\mathbb{E}_{\nu}( \phi ) > \sup_{\mu\in \mathcal{H}_{0}} \mathbb{E}_{\mu}( \phi ) + \varepsilon \iff \mathrm{d}_{TV}(\co(\mathcal{H}_{0}), \co(\mathcal{H}_{1})) > \varepsilon. 
\end{equation*}
In fact, we have 
\begin{equation*}
  \inf_{\substack{\phi \in L^{\infty}(\rho)\\ 0 \leq \phi \leq 1 \rho\text{-a.s.}}} R(\phi) = 1- \mathrm{d}_{TV}(\co(\mathcal{H}_{0}), \co(\mathcal{H}_{1})). 
\end{equation*}
\end{theorem}
In the case of non-dominated families of hypothesis, there exist two main approaches.
Le Cam \cite{lecam} suggested so-called ``generalized tests''.
A generalized test is an element $\phi\in \mathrm{ca}'$, such that $0\leq \phi(\mu)\leq 1$ for every $\mu\in \mathrm{ca}_{1}$, where $\mathrm{ca}$ denotes the set of all finite signed measures on $(\Omega, \mathcal{F})$ and $\mathrm{ca}_{1} \subseteq \mathrm{ca}$ denotes the set of probability measures. 
The existence of a strictly unbiased generalized test has an analogous characterization as the one provided in Theorem \ref{thm:kraft}. 
\begin{theorem}[\cite{lecam}]
  Let $\mathcal{H}_{0}$ and $\mathcal{H}_{1}$ be arbitrary families of probability measures. Then for every $\varepsilon>0$
  \begin{equation*}
    \exists \text{ generalized test }\phi: \inf_{\nu\in \mathcal{H}_{1}}\mathbb{E}_{\nu}( \phi ) > \sup_{\mu\in \mathcal{H}_{0}} \mathbb{E}_{\mu}( \phi ) + \varepsilon \iff \mathrm{d}_{TV}(\co(\mathcal{H}_{0}), \co(\mathcal{H}_{1})) > \varepsilon.
\end{equation*}
In fact, we have 
\begin{equation*}
  \inf_{\substack{\phi \in \mathrm{ca}'\\ \forall \mu\in \mathrm{ca}_{1}: 0 \leq \phi(\mu)\leq 1}} R(\phi) = 1- \mathrm{d}_{TV}(\co(\mathcal{H}_{0}), \co(\mathcal{H}_{1})). 
\end{equation*}
\end{theorem}
As has already been pointed out by Le Cam, a generalized test may (in general) not correspond to a measurable function, and it is a priori unclear how a generalized test is to be computed in practice. 

The second approach to non-dominated families of hypothesis is due to \cite{larsson2026completecharacterizationtestablehypotheses}. 
The results in \cite{larsson2026completecharacterizationtestablehypotheses} avoid generalized tests.
As a price, their characterizations are formulated in terms of non-$\sigma$-additive sets functions. 

Let $\mathrm{ba}$ be the set of finitely additive set functions on $(\Omega, \mathcal{F})$, and denote with $\Phi_{0} = \{f: \Omega\to [0,1]\mid f \text{ measurable}\}$. 
The space $\mathrm{ba}$ is the dual space of the space of bounded, real-valued, $\mathcal{F}$-measurable functions, equipped with the supremum norm. We denote with $\overline{\co}^{*}(A)$ the closure of a set $A \subseteq \mathrm{ba}$ in the weak* topology. 
In \cite{larsson2026completecharacterizationtestablehypotheses}, the following theorem is established, as a robust generalization of Theorem \ref{thm:kraft}. 
\begin{theorem}\label{thm:larsson}
  Let $\mathcal{H}_{0}$ and $\mathcal{H}_{1}$ be arbitrary sets of probability measures, then for every $\varepsilon>0$
  \begin{equation*}
    \exists \text{ test }\phi \in \Phi_{0}: \inf_{\nu\in \mathcal{H}_{1}}\mathbb{E}_{\nu}( \phi ) > \sup_{\mu\in \mathcal{H}_{0}} \mathbb{E}_{\mu}( \phi ) + \varepsilon \iff \mathrm{d}_{TV}(\overline{\co}^{*}(\mathcal{H}_{0}), \overline{\co}^{*}(\mathcal{H}_{1})) > \varepsilon. 
\end{equation*}
In fact, we have 
\begin{equation*}
  \inf_{\substack{\phi \text{ meas.}\\ 0 \leq \phi \leq 1}} R(\phi) = 1- \mathrm{d}_{TV}(\overline{\co}^{*}(\mathcal{H}_{0}), \overline{\co}^{*}(\mathcal{H}_{1})). 
\end{equation*}
\end{theorem}

If one is aiming for the utmost generality, i.e., arbitrary sets $\mathcal{H}_{0}$ and $\mathcal{H}_{1}$, then the examples in \cite{larsson2026completecharacterizationtestablehypotheses} suggest that passing to non-$\sigma$-additive set functions is somewhat unavoidable. 
However, using the theory developed in this paper, we can show that for a large class of non-dominated models, neither generalized tests, nor non-$\sigma$-additive measures are necessary.  
As in \cite{larsson2026completecharacterizationtestablehypotheses}, we apply a minimax theorem. 
\begin{theorem}[Sion's 
Minimax Theorem, \cite{sion}]\label{thm:sion}
  Let $T$ and $S$ be topological vector spaces, $X \subseteq T$ convex and $Y \subseteq S$ convex and compact. Assume that $f: X \times Y \to \mathbb{R}$ satisfies
  \begin{enumerate}
  \item $f(\cdot, y)$ is upper semicontinuous and quasi-concave on $X$, for every fixed $y\in Y$, and 
  \item $f(x, \cdot)$ is lower semicontinuous and quasi-convex on $Y$, for every fixed $x\in X$.
  \end{enumerate}
  Then we have
  \begin{equation*}
 \sup _{x\in X}\newinf_{y\in Y}f(x,y)=\newinf _{y\in Y}\sup _{x\in X}f(x,y).
  \end{equation*}

\end{theorem}
However, there is an important difference: in \cite{larsson2026completecharacterizationtestablehypotheses} the weak*-compactness is imposed on the measures $\co(\mathcal{H}_{0})$ and $\co(\mathcal{H}_{1})$ by passing to the weak*-closure in $\mathrm{ba}$. 
We exploit that the set of tests itself is already weak*-compact, if $\mathrm{ca}(\pp)' \simeq \LL(\pp)$, resulting in the following theorem. 

\begin{theorem}
Let $\pp$ be a pre-Hahn-localizable family of probability measures on $(\Omega, \mathcal{F})$ with strictly disjoint support sets, and assume that $\mathcal{H}_{0},\mathcal{H}_{1} \lll \pp$. Then for every $\varepsilon>0$
\begin{equation*}
  \exists \text{ test }\phi \in \LL(\mathcal{H}_{\mathcal{F}}, \pp): \inf_{\nu\in \mathcal{H}_{1}}\mathbb{E}_{\nu}( \phi ) > \sup_{\mu\in \mathcal{H}_{0}} \mathbb{E}_{\mu}( \phi ) + \varepsilon \iff \mathrm{d}_{TV}(\co(\mathcal{H}_{0}), \co(\mathcal{H}_{1})) > \varepsilon. 
\end{equation*}

\end{theorem}
\begin{proof}
  Using Lemma \ref{lemma:hahn-extension:absolute-continuity}, we can uniquely extend $\mathcal{H}_{0}$ and $\mathcal{H}_{1}$ 
    to the Hahn-extension $\mathcal{H}_{\mathcal{F}}$ of $\mathcal{F}$, while preserving the property that $\mathcal{H}_{0}, \mathcal{H}_{1}\lll \pp$. 
  Hence, we may assume without loss of generality, that $(\Omega, \mathcal{F}, \pp)$ is Hahn-localizable. 
  We denote with 
  \begin{equation*}
    \Phi = \{f\in \LL(\pp)\mid 0 \leq f \leq 1, \pp\text{-q.s.}\}
  \end{equation*}
  the set of tests. 
  Note that $\Phi$ is the intersection of the unit ball in $\LL(\pp)$ with the set
  \begin{equation*}
    \Phi_{+} = \{f\in \LL(\pp)\mid f \geq 0, \pp\text{-q.s.}\}. 
  \end{equation*}
  Since $\mathrm{ca}(\pp)'\simeq \LL(\pp)$ by Theorem \ref{thm:duality}, the unit ball is weak*-compact.
  Moreover, $\Phi_{+}$ is weak*-closed, and hence $\Phi$ is weak*-compact as the intersection of a closed with a compact set. 
  Recall that the total variation distance between two probability measurs $\mu, \nu\in \mathrm{ca}(\pp)$ is given by
  \begin{equation*}
    \mathrm{d}_{TV}(\mu,\nu) = \sup_{\phi\in \Phi}\mathbb{E}_{\nu}( \phi ) - \mathbb{E}_{\mu}( \phi ),
  \end{equation*}
  and hence
  \begin{equation*}
    \begin{split}
      \mathrm{d}_{TV}(\co(\mathcal{H}_{0}), \co(\mathcal{H}_{1})) & = \inf_{\substack{\mu \in \co(\mathcal{H}_{0})\\ \nu\in \co(\mathcal{H}_{1})}}\sup_{\phi\in \Phi}\mathbb{E}_{\nu}( \phi ) - \mathbb{E}_{\mu}( \phi ) \\
                                                                  & = - \sup_{\substack{\mu \in \co(\mathcal{H}_{0})\\ \nu\in \co(\mathcal{H}_{1})}}\inf_{\phi\in \Phi}\mathbb{E}_{\mu}( \phi ) - \mathbb{E}_{\nu}( \phi ),
    \end{split}
  \end{equation*}
  where the last manipulation took place to transform the problem into the setting of Sion's Minimax Theorem. 
 We apply Theorem \ref{thm:sion} with the following data: $T= \mathrm{ca}(\mathcal{H}_{0})\times \mathrm{ca}(\mathcal{H}_{1})$ (equipped with the product of the norm topology), $Y = \LL(\pp)$ (equipped with the weak* topology), $X = \co(\mathcal{H}_{0})\times \co(\mathcal{H}_{1})$, $Y = \Phi$, and $f((\mu,\nu), \phi) = \mathbb{E}_{\mu}( \phi ) - \mathbb{E}_{\nu}( \phi )$, which is linear and continuous in both its arguments.
  This yield
  \begin{equation*}
    \begin{split}
    \mathrm{d}_{TV}(\co(\mathcal{H}_{0}), \co(\mathcal{H}_{1})) & = - \inf_{\phi\in \Phi} \sup_{\substack{\mu \in \co(\mathcal{H}_{0})\\ \nu\in \co(\mathcal{H}_{1})}}\mathbb{E}_{\mu}( \phi ) - \mathbb{E}_{\nu}( \phi )\\
                                                                & = 1 - \inf_{\phi\in \Phi} \sup_{\substack{\mu \in \co(\mathcal{H}_{0})\\ \nu\in \co(\mathcal{H}_{1})}}\mathbb{E}_{\mu}( \phi ) - \mathbb{E}_{\nu}(1- \phi )\\
                                                                & = 1 - \inf_{\phi\in \Phi} \sup_{\substack{\mu \in \mathcal{H}_{0}\\ \nu\in \mathcal{H}_{1}}}\mathbb{E}_{\mu}( \phi ) - \mathbb{E}_{\nu}(1- \phi ) = 1 - \inf_{\phi\in \Phi}R(\phi). \\
    \end{split}
  \end{equation*}
  From this, the theorem follows immediately.
\end{proof}

The state of the literature is summerized in the following table: 
\begin{center}
  \begin{tabular}{cccc}
  \toprule
  Reference &Domination & Tests & Characterization\\
  \midrule
  Kraft \cite{kraft} & $\sigma$-finite measure $\rho$ & $0 \leq \phi \leq 1$ $\rho$-a.s. & $\co(\mathcal{H}_{0})$ vs. $\co(\mathcal{H}_{1})$\\
  Le Cam \cite{lecam} & arbitrary & generalized tests & $\co(\mathcal{H}_{0})$ vs. $\co(\mathcal{H}_{1})$\\
  Larsson et al. \cite{larsson2026completecharacterizationtestablehypotheses} & arbitrary & $0 \leq \phi \leq 1$ pointwise & $\overline{\co}^{*}(\mathcal{H}_{0})$ vs. $\overline{\co}^{*}(\mathcal{H}_{1})$ \\
  present work & pre-Hahn-loc. $\pp$ & $0\leq \phi \leq 1$ $\pp$-q.s. & $\co(\mathcal{H}_{0})$ vs. $\co(\mathcal{H}_{1})$\\
  \bottomrule
\end{tabular}
\end{center}
\section{Appendix: Some remarks on No Arbitrage in the robust binomial model and its Hahn extension}\label{binomNA} 
 For readers with a particular interest in mathematical finance we will comment on  no arbitrage in the robust binomial model as defined in Subsection~\ref{binom1} in the context of the Hahn-extension. We recall the robust no arbitrage condition NA($\mathcal{P}$) for a set of probability measures $\mathcal{P}$ of \cite{BN:15} as well as the classical no arbitrage condition NA($P$) for a fixed probability measure $P$. 
 We will give a short overview of the implications of the strictly supported alternative introduced in Definition~\ref{def_disj_supp}.  Observe that concerning arbitrage considerations the newly introduced set $\widetilde{\rrr}$ does not restrict the generality of these aspects, see Remark~\ref{NA_equal_r_tilde_r}. Recall  that in Subsection~\ref{binom1} the interest rate $r=0$ and that there is only one risky asset which  is $S=(S_0,S_1)$. By definition $S_0=1$ and $S_1$ is $\mathcal{F}$-measurable. As trading strategies are predictable, which means here $\mathcal{F}_0$-measurable, where $\mathcal{F}_0=\{\varnothing, \Omega\}$, we have that they are just constants $H=H_0\in\mathbb{R}$. Therefore everything only depends on the sign of $S_1-S_0$. Let us first give the definition of the classical NA condition for a fixed probability measure $P$.

\begin{mdef}\label{na_standard}
The market satisfies the condition NA$(P)$ if, for all $H\in\mathbb{R}$,
$H(S_1-S_{0})\geq 0\quad\text{${P}$-a.s.}\quad\text{implies}\quad H(S_1-S_{0})= 0\quad\text{${P}$-a.s.}$ This is obviously equivalent to the following: $P(S_1-S_0>0)>0$ if and only if $P(S_1-S_0<0)>0$.
\end{mdef}

Note that the classical binomial model with $r=0$, which can be described by a fixed measure $R\in\mathcal{R}$ such that $R\circ Y^{-1}=\pi\delta_{u}+(1-\pi)\delta_{d}$, satisfies the condition NA$(R)$ if and only if  $d<1<u$, see, e.g., \cite{shreve1}. 
The robust NA conditon is now a generalization to a set $\pp'$ replacing the specific $P$. By this replacement model uncertainty is introduced.

\begin{mdef}\label{na_qs}
For a set of probability measures $\pp'$ the one step robust binomial market  satisfies the condition NA$(\mathcal{P}')$ if, for all $H\in\mathbb{R}$,
$H(S_1-S_{0})\geq 0\quad\text{$\mathcal{P}'$-q.s.}\quad\text{implies}\quad H(S_1-S_{0})= 0\quad\text{$\mathcal{P}'$-q.s.}$ This is equivalent to the following: there exists $P\in\pp'$ with $P(S_1-S_0>0)>0$ if and only if there exists $P'\in\pp'$ such that  $P'(S_{1}-S_{0}<0)>0$.
\end{mdef}

Let us give the very easy proof of the claim in Definition~\ref{na_qs}. Indeed, assume that NA$(\mathcal{P}')$ holds and suppose that there would exist $P\in\pp'$ with $P(S_1-S_0>0)>0$ but $P'(S_1-S_0\geq 0)=1$ for all $P'\in\pp'$. Then this is an obvious arbitrage for any $H>0$, a contradicition. Assume on the other hand the condition on the set $\pp'$ is satisfied and there would be an arbitrage, i.e., an $H\in\mathbb{R}$ such that $P'(H(S_1-S_0)\geq 0)=1$, for all $P'\in\pp'$ and there would exist $P$ with $P(H(S_1-S_0)>0)>0$. If $H>0$ then $P(S_1-S_0>0)>0$, hence there exists $P'$ such that $P'(S_1-S_0<0)>0$, a contradiction. Analogously for $H<0$. 
\newline\newline
Note that under the assumptions on the parameters of the robust binomial model given in Subsection~\ref{binom1} the market satisfies NA($\mathcal{R}$) and therefore also NA($\pp$) and NA($\widetilde{\pp})$.  

In Lemma \ref{sigma_p_prop} below, we summarize these NA properties and the fact that the sets  $\mathcal{R}$ and $\widetilde{\pp}$ have the same polar sets.
  We give the straightforward proof of Lemma~\ref{sigma_p_prop}  for our case of one time period. Observe that the lemma is correct in the multiperiod case as well.

Indeed, if  the more involved definition of the corresponding set $\pp$ in \cite{blanch_carassus} is used,  Lemma~4.4 therein shows that NA holds for the corresponding more general sets $\mathcal{\rrr}$ and $\mathcal{\pp}$. It is easy to see that then NA also holds for  $\widetilde{\pp}=\sconv(\mathcal{P})$. Note that for $T=1$ the more general definition of $\mathcal{\pp}$ reduces exactly to our definition. 
The $\mathcal{R}$ is a supported alternative of $\mathcal{\pp}$ by Proposition~3.10 in \cite{liebrich-model-uncertainty:reverse-approach}, for the multiperiod case (hence, in particular it holds in the one period case).

\begin{lemma}\label{sigma_p_prop}
$\widetilde{\pp}$ is $\sigma$-convex, by definition, and has the same polar sets as $\mathcal{R}$. Thus $\widetilde{\pp}$ is of class (S) with supported alternative $\mathcal{R}$. The conditions $NA(\mathcal{R})$ and
$NA(\widetilde{\pp})$ are satisfied .    
\end{lemma}

\begin{proof}
That both sets of probability measures have the same polar sets is obvious as $\widetilde{\pp}=\co_{\sigma}(\rrr)$. 
That $\mathcal{\rrr}$ is supported follows from  \cite{liebrich-model-uncertainty:reverse-approach} as said above.
We will now show NA($\mathcal{R}$). Choose $R\in\mathcal{R}$ as follows: $R\circ Y^{-1}=\pi_0\delta_{U_0}+(1-\pi_0)\delta_{d_0}$. By assumption $0<\pi_0<1$, $U_0>1$ and $d_0<1$. Therefore $R\in\mathcal{R}$  and the market is a one period binomial model under $R$ with $d=d_0<1<U_0=u$, therefore $R(S_1-S_0>0)=R(S_1-S_0=U_0-1)=\pi_0>0$ and $R(S_1-S_0<0)=R(S_1-S_0=d_0-1)=1-\pi_0>0$. Therefore NA($\mathcal{P}'$) holds for $\mathcal{P}'=\mathcal{R}$ and, as $\mathcal{\rrr}\subset \mathcal{P}\subset\widetilde{\mathcal{P}}$ NA($\mathcal{P}'$) holds for $\mathcal{P}'=\mathcal{P}$ and $=\widetilde{\pp}$ as well. 
\end{proof}

In Subsection~\ref{binom1} we made the additional Assumption~\ref{addass_ud}. Observe that these additional parameter restrictions still allow for various combinations of no arbitrage and arbitrage for particular choices of measures $P$ in the sense of Definition~\ref{na_standard}. Indeed, the following Lemma~\ref{na_combinations} holds. Note that this lemma shows that the strong NA condition of  \cite{blanch_carassus} still does not hold, compare this to Lemma~4.4 of \cite{blanch_carassus}.

\begin{lemma}\label{na_combinations} Under Assumption~\ref{addass_ud} the condition NA$(\pp')$ holds for $\pp'=\rrr, \pp, \widetilde{\pp}$. However, the set $E_0$ can be chosen such that there exist measures $R\in\rrr$ for which NA$(R)$ is not satisfied.
\end{lemma}

\begin{proof} NA$(\pp')$ still holds for $\pp'=\rrr, \pp, \widetilde{\pp}$, see the proof of Lemma~\ref{sigma_p_prop}, the measure $R$ with $R\circ Y^{-1}=\pi_0\delta_{U_0}+(1-\pi_0)\delta_{d_0}$ is still in $\mathcal{R}$.

Now, choose $E_0$ such that, e.g., $0<d_0<u_0\leq D_0<1<U_0$, which is possible under Assumption~\ref{addass_ud}. Now choose any $d\in[d_0,D_0]$, $u\in[u_0,1]$, $\pi\in[\pi_0,\Pi_0]$. Obviously this choice is not unique. For the measure $R\in\rrr$ corresponding to $(u,d,\pi)$ we have that $0<d<1$ but $u\leq 1$. Hence there is an arbitrage in the classical binomial model given by $R$ with these parameters.
  
\end{proof}
\begin{remark}\label{NA_equal_r_tilde_r}
Recall the Hahn-extension from Subsection~\ref{binom1} where $\widetilde{\mathcal{F}} = \mathcal{H}_{\mathcal{F}}^{\widetilde{ \mathcal{R} }}$ and $\widetilde{ \pp }_{1}$ is the extension of $\widetilde{ \pp }$ to $\widetilde{\mathcal{F}}$ (as in Section \ref{sec:hahn_extension}). In Corollary~\ref{Hahnlocaliz_binom} we saw that
 $(\Omega, \widetilde{\mathcal{F}},\widetilde{ \pp }_{1})$ is Hahn-localizable, and $\LL(\widetilde{ \pp }_{1})$ satisfies the robust $L^{1}$-$L^{\infty}$ duality by Theorem \ref{thm:duality}. Let us now shortly remark that the Hahn extension does not change the robust NA condition: 
    observe that $\mathrm{NA}(\widetilde{\pp})$ holds if and only if $\mathrm{NA}(\widetilde{ \pp }_{1})$  holds.
        To see this, note that the events in Definition \ref{na_qs}  are already in $\mathcal{F}$, and hence the extended probability measures on $\widetilde{\mathcal{F}}$ agree with the original probability measures on $\mathcal{F}$. 
\end{remark}

\bibliographystyle{plain}
\bibliography{main}
\end{document}